\documentclass[12pt,reqno]{amsart}
\usepackage[margin=1in]{geometry}
\usepackage{amsmath,amssymb,amsthm,graphicx,amsxtra, setspace}
\usepackage[utf8]{inputenc}
\usepackage{mathrsfs}
\usepackage{hyperref}
\usepackage{upgreek}
\usepackage{mathtools}
\usepackage{xcolor}

\usepackage{multirow}
\usepackage{booktabs}
\usepackage{caption}
\usepackage{subcaption}

\usepackage[mathcal]{euscript}
\allowdisplaybreaks

\usepackage[pagewise]{lineno}

\DeclareMathAlphabet{\mathpzc}{OT1}{pzc}{m}{it}

\usepackage[cyr]{aeguill}

\colorlet{darkblue}{blue!50!black}

\hypersetup{
	colorlinks,%
	citecolor=blue,%
	filecolor=red,%
	linkcolor=red,%
	urlcolor=blue,%
	pdfnewwindow=true,%
	pdfstartview={FitH}
}

\newtheorem{theorem}{Theorem}[section]
\newtheorem{lemma}[theorem]{Lemma}
\newtheorem{proposition}[theorem]{Proposition}

\newtheorem{definition}[theorem]{Definition}
\newtheorem{problem}[theorem]{Problem}

\newtheorem{remark}[theorem]{Remark}
\newtheorem{algorithm}[theorem]{Algorithm}

\newtheorem{hypothesis}[theorem]{Hypothesis}

\allowdisplaybreaks

\let\originalleft\left
\let\originalright\right
\renewcommand{\left}{\mathopen{}\mathclose\bgroup\originalleft}
\renewcommand{\right}{\aftergroup\egroup\originalright}

\renewcommand{\d}{\/\mathrm{d}\/}

\def\L{\mathbb{L}}
\def\A{\mathcal{A}}

\def\I{\mathrm{I}}
\def\F{\mathcal{F}}
\def\C{\mathcal{C}}
\def\f{\boldsymbol{f}}

\def\B{\mathcal{B}}
\def\D{\mathrm{D}}
\def\y{\boldsymbol{y}}

\def\X{\mathbb{X}}
\def\x{\boldsymbol{x}}

\def\g{\boldsymbol{g}}

\def\z{\boldsymbol{z}}
\def\bv{\boldsymbol{v}}
\def\V{\mathcal{v}}
\def\bw{\boldsymbol{w}}

\def\N{\mathbb{N}}

\def\V{\mathcal{V}}
\def\wi{\widetilde}

\def\bu{\boldsymbol{u}}
\def\H{\mathbb{H}}
\def\n{\boldsymbol{n}}

\newcommand{\eps}{\varepsilon}

\newcommand{\R}{\mathbb{R}}

\renewcommand{\d}{\/\mathrm{d}\/}

\newcommand{\Addresses}{{
		\footnote{
			
			\noindent \textsuperscript{1,2}Department of Mathematics, Indian Institute of Technology Roorkee-IIT Roorkee,
			Haridwar Highway, Roorkee, Uttarakhand 247667, INDIA.\par\nopagebreak
			\noindent  \textit{e-mail:} \texttt{Manil T. Mohan: maniltmohan@ma.iitr.ac.in, maniltmohan@gmail.com.}
			
			\textit{e-mail:} \texttt{Wasim Akram: wakram2k11@gmail.com.}
			
			\noindent \textsuperscript{*}Corresponding author.
			
			\textit{Key words:} Stationary convective Brinkman-Forchheimer Extended Darcy equations, hemivariational inequality, pseudomonotonicity, mixed finite element method, error estimates. 
			
			Mathematics Subject Classification (2020): 65N30, 35Q35, 47J20, 65J15, 65N15.

}}}

\begin{document}
	
	\title[Hemivariational inequality for 2D and 3D CBFeD equations]{Mixed Finite Element Method for a Hemivariational Inequality of Stationary convective Brinkman-Forchheimer Extended Darcy equations
		\Addresses}
	\author[W. Akram and M. T. Mohan]
	{Wasim Akram\textsuperscript{2} and Manil T. Mohan\textsuperscript{1*}}
	
	\maketitle

	\begin{abstract}
	This paper presents the formulation and analysis of a mixed finite element method for a hemivariational inequality arising from the stationary convective Brinkman-Forchheimer extended Darcy (CBFeD) equations. This model extends the incompressible Navier-Stokes equations by incorporating both damping and pumping effects. The hemivariational inequality describes the flow of a viscous, incompressible fluid through a saturated porous medium, subject to a nonsmooth, nonconvex friction-type slip boundary condition. The incompressibility constraint is handled via a mixed variational formulation. We establish the existence and uniqueness of solutions by utilizing the pseudomonotonicity and coercivity properties of the underlying operators and provide a detailed error analysis of the proposed numerical scheme. Under suitable regularity assumptions, the method achieves optimal convergence rates with low-order mixed finite element pairs. The scheme is implemented using the 
	$\text{P1b/P1}$ element pair, and numerical experiments are presented to validate the theoretical results and confirm the expected convergence behavior.
	\end{abstract}

	\section{Introduction}\label{sec1}\setcounter{equation}{0}
	\subsection{The model}\label{sub-sec-1}	
	Let $\mathcal{O} \subset \mathbb{R}^d$ ($d = 2,3$) be a bounded, connected, open domain with Lipschitz continuous boundary $\Gamma$. A generic point in $\mathcal{O}$ or on $\Gamma$ is denoted by $\x$. Let $\mathbb{S}^d$ denote the space of symmetric $d \times d$ matrices. On both $\mathbb{R}^d$ and $\mathbb{S}^d$, we use the following standard inner products:
		\begin{align}
		\bu\cdot\bv&=u_iv_i,\ \bu,\bv\in\R^d,\\
		\boldsymbol{\sigma}\cdot\boldsymbol{\tau}&=\sigma_{ij}\tau_{ij},\ \boldsymbol{\sigma},\boldsymbol{\tau} \in\mathbb{S}^d,
	\end{align}
	where we have used the	Einstein summation convention.

We consider the following two- and three-dimensional \emph{convective Brinkman–Forchheimer extended Darcy (CBFeD)} equations, which model the steady flow of an incompressible, viscous fluid through a porous medium:
\begin{equation} \label{eqn-CBF}
	\left\{
\begin{aligned}
	-\mu\Delta\bu + (\bu \cdot \nabla)\bu + \alpha\bu + \beta|\bu|^{r-1}\bu + \kappa|\bu|^{q-1}\bu + \nabla p &= \f, &&\text{in } \mathcal{O}, \\
	\mathrm{div}\,\bu &= 0, &&\text{in } \mathcal{O}.
\end{aligned}
\right. 
\end{equation}
In this system, $\bu : \mathcal{O} \to \mathbb{R}^d$ denotes the velocity field, $p : \mathcal{O} \to \mathbb{R}$ the pressure, and $\f : \mathcal{O} \to \mathbb{R}^d$ the external force. The coefficient $\mu > 0$ represents the Brinkman term, modeling the effective viscosity. The positive constants $\alpha$ and $\beta$ correspond to Darcy and Forchheimer damping effects, accounting for resistance due to permeability and porosity, respectively.
An additional nonlinear term $\kappa|\bu|^{q-1}\bu$ is included to model a potential pumping mechanism, which counteracts the damping when $\kappa < 0$, as assumed throughout this work. The exponent $r \in [1, \infty)$, known as the \emph{absorption exponent}, characterizes the strength of the Forchheimer damping, with $r = 3$ identified as the \emph{critical exponent} (\cite[Proposition 1.1]{KWH}). The parameter $q \in [1, r)$ governs the strength of the pumping effect. The system reduces to the classical Navier-Stokes equations when $\alpha = \beta = \kappa = 0$, and becomes a damped Navier-Stokes system when $\alpha, \beta > 0$ and $\kappa = 0$. The CBFeD model arises from an extended Darcy-Forchheimer law, as detailed in \cite{SGKKMTM, MTT}. The divergence-free condition in the second equation enforces the incompressibility of the fluid flow.

	The CBFeD system given in \eqref{eqn-CBF} is supplemented by appropriate boundary conditions to complete the problem formulation. We assume that the boundary $\Gamma$ of the domain $\mathcal{O}$ is partitioned into two disjoint measurable subsets, $\Gamma_0$ and $\Gamma_1$, each having strictly positive surface measure.	Let $\boldsymbol{n} = (n_1, \dots, n_d)$ denote the outward unit normal vector on $\Gamma$. For any vector field $\boldsymbol{u}$ defined on the boundary, we define its \emph{normal component} as $u_n = \boldsymbol{u} \cdot \boldsymbol{n}$, and its \emph{tangential component} as $\boldsymbol{u}_\tau = \boldsymbol{u} - u_n \boldsymbol{n}$. Given the flow velocity $\boldsymbol{u}$ and pressure $p$, the \emph{rate of deformation tensor} is defined by
	$$
	\boldsymbol{\varepsilon}(\boldsymbol{u}) = \frac{1}{2}(\nabla \boldsymbol{u} +( \nabla \boldsymbol{u})^\top),
	$$
which characterizes the symmetric part of the velocity gradient and describes the strain rate in the fluid,	and the corresponding Cauchy stress tensor is given by
	$$
	\boldsymbol{\sigma} = 2\mu \boldsymbol{\varepsilon}(\boldsymbol{u}) - p \I,
	$$
	where $\I$ denotes the identity matrix. On the boundary $\Gamma$, we define the normal and tangential components of the stress vector $\boldsymbol{\sigma} \boldsymbol{n}$ as:
	$$
	\sigma_n = \boldsymbol{n} \cdot \boldsymbol{\sigma} \boldsymbol{n}, \quad \boldsymbol{\sigma}_\tau = \boldsymbol{\sigma} \boldsymbol{n} - \sigma_n \boldsymbol{n}.
	$$
	The identities
		$$
	\boldsymbol{u} \cdot \boldsymbol{v} = u_n v_n + \boldsymbol{u}_\tau \cdot \boldsymbol{v}_\tau \quad \text{and} \quad (\boldsymbol{\sigma} \boldsymbol{n}) \cdot \boldsymbol{v} = \sigma_n v_n + \boldsymbol{\sigma}_\tau \cdot \boldsymbol{v}_\tau
	$$
		are particularly useful in the derivation of the associated hemivariational inequality.
		The system is supplemented with the following boundary conditions:
	\begin{align}
		\boldsymbol{u} &= \boldsymbol{0} \quad \text{on } \Gamma\_0, \label{eqn-boundary-1} \\
		u_n &= 0, \quad -\boldsymbol{\sigma}*\tau \in \partial j(\boldsymbol{u}*\tau) \quad \text{on } \Gamma_1, \label{eqn-boundary-2}
	\end{align}
	where $j(\boldsymbol{u}_\tau)$ denotes $j(\x, \boldsymbol{u}_\tau)$, and the function $j : \Gamma_1 \times \mathbb{R}^d \to \mathbb{R}$ is known as the \emph{superpotential}. We assume that $j$ is locally Lipschitz continuous in its second argument. The symbol $\partial j$ refers to the \emph{Clarke subdifferential} of the function $j(\x, \cdot)$, interpreted in the sense of Clarke’s generalized gradient.	The condition \eqref{eqn-boundary-2} models a generalized slip boundary condition. Importantly, if $j(\x, \cdot)$ is convex in its second argument, the system \eqref{eqn-CBF}–\eqref{eqn-boundary-2} reduces to a classical \emph{variational inequality}. However, in this work, we do not impose any convexity assumption on $j$, placing the problem within the more general setting of \emph{hemivariational inequalities}, which are well-suited for capturing nonsmooth and nonconvex energy interactions, particularly those arising from nonmonotone friction-type boundary conditions.

	\subsection{Literature review}
	Boundary value problems involving viscous incompressible fluid flows with nonsmooth slip or leak boundary conditions of friction type have garnered considerable attention since the foundational work of Fujita in the early 1990s \cite{HFu, HFu1}. This interest is reflected in an extensive body of literature addressing both the mathematical modeling and analytical aspects of such problems \cite{HFuHK, CLeR, CLeRAT, FSa}, as well as the development of numerical methods for their solution \cite{JKDJK, FJWHWY, FJJLZC, TKa, TKa1, YLRA, YLRA1, YLKL}.	In many of these works, the slip or leak boundary conditions are described by monotone relationships between the relevant physical quantities. This monotonicity enables the formulation of the problem as a variational inequality, often governed by the Stokes equations (\cite{JHRKVS, FJWHYZ, JLFJZC, NSa, FWMLWH}) or the Navier-Stokes equations (\cite{JKD, YLRA2, HQRALM}). However, in cases where the boundary behavior exhibits non-monotonicity, typically due to more complex frictional effects, the resulting weak formulation leads to a broader class of problems known as \emph{hemivariational inequalities} \cite{CFKCWH, CFWH, SMAO, SMi}. These arise from nonsmooth and nonconvex energy functionals, often referred to as \emph{superpotentials}. Within the theoretical framework of hemivariational inequalities, the existence of solutions is commonly established through abstract surjectivity results for \emph{pseudomonotone operators}, as developed in the context of nonlinear functional analysis (see \cite[Section 32.4]{EZ}, \cite[Theorem 2.6]{TRo}).

	The theory of \emph{hemivariational inequalities} was pioneered by Panagiotopoulos in the early 1980s \cite{PDP}, building on the notion of generalized directional derivatives and subdifferentials of locally Lipschitz functionals in the sense of Clarke \cite{FHC, FHC1}. This theoretical framework has since become a powerful tool for modeling nonmonotone, nonsmooth, and set-valued phenomena across various scientific and engineering domains, including contact mechanics and fluid dynamics. Over the past four decades, hemivariational inequalities have proven to be remarkably versatile and effective in capturing complex behaviors in numerous applications. Consequently, the body of mathematical literature devoted to this area has grown substantially and continues to expand rapidly. For comprehensive treatments of the subject, we refer the reader to \cite{SCVKL, JHMMPD, SMAOMS, ZNPDP, PDP1, MSSM}.

The theory, numerical analysis, and applications of hemivariational inequalities have been extensively explored and are well-documented in several authoritative monographs \cite{SCVKL, SMAOMS, DMPDP, ZNPDP, PDP1, MSSM}, along with an expanding corpus of journal publications. A detailed study of the finite element method (FEM) for hemivariational inequalities, including convergence analysis and solution algorithms, is presented in the monograph \cite{JHMMPD}. In recent years, significant progress has been made in deriving optimal-order error estimates for FEM applied to various classes of hemivariational inequalities. These include elliptic problems \cite{WHSM1, WHMS1, WHMS2}, evolutionary formulations \cite{MBKB, WHCW}, and history-dependent models \cite{WXZH}; see also the survey article \cite{WHMS} for a broader overview. Beyond FEM, alternative numerical approaches have been developed. Notably, the virtual element method (VEM) has been successfully applied to hemivariational inequalities arising in contact mechanics \cite{FFWH}. A nonconforming VEM was proposed and analyzed in \cite{MLFW} for solving stationary Stokes hemivariational inequalities subject to slip boundary conditions. These developments highlight the robustness and adaptability of the hemivariational inequality framework in addressing a wide spectrum of challenging nonlinear problems.

The Navier-Stokes equations, fundamental to the mathematical modeling of fluid dynamics, have been extensively studied across theory, numerical methods, and practical applications \cite{VGPR, Te, Te1}. Given the lack of explicit analytical solutions for hemivariational inequalities, numerical techniques play a crucial role in approximating their solutions. An early contribution in this direction is the development of a mixed finite element method for the Stokes hemivariational inequality with slip boundary conditions, as presented in \cite{CFKCWH}, which provided a foundation for further exploration of more complex flow models. Building on this, the study \cite{WHKCFJ} extended the analysis to the stationary Navier-Stokes hemivariational inequality (NS-HVI), offering a more general and physically realistic framework. Further theoretical investigations into stationary NS-HVIs were conducted in \cite{JMSS}, while associated optimal control problems were addressed in \cite{SMi1}. A range of theoretical and numerical results for variational-hemivariational inequalities arising from different fluid dynamic models have also been established in recent works, including \cite{JCVTN, WHFJYY, FJWHGZ, MLWH1, MLWHSZ, XTTC, QXXCKL}.

Damping effects are essential for accurately modeling fluid flows subject to resistance. The Stokes and Navier-Stokes equations augmented with damping terms are widely employed to represent various physical mechanisms such as drag forces, viscous dissipation, and energy loss in fluid systems (\cite{KWH, CFKCWH, SGMTM, SGKKMTM, MTT}). These extended models provide a more realistic description of flows through porous or resistive media. Analytical and numerical investigations of Stokes variational inequalities incorporating damping have been carried out in works such as \cite{HQLM}. A variational formulation for the Stokes hemivariational inequality with damping was proposed in \cite{WHHQLM} (see also \cite{MLWH} for a general treatment of the Stokes hemivariational inequality), where mixed finite element methods were developed along with corresponding error estimates. To the best of our knowledge, the theoretical and numerical treatment of Navier-Stokes hemivariational inequalities for viscous incompressible fluids subject to damping has not yet been explored in the existing literature. This work aims to bridge that gap by formulating and analyzing a class of Navier-Stokes hemivariational inequalities that incorporate both damping and pumping effects, thereby advancing the current understanding of such nonlinear fluid models.

	\subsection{Our contribution}
This paper develops and analyzes a mixed finite element method for solving a hemivariational inequality arising from the stationary convective Brinkman-Forchheimer extended Darcy (CBFeD) equations. This system models complex fluid flow in saturated porous media by incorporating physical mechanisms beyond those captured by the classical Navier-Stokes framework. In particular, the CBFeD model extends the incompressible Navier-Stokes equations through the inclusion of:
\begin{itemize}
	\item a \emph{damping term} representing inertial resistance in porous media, and
	\item a \emph{pumping term} that enhances momentum transport, akin to viscous effects.
\end{itemize}
These extensions render the CBFeD model particularly suitable for applications such as groundwater movement, filtration processes, and biological fluid transport.

The boundary of the domain is subject to \emph{slip-type conditions} governed by a nonsmooth and potentially nonconvex friction law, which allows fluid to slip along the boundary rather than adhere to it. This naturally leads to a hemivariational inequality formulation, where the boundary behavior is captured by a nonsmooth, nonconvex energy functional.

To accommodate the incompressibility constraint (i.e., divergence-free velocity field), the problem is posed in a \emph{mixed variational framework}, where both velocity and pressure are simultaneously approximated within appropriate function spaces. The mixed finite element method is particularly well-suited to such saddle-point problems. Leveraging the \emph{inf-sup condition}, a rigorous mathematical analysis of the proposed scheme is performed. The existence and uniqueness of weak solutions rely crucially on the \emph{pseudomonotonicity} and \emph{coercivity} of the associated nonlinear operators, along with an abstract \emph{surjectivity theorem} from \cite[Theorem 1.3.70]{ZdSmP}. We emphasize the following scope of the analysis:
\begin{itemize}
	\item Problem \ref{prob-hemi-2} is addressed for spatial dimensions $d \in \{2,3\}$ and for the full range $r \in [1, \infty)$;
	\item Problem \ref{prob-hemi-1} is solved for $d = 2$ and $r \in [1, \infty)$, and for $d = 3$ with $r \in [1,5]$, the latter due to the reliance on the inf-sup condition in the existence proof.
\end{itemize}
Assuming standard regularity of the exact solution, \emph{optimal-order error estimates} are derived for the proposed numerical method. The implementation utilizes the classical $\text{P1b/P1}$ finite element pair, where the velocity field is approximated using enriched piecewise linear functions (with bubble functions), and the pressure is approximated using standard piecewise linear elements. In summary, this work advances the theoretical and computational treatment of incompressible fluid flows governed by hemivariational inequalities. It extends previous analyses from the Navier-Stokes regime to the CBFeD setting and provides new insights into the numerical analysis of nonsmooth, nonmonotone boundary conditions in viscous flow models.

	\subsection{Structure of the paper} 
The remainder of the paper is structured as follows. In the next section, we introduce the necessary preliminaries and establish the functional framework relevant to the model described in Subsection \ref{sub-sec-1}. We also formulate the hemivariational inequality associated with the system defined by equations \eqref{eqn-CBF}-\eqref{eqn-boundary-2}, corresponding to Problems \ref{prob-hemi-1} and \ref{prob-hemi-2}.
		In Section \ref{sec3}, we analyze Problem \ref{prob-hemi-2}, proving the existence of weak solutions for spatial dimensions $d \in \{2,3\}$ and parameter range $r \in [1, \infty)$ (see Theorem \ref{thm-main-hemi}). This existence result is derived using the coercivity and pseudomonotonicity of the underlying operators, together with an abstract surjectivity theorem from \cite[Theorem 1.3.70]{ZdSmP}. Furthermore, Proposition \ref{prop-energy-est} establishes the boundedness of all solutions to Problem \ref{prob-hemi-2}, and Theorem \ref{thm-unique} addresses the uniqueness of weak solutions. To handle Problem \ref{prob-hemi-1}, we apply the inf-sup condition (also known as the Ladyzhenskaya-Babuška-Brezzi or LBB condition) and prove existence results in Theorem \ref{thm-equivalent} for $d = 2$ with $r \in [1, \infty)$ and $d = 3$ with $r \in [1,5]$.	Section \ref{sec4} is devoted to the numerical analysis of the CBFeD hemivariational inequality using a mixed finite element method. We first derive a version of Céa’s lemma to obtain error bounds for the finite element approximation (Theorem \ref{thm-cea}), followed by optimal-order error estimates for the $\text{P1b/P1}$ finite element pair under suitable regularity assumptions (Theorem \ref{thm-optimal}). 	Finally, the last section presents an iterative algorithm for numerically solving Problem \ref{prob-hemi-1}, along with computational results that illustrate the accuracy and effectiveness of the proposed method.

	\section{Mathematical Formulation}\label{sec2}\setcounter{equation}{0}
The primary aim of this section is to present the key mathematical preliminaries required for the theoretical analysis developed in this work. In addition, we establish the functional framework corresponding to the model introduced in Section \ref{sec1}. Finally, we formulate the hemivariational inequality associated with the system defined by equations \eqref{eqn-CBF}-\eqref{eqn-boundary-2}.

	\subsection{Preliminaries}
	All function spaces considered in this work are defined over the field of real numbers. Let $\mathbb{X}$ be a normed space with norm denoted by $\|\cdot\|_{\mathbb{X}}$. Its topological dual is denoted by $\mathbb{X}^*$, and the duality pairing between $\mathbb{X}^*$ and $\mathbb{X}$ is written as ${}_{\mathbb{X}^*}\langle \cdot, \cdot \rangle_{\mathbb{X}}$. The notation $\mathbb{X}_w$ refers to the space $\mathbb{X}$ endowed with the weak topology. We use $2^{\mathbb{X}^*}$ to denote the collection of all subsets of $\mathbb{X}^*$. Unless otherwise stated, we assume throughout that $\mathbb{X}$ is a Banach space.

	We start with the definition of a locally Lipschitz function.
	\begin{definition}
	A function $\psi: \mathbb{X} \to \mathbb{R}$ is said to be \emph{locally Lipschitz} if, for every point $\x \in \mathbb{X}$, there exists a neighborhood $U \subset \mathbb{X}$ of $\x$ and a constant $L_U > 0$ such that
		$$
	|\psi(\y) - \psi(\z)| \leq L_U \|\y - \z\|_{\mathbb{X}} \  \text{ for all }\  \y, \z \in U.
	$$

	\end{definition}
	
	We now recall the definitions of the \emph{generalized directional derivative} and the \emph{generalized gradient} (in the sense of Clarke) for a locally Lipschitz function.
	
	\begin{definition}[{\cite[Definition  5.6.3]{ZdSm1}}]
		Let $\psi : \X \to \R$ be a locally Lipschitz function. The \emph{generalized directional derivative} of $\psi$ at a point $\x \in \X$ in the direction $\bv \in \X$, denoted by $\psi^0(\x;\bv)$, is defined as
		$$
		\psi^0(\x;\bv) = \lim_{\y \to \x} \sup_{\lambda \downarrow 0} \frac{\psi(\y + \lambda \bv) - \psi(\y)}{\lambda}.
		$$
		The \emph{generalized gradient} (or \emph{Clarke subdifferential}) of $\psi$ at $\x$, denoted by $\partial \psi(\x)$, is the subset of the dual space $\X^{*}$ given by
			$$
		\partial \psi(\x) = \left\{ \zeta \in \X^{*} : \psi^0(\x;\bv) \geq {}_{\X^{*}}\langle \zeta, \bv \rangle_{\X} \ \text{ for all } \bv \in \X \right\}.
		$$
			A locally Lipschitz function $\psi$ is said to be \emph{regular (in the sense of Clarke)} at a point $\x \in \X$ if the one-sided directional derivative $\psi^{\prime}(\x;\bv)$ exists for all $\bv \in \X$, and satisfies $\psi^0(\x;\bv) = \psi^{\prime}(\x;\bv)$.
			\end{definition}

	The following results will be utilized in the subsequent analysis.

	\begin{proposition}[{\cite[Proposition 2.1.2]{FHC}, \cite[Proposition 5.6.9]{ZdSm1}}]\label{prop-clarke}
		If  $\psi:\X\to\R$ is locally Lipschitz, then 
		\begin{enumerate}
			\item $\partial \psi(\x)$ is a nonempty, convex, weak$^*$-compact subset of $\X^{*}$ and $\|\boldsymbol{\zeta}\|_{\X^{*}}\leq L_U$ for every $\boldsymbol{\zeta}\in\partial \psi(\x)$;
			\item for every $\bv,\x\in\X$, we have 
			$$	\psi^0(\x;\bv)=\max\left\{\langle\boldsymbol{\zeta},\bv\rangle:\boldsymbol{\zeta}\in\partial \psi(\x)\right\}.$$
		\end{enumerate}
	\end{proposition}
	Expression (2) provides a characterization of the generalized directional derivative in terms of the Clarke subdifferential $\partial \psi(\x)$. We now present an important property, the sub-additivity of the generalized directional derivative—which will be employed in the subsequent analysis:
	\begin{align}\label{eqn-subaddtive}
\psi^0(\x; \bv_1 + \bv_2) \leq \psi^0(\x; \bv_1) + \psi^0(\x; \bv_2) \ \text{ for all } \ \x, \bv_1, \bv_2 \in \X.
\end{align}
		Additional properties and in-depth discussions of the generalized directional derivative and the Clarke subdifferential can be found in \cite{FHC,SMAOMS}.

	\begin{proposition}[{\cite[Proposition 5.6.10]{ZdSm1}}]\label{prop-ups}
	If $\psi: \X \to \mathbb{R}$ is a locally Lipschitz function, then the set-valued mapping $\x \mapsto \partial \psi(\x)$ is upper semicontinuous from $\X$ into $\X^*$.
	\end{proposition}
	We subsequently review the notion of pseudomonotonicity in the context of single-valued operators.
	\begin{definition}[]\label{def-pseudo}
		A single-valued operator $\mathcal{G} : \X\to\X^{*}$  is said to be \emph{pseudomonotone}, if
		\begin{enumerate}
			\item $\mathcal{G}$ is bounded (that is, it maps bounded subsets of $\X$ into bounded subsets of $\X^*$);
			\item $\bu_n\rightharpoonup \bu$ in $\X$ and $\limsup\limits_{n\to\infty} {}_{\X^{*}}\langle\mathcal{G}(\bu_n),\bu_n-\bu\rangle_{\X}\leq 0$ imply 
			\begin{align*}
				{}_{\X^{*}}\langle\mathcal{G}(\bu),\bu-\bv\rangle_{\X}\leq\liminf_{n\to\infty}{}_{\X^{*}}\langle\mathcal{G}(\bu_n),\bu_n-\bv\rangle_{\X}\ \text{ for all }\ \bv\in\X.
			\end{align*}
		\end{enumerate}
	\end{definition}
	It is proved in \cite[Remark 2]{SMAO} that an operator $\mathcal{G} : \X\to\X^{*}$ is pseudomonotone if and only if it is bounded and $\bu_n\rightharpoonup \bu$ in $\X$ together with $\limsup\limits_{n\to\infty} {}_{\X^{*}}\langle\mathcal{G}(\bu_n),\bu_n-\bu\rangle_{\X}\leq 0$ imply 
	\begin{align}\label{eqn-con-pseudo}
		\mathcal{G}(\bu_n)\rightharpoonup \mathcal{G}(\bu)\ \text{ in }\ \X^{*}\ \text{ and }\ \lim_{n\to\infty}{}_{\X^{*}}\langle\mathcal{G}(\bu_n),\bu_n-\bu\rangle_{\X}=0. 
	\end{align}
	The following definition appears, for instance, in  \cite[Definition 1]{FBPH} or \cite[Definition 3.57]{SMAOMS}. 
	\begin{definition}
		Let $\X$ be a reflexive Banach space. A multi-valued operator $\mathcal{G} : \X\to 2^{\X^{*}}$ is \emph{pseudomonotone} if the following conditions hold:
		\begin{enumerate}
			\item $\mathcal{G}$ has values which are nonempty, bounded, closed and convex;
			\item $\mathcal{G}$  is upper semicontinuous from each finite dimensional subspace of $\X$ into $\X^{*}_w$;
			\item For sequences $\{\bu_n\}\subset\X$ and $\{\bu_n^*\}\subset\X^{*}$ such that $\bu_n\rightharpoonup \bu$ in $\X$, $\bu_n^*\in\mathcal{G}(\bu_n)$ and $\limsup\limits_{n\to\infty}{}_{\X^{*}}\langle\bu_n^*,\bu_n-\bu\rangle_{\X}\leq 0$, we have that for every $\bv\in\X$, there exists $\bu^*(\bv) \in\mathcal{G}(\bu)$  such that
			\begin{align*}
				{}_{\X^{*}}\langle\bu^*(\bv),\bu-\bv\rangle_{\X}\leq\liminf_{n\to\infty}{}_{\X^{*}}\langle\bu_n^*(\bv),\bu_n-\bv\rangle_{\X}.
			\end{align*}
		\end{enumerate}
	\end{definition}
	The following well-known result provides a sufficient condition for establishing the pseudomonotonicity of an operator.
	\begin{proposition}[{\cite[Proposition 1.3.66]{ZdSmP}}]\label{prop-suff-pseudo}
		Let $\X$ be a real reflexive Banach space, and assume that $\mathcal{G} : \X \to 2^{\X^{*}}$  satisfies the following conditions:
		\begin{enumerate}
			\item for each $v \in \X$, $\mathcal{G}(\bv)$  is a nonempty, closed, and convex subset of $\X^*$;
			\item $\mathcal{G}$  is bounded;
			\item if $\bv_n\rightharpoonup \bv$ in $\X$, $\bv_n^*\rightharpoonup \bv^*$ in $\X^{*}$ with $\bv_n^*\in\mathcal{G}(\bv_n)$, and $\limsup\limits_{n\to\infty}{}_{\X^{*}}\langle\bv_n^*,\bv_n-\bv\rangle_{\X}\leq 0$, then $\bv^*\in\mathcal{G}(\bv)$ and ${}_{\X^{*}}\langle\bv_n^*,\bv_n\rangle_{\X}\to {}_{\X^{*}}\langle\bv^*,\bv\rangle_{\X}$.
		\end{enumerate}
		Then the operator $\mathcal{G}$ is pseudomonotone.
	\end{proposition}

	\begin{proposition}[{\cite[Proposition 1.3.68]{ZdSmP}}]
		If $\X$ is a reflexive Banach space and $\mathcal{G}_1,\mathcal{G}_2:\X\to 2^{\X^{*}}$ are pseudomonotone operators, then $\mathcal{G}_1+\mathcal{G}_2$ is also pseudomonotone. 
	\end{proposition}

	We now define the concept of coercivity, which will be utilized in the subsequent analysis.
	\begin{definition}
		An operator $\mathcal{G} : \X \to 2^{\X^{*}}$ is \emph{coercive} if either $\D(\mathcal{G})$, where $\D(\mathcal{G})$ is the domain of $\mathcal{G}$,   is bounded or $\D(\mathcal{G})$ is unbounded and
		\begin{align*}
			\lim\limits_{\|\bu\|_{\X}\to\infty,\ \bu\in\D(\mathcal{G})}\frac{\inf\left\{ {}_{\X^{*}}\langle\bu^*,\bu\rangle_{\X} :\bu^*\in\mathcal{G}(\bu)\right\}}{\|\bu\|_{\X}}=\infty. 
		\end{align*}
	\end{definition}
	We now state the central surjectivity theorem applicable to operators that are both pseudomonotone and coercive (see \cite[Section 32.4]{EZ}, \cite[Theorem 2.6]{TRo}).
	\begin{theorem}[{\cite[Theorem 1.3.70]{ZdSmP}}]\label{thm-surjective}
		Let $\X$ be a reflexive Banach space and $\mathcal{G}: \X \to 2^{\X^{*}}$ be pseudomonotone and coercive. Then $\mathcal{G}$ is surjective, that is, $R(\mathcal{G}) = \X^{*}$, where $R(\mathcal{G})$ is the range of the operator $\mathcal{G}$.
	\end{theorem}

	The following  result is needed to establish the covergence of the nonlinear term. 
	\begin{lemma}[Brezis-Lions Lemma, {\cite[Lemma 1.3]{JLL}}]\label{Lem-Lions}
		Let $\mathcal{O}$ be a bounded open set of $\R^n$, $\varphi_m$ and $\varphi$ be functions in $\mathrm{L}^q(\mathcal{O})$, for $m\in\N$ and $1<q <\infty$, such that
		\begin{align*}
			\|\varphi_m\|_{L^q(\mathcal{O})} \leq C,\ \ \mbox{for every}\ \ m\in\N \; \; \mbox{and} \ \ \varphi_m \to \varphi\;\; \mbox{a.e. in}\ \ \mathcal{O},\  \mbox{as}\ \ m \to \infty.
		\end{align*}
		Then, $\varphi_m \rightharpoonup  \varphi$ in $\mathrm{L}^q(\mathcal{O})$, as $m\to \infty$.
	\end{lemma}

	\subsection{Functional setting} 
	To express the problem \eqref{eqn-CBF}-\eqref{eqn-boundary-2} in a weak (variational) formulation, we begin by defining suitable function spaces. The velocity field is assumed to belong to the following space:
\begin{equation*}
	\mathcal{V}:=\left\{\bv\in\mathrm{H}^1(\mathcal{O};\mathbb{R}^d):\bv=\boldsymbol{0}\ \text{ on }\ \Gamma_0, \ v_n=0\ \text{ on }\ \Gamma_1\right\}.
\end{equation*}
Using Korn’s inequality and the assumption on $|\Gamma_0|>0$ imply (see \cite[Lemma 6.2]{NKJTO}) \begin{align}\label{eqn-equiv-1}\|\bv\|_{\mathrm{H}^1(\mathcal{O};\mathbb{R}^d)}\leq C_k\|\boldsymbol{\varepsilon}(\boldsymbol{v})\|_{\mathrm{L}^2(\mathcal{O};\mathbb{S}^d)},\end{align}  one can verify  that $\mathcal{V}$ is a Hilbert space with the inner product 
\begin{align*}
	(\bu,\bv)_{\mathcal{V}}=(\boldsymbol{\varepsilon}(\boldsymbol{u}),\boldsymbol{\varepsilon}(\boldsymbol{v}))_{\mathrm{L}^2(\mathcal{O};\mathbb{S}^d)}. 
\end{align*}
The norm in $\mathcal{V}$ is denoted by 
\begin{align*}
	\|\bu\|_{\mathcal{V}}=\|\boldsymbol{\varepsilon}(\boldsymbol{u})\|_{\mathrm{L}^2(\mathcal{O};\mathbb{S}^d)}=\left(\sum_{i,j=1}^d\int_{\mathcal{O}}|\varepsilon_{ij}(\bv)|^2\d\x\right)^{1/2},
\end{align*}
which is equivalent to the $\mathbb{H}^1(\mathcal{O})$ norm. In the formulation of the reduced problem, we will consider the following subspace of $\mathcal{V}$: 
\begin{align*}
	\V_{\sigma}:=\left\{\bv\in\mathcal{V}:\mathrm{div \ }\bv=0\ \text{ in }\ \mathcal{O}\right\}. 
\end{align*}
Let us define $\mathcal{H}:=\mathrm{L}^2(\mathcal{O};\mathbb{R}^d)$. Then the continuous following embeddings:
\begin{align*}
	\mathcal{V}\hookrightarrow\mathcal{H}\equiv\mathcal{H}^*\hookrightarrow\mathcal{V}^*
\end{align*}
and 
\begin{align*}
	\V_{\sigma}\hookrightarrow\mathcal{H}\equiv\mathcal{H}^*\hookrightarrow\V_{\sigma}^*
\end{align*}
are dense and compact. Let us further denote $\mathcal{V}_0:=\mathbb{H}^1_0(\mathcal{O})=\mathrm{H}_0^1(\mathcal{O};\R^d)$  and define 
\begin{align*}
	\V_{\sigma, 0}:=\left\{\bv\in\mathcal{V}_0:\mathrm{div \ }\bv=0\ \text{ in }\ \mathcal{O}\right\}. 
\end{align*}
The sum space $\V^*+{\L}^{p'},$ where $\frac{1}{p}+\frac{1}{p'}$, $p\in[2,\infty)$, is well-defined and forms a Banach space under the norm
\begin{align}\label{22}
	\|\bv\|_{\V^*+{\L}^{p'}} :&= \inf\left\{\|\bv_1\|_{\V^*} + \|\bv\|_{{\L}^{p'}} : \bv = \bv_1 + \bv_2, \bv_1 \in \V^*, \bv_2 \in {\L}^{p'}\right\}\nonumber\\
	&= \sup\left\{\frac{|\langle\bv, \g\rangle|}{\|\g\|_{\V\cap{\L}^p}} : \boldsymbol{0} \neq \g \in \V \cap {\L}^p\right\},
\end{align}
where the norm on the intersection space $\V \cap {\L}^p$ is defined by
$
\|\cdot\|_{\V\cap{\L}^p} := \max\left\{\|\cdot\|_{\V}, \|\cdot\|_{{\L}^p}\right\}.
$
This norm is equivalent to both $\|\bv\|_{\V} + \|\bv\|_{{\L}^p}$ and $\sqrt{\|\bv\|_{\V}^2 + \|\bv\|_{{\L}^p}^2}$ on the space $\V \cap {\L}^p$.
Furthermore, the following continuous embeddings hold:
$$
\V \cap {\L}^p \hookrightarrow \V \hookrightarrow \mathcal{H}\cong \mathcal{H}^{*} \hookrightarrow \V^{*} \hookrightarrow \V^{*}+{\L}^{p'},
$$
with the embedding $\V \hookrightarrow \mathcal{H}$ being compact.

For the pressure $p$, we define the space 
\begin{align*}
Q:=\left\{q\in\mathrm{L}^2(\mathcal{O};\mathbb{R}): (q,1)_{\mathcal{O}}:=\int_{\mathcal{O}}q(\x)\d\x=0\right\},
\end{align*}
with the norm $\|q\|_{Q}:=\|q\|_{\mathrm{L}^2(\mathcal{O})}$. 

\subsection{Bilinear, trilinear and nonlinear forms}
We now introduce three bilinear forms, one trilinear form and a nonlinear form, which will be used in the weak formulation of the problem. For all $\bu, \bv, \bw\in \mathcal{V}$ and $q \in Q$, these forms are defined as follows: 
\begin{align*}
	\mathfrak{a}(\bu,\bv)&=\int_{\mathcal{O}}2\boldsymbol{\varepsilon}(\boldsymbol{u}):\boldsymbol{\varepsilon}(\boldsymbol{v})\d\x,\ \
	\mathfrak{a}_0(\bu,\bv)=\int_{\mathcal{O}}\boldsymbol{u}\cdot\boldsymbol{v}\d\x,\\
	\mathfrak{d}(\bv,q)&=-\int_{\mathcal{O}}q\mathrm{div\ }\bv\d\x,\\
	\mathfrak{b}(\bu,\bv,\bw)&=\int_{\mathcal{O}}(\bu\cdot\nabla)\bv\cdot\bw\d\x,\\
	\mathfrak{c}(\bu,\bv)&=\int_{\mathcal{O}}|\bu|^{r-1}\bu\cdot\bv\d\x,\ \
	\mathfrak{c}_0(\bu,\bv)=\int_{\mathcal{O}}|\bu|^{q-1}\bu\cdot\bv\d\x.
\end{align*}
Moreover, for $\f\in\L^2(\mathcal{O}):=\mathrm{L}^2(\mathcal{O};\mathbb{R}^d)$, we write 
\begin{align*}
	(\f,\bv)_{\L^2(\mathcal{O})}=\int_{\mathcal{O}}\f\cdot\bv\d\x,
\end{align*}
and for $\f\in\V^*$, we use the notation $\langle\f,\bv\rangle$, for any $\bv\in\V$. 

To develop a numerical algorithm based on the linearization procedure described in Subsection \ref{sub-sol-alg}, we also introduce the following notations:
\begin{align}\label{eqn-c-c0}
	\mathfrak{c}(\bu,\bv,\bw)=\int_{\mathcal{O}}|\bu|^{r-1}\bv\cdot\bw\d\x,\ \
	\mathfrak{c}_0(\bu,\bv,\bw)=\int_{\mathcal{O}}|\bu|^{q-1}\bv\cdot\bw\d\x.
\end{align}
Note that the bilinear form $\mathfrak{a}(\cdot, \cdot)$ is both bounded and coercive on the space $\mathcal{V}$. Specifically, this implies that $\mathfrak{a}(\bu, \bv)$ satisfies the continuity and ellipticity (coercivity) conditions for all $\bu, \bv \in \mathcal{V}$. More precisely, we have:
\begin{align}
|\mathfrak{a}(\bu,\bv)|	&\leq 2\|\bu\|_{\V}\|\bv\|_{\V}\ \text{ for all }\ \bu,\bv\in\V,\label{eqn-a-est-1}\\
\mathfrak{a}(\bu,\bu)&=2\|\bu\|_{\V}^2 \ \text{ for all }\ \bu\in\V. \label{eqn-a-est-2}
\end{align}
The bilinear form $\mathfrak{d}(\cdot,\cdot)$ is bounded in $\V\times Q$, since 
\begin{align}
	|\mathfrak{d}(\bv,q)|\leq C\|\bv\|_{\V}\|q\|_Q \ \text{ for all } \bv \in \V, q \in Q.
\end{align}
Using the Gagliardo-Nirenberg inequality (\cite[Theorem 1, pages 11-12]{LNi}) and  Korn's inequality  (see \eqref{eqn-equiv-1}), one can show that the trilinear form $\mathfrak{b}(\cdot,\cdot,\cdot)$ is bounded, since 
\begin{align}\label{eqn-b-est-1}
	|\mathfrak{b}(\bu,\bv,\bw)|&\leq \|\bu\|_{\mathbb{L}^4(\mathcal{O})}\|\nabla\bv\|_{\mathbb{L}^2(\mathcal{O})}\|\bw\|_{\mathbb{L}^4(\mathcal{O})}\nonumber\\&\leq C_k C_g^2\|\bu\|_{\H}^{1-\frac{d}{4}}\|\bu\|_{\H^1}^{\frac{d}{4}}\|\bv\|_{\V}\|\bw\|_{\H}^{1-\frac{d}{4}}\|\bw\|_{\H^1}^{\frac{d}{4}}\nonumber\\&\leq C_b\|\bu\|_{\V}\|\bv\|_{\V}\|\bw\|_{\V},
\end{align}
where $$C_b=C_k^3C_g^2,$$ $C_g$ is the constant appearing in the Gagliardo-Nirenberg inequality  and $C_k$ is defined in \eqref{eqn-equiv-1}. Moreover, we have 
\begin{align}
	&\mathfrak{b}(\bu,\bv,\bw)=-\mathfrak{b}(\bu,\bw,\bv)\ \text{ for all }\bu,\bv,\bw\in\V, \label{eqn-b-est-3}\\
	&\mathfrak{b}(\bu,\bv,\bv)=0\ \text{ for all }\bu,\bv\in\V. \label{eqn-b-est-4}
\end{align}
The nonlinear form $\mathfrak{c}(\cdot,\cdot)$ is bounded in $\L^{r+1}(\mathcal{O})\times \L^{r+1}(\mathcal{O})$, since 
\begin{align}
	|\mathfrak{c}(\bu,\bv)|\leq\|\bu\|_{\L^{r+1}}^r\|\bv\|_{\L^{r+1}}. 
\end{align}
Furthermore, we infer 
\begin{align}\label{eqn-c-est-1}
	\mathfrak{c}(\bu,\bu)=\|\bu\|_{\L^{r+1}}^{r+1}. 
\end{align}
From \cite[Section 2.4]{SGMTM}, we infer 
for all 	$\bu,\bv\in\L^{r+1}(\mathcal{O})$ and $r\geq 1$ that 
\begin{align}\label{2.23}
	&\langle\bu|\bu|^{r-1}-\bv|\bv|^{r-1},\bu-\bv\rangle\geq \frac{1}{2}\||\bu|^{\frac{r-1}{2}}(\bu-\bv)\|_{\H}^2+\frac{1}{2}\||\bv|^{\frac{r-1}{2}}(\bu-\bv)\|_{\H}^2\geq 0,
\end{align}
and 
\begin{align}\label{Eqn-mon-lip}
	&\langle\bu|\bu|^{r-1}-\bv|\bv|^{r-1},\bu-\bv\rangle
	\geq \frac{1}{2^{r-1}}\|\bu-\bv\|_{\L^{r+1}}^{r+1}.
\end{align}
For $\mathcal{C}(\bu)=|\bu|^{r-1}\bu$ such that $\mathfrak{c}(\bu,\bv)=\langle\mathcal{C}(\bu),\bv\rangle$, for all $\bv\in{\L}^{r+1}$,  the Gateaux derivative is given by 
\begin{align}\label{Gaetu}
	\mathcal{C}'(\bu)\bv&=\left\{\begin{array}{cl}\bv,&\text{ for }r=1,\\ \left\{\begin{array}{cc}|\bu|^{r-1}\bv+(r-1)\frac{\bu}{|\bu|^{3-r}}(\bu\cdot\bv),&\text{ if }\bu\neq \boldsymbol{0},\\\boldsymbol{0},&\text{ if }\bu=\boldsymbol{0},\end{array}\right.&\text{ for } 1<r<3,\\ |\bu|^{r-1}\bv+(r-1)\bu|\bu|^{r-3}(\bu\cdot\bv), &\text{ for }r\geq 3.\end{array}\right.
\end{align}
Similar estimates hold for $\mathfrak{c}_0(\bu,\bv)$ also. 

\subsection{Assumptions on the superpotential}
With regard to the superpotential $j$,  we impose the following assumptions:
\begin{hypothesis}\label{hyp-sup-j}
	$j:\Gamma_1\times\R^d\to \R$ is such that 
	\begin{itemize}
		\item [(H1)] $j (\cdot,\boldsymbol{\xi})$ is measurable on $\Gamma_1$ for all $\boldsymbol{\xi}\in\R^d$  and $j (\cdot,\boldsymbol{0})\in \mathbb{L}^1(\Gamma_1)$;
		\item [(H2)] $j(\x,\cdot)$ is locally Lipschitz on $\R^d$ for a.e. $\x\in\Gamma_1$; 
		\item [(H3)] $|\boldsymbol{\eta}|\leq k_0+k_1|\boldsymbol{\xi}|$ for all $\boldsymbol{\xi}\in\R^d$, $\boldsymbol{\eta}\in\partial j(\x,\boldsymbol{\xi})$ for a.e. $\x\in\Gamma_1$ with $k_0,k_1\geq 0$;
		\item [(H4)] $(\boldsymbol{\eta}_1-\boldsymbol{\eta}_2)\cdot(\boldsymbol{\xi}_1-\boldsymbol{\xi}_2)\geq -\delta_1|\boldsymbol{\xi}_1-\boldsymbol{\xi}_2|^2$ for all $\boldsymbol{\xi}_i\in\R^d$, $\boldsymbol{\eta}_i\in \partial j(\x,\boldsymbol{\xi}_i)$, $i=1,2,$ for a.e. $\x\in\Gamma_1$  with $\delta_1\geq 0$. 
	\end{itemize}
\end{hypothesis}

Condition (H4) is commonly referred to in the literature as a relaxed monotonicity condition (see \cite[Definition 3.49]{SMAOMS}). It can also be equivalently formulated as follows:
\begin{align}\label{eqn-alternative}
	j^0(\boldsymbol{\xi}_1;\boldsymbol{\xi}_2-\boldsymbol{\xi}_1)+	j^0(\boldsymbol{\xi}_2;\boldsymbol{\xi}_1-\boldsymbol{\xi}_2)\leq\delta_1|\boldsymbol{\xi}_1-\boldsymbol{\xi}_2|^2\ \text{ for all }\ \boldsymbol{\xi}_1,\boldsymbol{\xi}_2\in\R^d. 
	\end{align}
	Let us now consider the functional $J: \L^2(\Gamma)\to\R$ defined by 
	\begin{align}\label{eqn-J-def}
		J(\bv)=\int_{\Gamma_1}j(\x, \bv_{\tau}(\x))\d S,\ \bv\in\L^2(\mathcal{O}). 
	\end{align}
	
	The following result is adapted from \cite[Lemma 13]{SMAO} and \cite[Lemma 6.2]{CFWH}, with some minor modifications.
	
	\begin{lemma}\label{lem-hemi-var}
		Assume that $j : \Gamma_1 \times \R \to \R$ satisfies Hypothesis \ref{hyp-sup-j}. Then the functional $J$, defined by equation \eqref{eqn-J-def}, has the following properties:
			\begin{enumerate}
			\item $J(\cdot)$ is locally Lipschitz in $\mathbb{L}^2(\Gamma_1)$. 
			\item $\|\z\|_{\mathbb{L}^2(\Gamma_1)}\leq k_0|\Gamma_1|^{1/2}+k_1\|\bv\|_{\L^2(\Gamma_1)}$ for all $\bv\in\L^2(\Gamma_1)$, $\z\in\partial J(\bv)$ with $k_1,k_2\geq 0$,
			\item $J^0(\bu;\bv)\leq\int_{\Gamma_1}j^0(\bu_{\tau};\bv_{\tau})\d S$ for all $\bu,\bv\in\mathbb{L}^2(\Gamma_1)$. 
			\item $(\z_1-\z_2,\bu_1-\bu_2)_{\mathbb{L}^2(\Gamma_1)}\geq -\delta_1\|\bu_1-\bu_2\|_{\mathbb{L}^2(\Gamma_1)}^2$ for all $\z_i\in\partial J(\bu_i)$, $\bu_i\in \L^2(\Gamma_1)$, $i=1,2$, with $\delta_1\geq 0$. 
		\end{enumerate}
	\end{lemma}

	\subsection{Problem formulation}
	To derive the weak formulation of the problem \eqref{eqn-CBF}-\eqref{eqn-boundary-2}, it is helpful to rewrite the first equation \eqref{eqn-CBF} as follows: 
	\begin{align}\label{eqn-concrete-new}
		-2\mu\mathrm{div}(\boldsymbol{\varepsilon}(\bu))+(\bu\cdot\nabla)\bu+\alpha\bu+\beta|\bu|^{r-1}\bu+\kappa|\bu|^{q-1}\bu+\nabla p=\f, \ \text{ in }\ \mathcal{O}.
	\end{align}
	Assume that the problem \eqref{eqn-CBF}–\eqref{eqn-boundary-2} admits a smooth solution $(\bu, p)$, so that all subsequent calculations are well-defined. Taking the inner product of equation \eqref{eqn-concrete-new} with an arbitrary smooth test function $\bv \in \V \cap \L^{r+1}(\mathcal{O})$, we obtain:
		\begin{align}
		&\int_{\mathcal{O}}\left[-2\mu\mathrm{div}(\boldsymbol{\varepsilon}(\bu))\cdot\bv+(\bu\cdot\nabla)\bu\cdot\bv+\alpha\bu\cdot\bv+\beta|\bu|^{r-1}\bu\cdot\bv+\kappa|\bu|^{q-1}\bu\cdot\bv+\nabla p\cdot \bv\right]\d\x \nonumber\\&\qquad=\int_{\mathcal{O}}\f\cdot\bv\d\x. 
	\end{align}
	Performing integration by parts, we deduce 
	\begin{align}
		\int_{\mathcal{O}}\Big[2\mu\boldsymbol{\varepsilon}(\bu)\cdot\boldsymbol{\varepsilon}(\bu)+(\bu\cdot\nabla\bu)\cdot\bv+\alpha\bu\cdot\bv&+\beta|\bu|^{r-1}\bu\cdot\bv+\kappa|\bu|^{q-1}\bu\cdot\bv-p\mathrm{div\ }\bv\Big]\d\x\nonumber\\-\int_{\Gamma}\boldsymbol{\sigma}\n\cdot\bv\d S &=\int_{\mathcal{O}}\f\cdot\bv\d\x. 
	\end{align}
	Applying the boundary conditions satisfied by $\bv$, we obtain 
	\begin{align}
		\mu \mathfrak{a}(\bu,\bv)+\mathfrak{b}(\bu,\bu,\bv)+\alpha \mathfrak{a}_0(\bu,\bv)&+\beta \mathfrak{c}(\bu,\bv)+\kappa \mathfrak{c}_0(\bu,\bv)+\mathfrak{d}(\bv,p)\nonumber\\+\int_{\Gamma_1}(-\boldsymbol{\sigma}_\tau)\cdot\bv_{\tau}\d S&=	(\f,\bv)_{\L^2(\mathcal{O})}. 
	\end{align}
	Using the boundary condition \eqref{eqn-boundary-2}, we infer 
	\begin{align}
		-\boldsymbol{\sigma}_\tau\in\partial j(\bu_{\tau})\ \text{ on }\ \Gamma_1, 
	\end{align}
	so that 
	\begin{align}
		\int_{\Gamma_1}(-\boldsymbol{\sigma}_\tau)\cdot\bv_{\tau}\d S\leq \int_{\Gamma_1}j^0(\bu_{\tau};\bv_{\tau})\d S.
	\end{align}
	Therefore, assuming that $\f\in\V^*$,  for smooth $\bv\in\V\cap\L^{r+1}$, we have 
	\begin{align}
		\mu \mathfrak{a}(\bu,\bv)+\mathfrak{b}(\bu,\bu,\bv)+\alpha \mathfrak{a}_0(\bu,\bv)&+\beta \mathfrak{c}(\bu,\bv)+\kappa \mathfrak{c}_0(\bu,\bv)+\mathfrak{d}(\bv,p)\nonumber\\+\int_{\Gamma_1}j^0(\bu_{\tau};\bv_{\tau})\d S&\geq 	\langle\f,\bv\rangle. 
	\end{align}
	Next, taking the inner product of the second equation in \eqref{eqn-CBF} with an arbitrary function $q\in Q$, we obtain
	\begin{align}
		\mathfrak{d}(\bu,q)=0. 
	\end{align}
	
	In summary, we have derived the following hemivariational inequality corresponding to the problem \eqref{eqn-CBF}-\eqref{eqn-boundary-2}: 
	\begin{problem}\label{prob-hemi-1}
		Find $\bu\in\V\cap\L^{r+1}(\mathcal{O})$ and $p\in Q$  such that
		\begin{equation}\label{eqn-hemi-1}
			\left\{
			\begin{aligned}
					\mu \mathfrak{a}(\bu,\bv)+\mathfrak{b}(\bu,\bu,\bv)+\alpha \mathfrak{a}_0(\bu,\bv)&+\beta \mathfrak{c}(\bu,\bv)+\kappa \mathfrak{c}_0(\bu,\bv)+\mathfrak{d}(\bv,p)\\+\int_{\Gamma_1}j^0(\bu_{\tau};\bv_{\tau})\d S&\geq 	\langle\f,\bv\rangle \ \text{ for all }\ \bv\in \V\cap\L^{r+1}(\mathcal{O}),\\
					\mathfrak{d}(\bu,q)&=0\ \text{ for all }\ q\in Q. 
			\end{aligned}\right.
		\end{equation}
	\end{problem}
	
	The unknown variable $p$ can be eliminated, resulting in the following reduced hemivariational inequality: 
	\begin{problem}\label{prob-hemi-2}
		Find $\bu\in\V_{\sigma}\cap\L^{r+1}(\mathcal{O})$ such that 
			\begin{equation}\label{eqn-hemi-2}
			\left\{
			\begin{aligned}
				\mu \mathfrak{a}(\bu,\bv)+\mathfrak{b}(\bu,\bu,\bv)+\alpha \mathfrak{a}_0(\bu,\bv)&+\beta \mathfrak{c}(\bu,\bv)+\kappa \mathfrak{c}_0(\bu,\bv)+\int_{\Gamma_1}j^0(\bu_{\tau};\bv_{\tau})\d S\\&\geq 	\langle\f,\bv\rangle \ \text{ for all }\ \bv\in \V_{\sigma}\cap\L^{r+1}(\mathcal{O}).
			\end{aligned}\right.
		\end{equation}
	\end{problem}
	
	The existence and uniqueness of solutions to Problems \ref{prob-hemi-1} and \ref{prob-hemi-2} will be established in the following section (Theorems \ref{thm-main-hemi}, \ref{thm-unique} and \ref{thm-equivalent}).

	\section{The CBFeD Hemivariational Inequality}\label{sec3}\setcounter{equation}{0}
	
	Our first aim is to study Problem \ref{prob-hemi-2}. We impose the following so-called \emph{smallness condition} for the existence of a solution:
	\begin{align}\label{eqn-cond}
		k_1<2\mu\lambda_0,
	\end{align}
where	the constant $k_1$ comes from the Hypothesis \ref{hyp-sup-j}, while $\lambda_0$ denotes the smallest eigenvalue of the corresponding eigenvalue problem: 
\begin{align}
	\bu\in\V, \ \int_{\mathcal{O}}\boldsymbol{\varepsilon}(\bu):\boldsymbol{\varepsilon}(\bv)\d\x =\lambda\int_{\Gamma_1}\bu_{\tau}\cdot\bv_{\tau}\d S\ \text{ for all }\ \bv\in\V. 
\end{align}
Since the trace operator $\bu\mapsto\bu_{\tau}\big|_{\Gamma_1}$ is compact from $\V$ to $\L^2(\Gamma_1)$ 
 (since $\H^1(\mathcal{O})\hookrightarrow \L^2(\Gamma)$ compactly, see \cite[Theorem 2.21]{SMAOMS}), by using the spectral theory for compact self-adjoint operators, there exists a sequence $\{\lambda_k\}_{k=0}^{\infty}$ such that 
$ \lambda_k > 0,$ $ \lambda_k \to +\infty$.  We also have the following trace inequality:
\begin{align}\label{eqn-trace}
	\|\bv_{\tau}\|_{\L^2(\Gamma_1)}\leq\lambda_0^{-1/2}\|\bv\|_{\V} \ \text{ for all }\ \bv\in\V. 
\end{align}

The result below establishes the existence of a solution to Problem \ref{prob-hemi-2}. It is important to note that solvability is ensured under the assumptions $r \in [1, \infty)$ and $q \in [1, r)$, without requiring any further conditions on the exponent $r$.

\begin{theorem}\label{thm-main-hemi}
	Under Hypothesis \ref{hyp-sup-j} (H1)-(H3) and \eqref{eqn-cond},  there exists a solution $\bu\in\V_{\sigma}\cap\L^{r+1}(\mathcal{O})$ to Problem \ref{prob-hemi-2}. 
\end{theorem}
\begin{proof}
	Let us define the linear operator $\mathcal{A}:\V\to\V^*$ by 
	\begin{align*}
		\langle\mathcal{A}\bu,\bv\rangle=\mathfrak{a}(\bu,\bv),\  \bu,\bv\in\V, 
	\end{align*}
the bilinear operator  $\mathcal{B}:\V\to\V^*$ by 
	\begin{align*}
		\langle\B(\bu,\bv),\bw\rangle=\mathfrak{b}(\bu,\bv,\bw),\ \bu,\bv,\bw\in\V,
	\end{align*}
the nonlinear  operators $\mathcal{C}, \mathcal{C}_0:\L^{r+1}\to\L^{\frac{r+1}{r}}$ by 
	\begin{align*}
		\langle\mathcal{C}(\bu),\bv\rangle=\mathfrak{c}(\bu,\bv), \	\langle\mathcal{C}_0(\bu),\bv\rangle=\mathfrak{c}_0(\bu,\bv),  \ \bu,\bv\in\L^{r+1}. 
	\end{align*}
	\vskip 0.1cm
	\noindent 
	\textbf{Step 1:} \emph{Coercivity and pseudomonotonicity of the operator $\mathcal{F}$.} 
	Since $\V_{\sigma}$ is a subspace of $\V$,	we will show that the operator $\mathcal{F}:\V_{\sigma}\cap\L^{r+1}\to\V_{\sigma}^*+\L^{\frac{r+1}{r}}$ defined by 
	\begin{align}
		\mathcal{F}(\bu):=\mu\A\bu+\B(\bu)+\alpha\bu+\beta\C(\bu)+\kappa\mathcal{C}_0(\bu),\ \bu\in \V_{\sigma}\cap\L^{r+1},
	\end{align}
	is coercive and pseudomonotone. 	The proof is divided into the following parts:
	\vskip 0.1cm
	\noindent 
	\underline{\emph{Part 1.} \emph{$\mathcal{F}:\V_{\sigma}\cap\L^{r+1}\to\V_{\sigma}^*+\L^{\frac{r+1}{r}}$ is coercive.}}
	Let us first prove the coerciveness of the operator $\mathcal{F}(\bu):=\mu\A\bu+\B(\bu)+\alpha\bu+\beta\mathcal{C}(\bu)$. We infer from \eqref{eqn-a-est-1}, \eqref{eqn-b-est-4}, and \eqref{eqn-c-est-1} that 
	\begin{align}\label{eqn-coer}
		\langle\mathcal{F}(\bu),\bu\rangle =2\mu\|\bu\|_{\V}^2+\alpha\|\bu\|_{\H}^2+\beta\|\bu\|_{\L^{r+1}}^{r+1}+\kappa\|\bu\|_{\L^{q+1}}^{q+1}. 
	\end{align}
	Using H\"older's and Young's inequalities, we find 
	\begin{align}\label{eqn-pump-est}
	\kappa	\|\bu\|_{\L^{q+1}}^{q+1}&=|\kappa|\int_{\mathcal{O}}|\bu(\x)|^{q+1}\d\x\leq|\kappa||\mathcal{O}|^{\frac{r-q}{r+1}}\bigg(\int_{\mathcal{O}}|\bu(\x)|^{r+1}\d\x\bigg)^{\frac{q+1}{r+1}}\nonumber\\&=|\kappa||\mathcal{O}|^{\frac{r-q}{r+1}}\|\bu\|_{\L^{r+1}}^{q+1}\leq\frac{\beta}{2}\|\bu\|_{\L^{r+1}}^{r+1}+|\kappa|^{\frac{r+1}{r-q}}|\mathcal{O}|, 
	\end{align}
where $ |\mathcal{O}|$ denotes the Lebesgue measure of $\mathcal{O}$. 	We defined an equivalent  norm in $\V_{\sigma}\cap\L^{r+1}$ as $\sqrt{\|\bu\|_{\V}^2+\|\bu\|_{\L^{r+1}}^{2}}$. Therefore, using \eqref{eqn-pump-est} in \eqref{eqn-coer}, we obtain  (\cite{SGMTM})
	\begin{align}\label{eqn-coercive-1}
		\frac{\langle\mathcal{F}(\bu),\bu\rangle}{\|\bu\|_{\V\cap\L^{r+1}}}&\geq \frac{2\mu\|\bu\|_{\V}^2+\alpha\|\bu\|_{\H}^2+\frac{\beta}{2}\|\bu\|_{\L^{r+1}}^{r+1}-|\kappa|^{\frac{r+1}{r-q}}|\mathcal{O}|}{\sqrt{\|\bu\|_{\V}^2+\|\bu\|_{\L^{r+1}}^{2}}}\nonumber\\&\geq\frac{2\min\{\mu,\beta\}\left(\|\bu\|_{\V}^2+\|\bu\|_{\L^{r+1}}^2-1\right)-|\kappa|^{\frac{r+1}{r-q}}|\mathcal{O}|}{\sqrt{\|\bu\|_{\V}^2+\|\bu\|_{\L^{r+1}}^{2}}},
	\end{align}
	where we have used the fact that $x^2\leq x^{r+1} + 1,$ for all $x\geq 0$ and $r\geq 1$. Finally, we deduce 
	\begin{align*}
		\lim\limits_{\|\bu\|_{\V\cap\L^{r+1}}\to\infty}	\frac{\langle\mathcal{F}(\bu),\bu\rangle}{\|\bu\|_{\V\cap\L^{r+1}}}=\infty,
	\end{align*}
	so that the operator $\mathcal{F}:\V_{\sigma}\cap\L^{r+1}\to \V_{\sigma}^*+\L^{\frac{r+1}{r}}$ is coercive. 
	
	\vskip 0.1cm
	\noindent 
\underline{\emph{Part 2.} \emph{Boundedness of $\mathcal{F}:\V_{\sigma}\cap\L^{r+1}\to\V_{\sigma}^{*}+\L^{\frac{r+1}{r}}$.}}
	In order to prove the pseudomonotone property of the operator $\mathcal{F}:\V_{\sigma}\cap\L^{r+1}\to \V_{\sigma}^{*}+\L^{\frac{r+1}{r}}$, we first show the boundedness of $\mathcal{F}$ (Definition \ref{def-pseudo}-(1)). Using \eqref{eqn-a-est-1}, we have  for all $\bu,\bv\in\V_{\sigma}$ 
	\begin{align*}
		|\langle\mathcal{A}\bu,\bv\rangle|&=|\mathfrak{a}(\bu,\bv)|\leq 2\|\bu\|_{\V}\|\bv\|_{\V}. 
	\end{align*}
	The bound \eqref{eqn-b-est-1} implies  for all $\bu,\bv\in\V$  that 
	\begin{align*}
		|\langle\B(\bu),\bv\rangle|\leq C_b\|\bu\|_{\V}^2\|\bv\|_{\V}. 
	\end{align*}
	Moreover, an application of H\"older's inequality yields for all $\bu,\bv\in\L^{r+1}$  that 
	\begin{align*}
		|\langle\mathcal{C}(\bu),\bv\rangle|&\leq\|\bu\|_{\L^{r+1}}^r\|\bv\|_{\L^{r+1}},\\
		|\langle\mathcal{C}_0(\bu),\bv\rangle|&\leq\|\bu\|_{\L^{q+1}}^q\|\bv\|_{\L^{q+1}}\leq|\mathcal{O}|^{\frac{r-q}{r+1}}\|\bu\|_{\L^{r+1}}^q\|\bv\|_{\L^{r+1}}. 
	\end{align*}
	Therefore, for all $\bu\in\V_{\sigma}\cap\L^{r+1}$, we have 
	\begin{align*}
		\sup_{\|\bv\|_{\V\cap\L^{r+1}}\leq 1}|\langle\mathcal{F}(\bu),\bv\rangle|\leq \left(2\mu\|\bu\|_{\V}+C\alpha\|\bu\|_{\H}+C_b\|\bu\|_{\V}^2+\beta\|\bu\|_{\L^{r+1}}^r +|\mathcal{O}|^{\frac{r-q}{r+1}}\|\bu\|_{\L^{r+1}}^q\right)<\infty, 
	\end{align*}
	and 	the boundedness of the operator $\mathcal{F}:\V_{\sigma}\cap\L^{r+1}\to \V_{\sigma}^{*}+\L^{\frac{r+1}{r}}$ follows.
	
		\vskip 0.1cm
	\noindent 
	\underline{\emph{Part 3.} \emph{Local monotonicity of the operator  $\mathcal{F}:\V_{\sigma}\cap\L^{r+1}\to\V_{\sigma}^{*}+\L^{\frac{r+1}{r}}$.}} We divide the proof into the following cases:  
	
\noindent \textbf{Case 1:} \emph{$d\in\{2,3\}$ with $r\in(3,\infty)$.} 
	We first consider the supercritical  case $d\in\{2,3\}$  with $3<r<\infty$. 	Using \eqref{eqn-a-est-2}, we infer for all $\bu,\bv\in\V_{\sigma}$ that 
	\begin{align}\label{eqn-aes}
		\mu\langle\A(\bu-\bv),\bu-\bv\rangle=\mu \mathfrak{a}(\bu-\bv,\bu-\bv)=2\mu\|\bu-\bv\|_{\V}^2. 
	\end{align}
	From \eqref{2.23}, we have  
	\begin{align}\label{eqn-ces}
		&\langle\mathcal{C}(\bu)-\mathcal{C}(\bv),\bu-\bv\rangle\geq \frac{1}{2}\||\bu|^{\frac{r-1}{2}}(\bu-\bv)\|_{\H}^2+\frac{1}{2}\||\bv|^{\frac{r-1}{2}}(\bu-\bv)\|_{\H}^2. 
	\end{align}
	Let us now estimate the term $|\langle\B(\bu)-\B(\bv),\bu-\bv\rangle|$ using  \eqref{eqn-b-est-4},  \eqref{eqn-equiv-1},  H\"older's and Young's inequalities as 
	\begin{align}\label{eqn-bdes}
	&	|	\langle\B(\bu)-\B(\bv),\bu-\bv\rangle|\nonumber\\&\leq |\langle \B(\bu,\bu-\bv),\bu-\bv\rangle|+ |\langle \B(\bu-\bv,\bv),\bu-\bv\rangle|\nonumber\\&=|\langle \B(\bu-\bv,\bv),\bu-\bv\rangle|\leq\|\nabla(\bu-\bv)\|_{\H}\|\bv(\bu-\bv)\|_{\H}\nonumber\\&\leq C_k\|\bu-\bv\|_{\V}\|\bv(\bu-\bv)\|_{\H}\leq\mu\|\bu-\bv\|_{\V}^2+\frac{C_k^2}{4\mu}\|\bv(\bu-\bv)\|_{\H}^2. 
	\end{align}
	For $r>3$, by using H\"older's and Young's inequalities, we estimate $\frac{C_k^2}{2\mu}\|\bv(\bu-\bv)\|_{\H}^2$  as (\cite[Theorem 2.5]{SGMTM}, \cite{MT2})
	\begin{align}\label{eqn-des}
	\frac{C_k^2}{4\mu}	\|\bv(\bu-\bv)\|_{\H}^2&=	\frac{C_k^2}{4\mu}\int_{\mathcal{O}}|\bv(\x)|^2|\bu(\x)-\bv(\x)|^2\d \x\nonumber\\&=\frac{C_k^2}{4\mu}\int_{\mathcal{O}}|\bv(\x)|^2|\bu(\x)-\bv(\x)|^{\frac{4}{r-1}}|\bu(\x)-\bv(\x)|^{\frac{2(r-3)}{r-1}}\d \x\nonumber\\&\leq\frac{C_k^2}{4\mu}\left(\int_{\mathcal{O}}|\bv(\x)|^{r-1}|\bu(\x)-\bv(\x)|^2\d \x\right)^{\frac{2}{r-1}}\left(\int_{\mathcal{O}}|\bu(\x)-\bv(\x)|^2\d \x\right)^{\frac{r-3}{r-1}}\nonumber\\&\leq\frac{\beta}{4}\int_{\mathcal{O}}|\bv(\x)|^{r-1}|\bu(\x)-\bv(\x)|^2\d \x+\varrho_{1,r}\int_{\mathcal{O}}|\bu(\x)-\bv(\x)|^2\d \x,
	\end{align}
	where 
	\begin{align}\label{eqn-rho-1}
		\varrho_{1,r}=\left(\frac{C_k^2}{4\mu}\right)^{\frac{r-1}{r-3}}\left(\frac{r-3}{r-1}\right)\left(\frac{8}{\beta (r-1)}\right)^{\frac{2}{r-3}},
	\end{align}
and $C_k$ is defined in \eqref{eqn-equiv-1}. Using \eqref{eqn-des} in \eqref{eqn-bdes}, we find 
	\begin{align}\label{eqn-bes}
		|	\langle\B(\bu)-\B(\bv),\bu-\bv\rangle|&\leq\mu\|\bu-\bv\|_{\V}^2+\frac{\beta}{4}\||\bv|^{\frac{r-1}{2}}(\bu-\bv)\|_{\H}^2+\varrho_{1,r}\|\bu-\bv\|_{\H}^2.
	\end{align}
	An application of  Taylor's formula (\cite[Theorem 7.9.1]{PGC}) yields 
	\begin{align}\label{eqn-est-c0-1}
	|	\kappa||\langle\mathcal{C}_0(\bu)-\mathcal{C}_0(\bv),\bu-\bv\rangle|&= |\kappa|\bigg|\bigg<\int_0^1\mathcal{C}_0^{\prime}(\theta\bu+(1-\theta)\bv)\d\theta(\bu-\bv),(\bu-\bv)\bigg>\bigg|\nonumber\\&\leq |\kappa| q2^{q-1}\bigg<\int_0^1|\theta\bu+(1-\theta)\bv|^{q-1}\d\theta|\bu-\bv|,|\bu-\bv|\bigg>\nonumber\\&\leq |\kappa|q2^{q-1}\left<\left(|\bu|^{q-1}+|\bv|^{q-1}\right)|\bu-\bv|,|\bu-\bv|\right>\nonumber\\&=   |\kappa| q2^{q-1}\||\bu|^{\frac{q-1}{2}}(\bu-\bv)\|_{\H}^2+ |\kappa| q2^{q-1}\||\bv|^{\frac{q-1}{2}}(\bu-\bv)\|_{\H}^2. 
	\end{align}
	Using H\"older's inequality, we estimate $ |\kappa| q2^{q-1}\||\bu|^{\frac{q-1}{2}}(\bu-\bv)\|_{\H}^2$ as 
	\begin{align}\label{eqn-est-c0-2}
		& |\kappa| q2^{q-1}\||\bu|^{\frac{q-1}{2}}(\bu-\bv)\|_{\H}^2\nonumber\\&= |\kappa| q2^{q-1}\int_{\mathcal{O}}|\bu(\x)|^{q-1}|\bu(\x)-\bv(\x)|^2\d\x \nonumber\\&= |\kappa| q2^{q-1}\int_{\mathcal{O}}|\bu(\x)|^{q-1}|\bu(\x)-\bv(\x)|^{\frac{2(q-1)}{r-1}}|\bu(\x)-\bv(\x)|^{\frac{2(r-q)}{r-1}}\d\x \nonumber\\&\leq |\kappa| q2^{q-1}\bigg(\int_{\mathcal{O}}|\bu(\x)|^{r-1}|\bu(\x)-\bv(\x)|^2\d\x\bigg)^{\frac{q-1}{r-1}}\bigg(\int_{\mathcal{O}}|\bu(\x)-\bv(\x)|^2\d\x\bigg)^{\frac{r-q}{r-1}}\nonumber\\&\leq\frac{\beta}{2}\int_{\mathcal{O}}|\bu(\x)|^{r-1}|\bu(\x)-\bv(\x)|^2\d\x+\varrho_{2,r}\int_{\mathcal{O}}|\bu(\x)-\bv(\x)|^2\d\x,
	\end{align}
	where 
	\begin{align}\label{eqn-rho-2}
		\varrho_{2,r}=\left(\frac{r-q}{r-1}\right)\left(\frac{2(q-1)}{\beta(r-1)}\right)^{\frac{q-1}{r-q}}\left( |\kappa| q2^{q-1}\right)^{\frac{r-1}{r-q}}. 
	\end{align}
	A similar calculation yields 
	\begin{align}\label{eqn-est-c0-3}
	& |\kappa| q2^{q-1}\||\bv|^{\frac{q-1}{2}}(\bu-\bv)\|_{\H}^2\nonumber\\&\leq\frac{\beta}{4}\int_{\mathcal{O}}|\bv(\x)|^{r-1}|\bu(\x)-\bv(\x)|^2\d\x+\varrho_{3,r}\int_{\mathcal{O}}|\bu(\x)-\bv(\x)|^2\d\x,
\end{align}
where 
	\begin{align}\label{eqn-rho-3}
	\varrho_{3,r}=\left(\frac{r-q}{r-1}\right)\left(\frac{4(q-1)}{\beta(r-1)}\right)^{\frac{q-1}{r-q}}\left( |\kappa| q2^{q-1}\right)^{\frac{r-1}{r-q}}. 
\end{align}
Using \eqref{eqn-est-c0-2} and \eqref{eqn-est-c0-3} in \eqref{eqn-est-c0-1}, we deduce 
\begin{align}\label{eqn-est-c0-4}
	|\kappa||\langle\mathcal{C}_0(\bu)-\mathcal{C}_0(\bv),\bu-\bv\rangle|&\leq\frac{\beta}{2}\||\bu|^{\frac{r-1}{2}}(\bu-\bv)\|_{\L^2}^2+\frac{\beta}{4}\||\bv|^{\frac{r-1}{2}}(\bu-\bv)\|_{\L^2}^2\nonumber\\&\quad+(\varrho_{2,r}+\varrho_{3,r})\|\bu-\bv\|_{\H}^2. 
\end{align}
	Combining \eqref{eqn-aes}, \eqref{eqn-ces}, \eqref{eqn-bes} and \eqref{eqn-est-c0-4}, we finally get for $r>3$ 
	\begin{align}\label{eqn-final-est}
		\langle\mathcal{F}(\bu)-\mathcal{F}(\bv),\bu-\bv\rangle+\varrho_{r}\|\bu-\bv\|_{\H}^2\geq\mu\|\bu-\bv\|_{\V}^2+\alpha\|\bu-\bv\|_{\H}^2\geq 0,
	\end{align}
	where $\varrho_r=\varrho_{1,r}+\varrho_{2,r}+\varrho_{3,r}$, and $\varrho_{1,r}$, $\varrho_{2,r}$ and $\varrho_{3,r}$ are defined in \eqref{eqn-rho-1}, \eqref{eqn-rho-2} and \eqref{eqn-rho-3}, respectively.

\noindent	\textbf{Case 2:} \emph{$d\in\{2,3\}$ with $r=3$ ($2\beta\mu\geq 1$).} 
			For the critical case $r=3$ ($d\in\{2,3\}$)  with $2\beta\mu \geq 1$, the operator $\F(\cdot):\V_{\sigma}\cap\L^{4}\to \V_{\sigma}^{*}+\L^{\frac{4}{3}}$ satisfies,  for all $\bu,\bv\in\V_{\sigma}$
	\begin{align}\label{218}
		\langle\F(\bu)-\F(\bv),\bu-\bv\rangle+(\varrho_{2,3}+\varrho_{3,3})\|\bu-\bv\|_{\H}^2\geq 0.
		\end{align}
	From \eqref{2.23}, we infer
	\begin{align}\label{231}
		\beta\langle\mathcal{C}(\bu)-\mathcal{C}(\bv),\bu-\bv\rangle\geq\frac{\beta}{2}\|\bu(\bu-\bv)\|_{\H}^2+\frac{\beta}{2}\|\bv(\bu-\bv)\|_{\H}^2. 
	\end{align}
	We estimate $|\langle\B(\bu-\bv,\bu-\bv),\bv\rangle|$ using H\"older's and Young's inequalities as 
	\begin{align}\label{232}
		|\langle\B(\bu-\bv,\bu-\bv),\bv\rangle|\leq\|\bv(\bu-\bv)\|_{\H}\|\bu-\bv\|_{\V} \leq\frac{1}{\beta} \|\bu-\bv\|_{\V}^2+\frac{\beta}{4 }\|\bv(\bu-\bv)\|_{\H}^2.
	\end{align}
Combining \eqref{eqn-aes}, \eqref{eqn-est-c0-4}, \eqref{231} and \eqref{232}, we deduce 
	\begin{align}\label{eqn-critical}
	&	\langle\F(\bu)-\F(\bv),\bu-\bv\rangle+(\varrho_{2,3}+\varrho_{3,3})\|\bu-\bv\|_{\H}^2\nonumber\\&\geq \left(2\mu-\frac{1}{\beta}\right)\|\bu-\bv\|_{\V}^2+\alpha\|\bu-\bv\|_{\H}^2+ \frac{\beta}{4}\|\bu(\bu-\bv)\|_{\H}^2\geq 0,
	\end{align}
	provided $2\beta\mu \geq 1$. 
	
\noindent		\textbf{Case 3:} \emph{$d\in\{2,3\}$ with $r\in[1,3]$.} 
 For this case, one can estimate $|\langle\B(\bu-\bv,\bu-\bv),\bv\rangle|$ using \eqref{eqn-equiv-1}, H\"older's, Gagliardo-Nirenberg's and Young's inequalities  as 
	\begin{align}\label{2.21}
		|\langle\B(\bu-\bv,\bu-\bv),\bv\rangle|&\leq \|\nabla(\bu-\bv)\|_{\H}\|\bu-\bv\|_{\L^4}\|\bv\|_{\L^4}
		\nonumber\\&\leq C_gC_k\|\bu-\bv\|_{\V}\|\bu-\bv\|_{\H}^{1-\frac{d}{4}}\|\bu-\bv\|_{\H^1}^{\frac{d}{4}}\|\bv\|_{\L^4}
		\nonumber\\&\leq C_gC_k\|\bu-\bv\|_{\H}^{1-\frac{d}{4}}\|\bu-\bv\|_{\V}^{1+\frac{d}{4}}\|\bv\|_{\L^4}\nonumber\\&\leq \mu\|\bu-\bv\|_{\V}^2+\varrho_4\|\bv\|_{\L^4}^{\frac{8}{4-d}}\|\bu-\bv\|_{\H}^2,
	\end{align}
	where 
	\begin{align}
		\varrho_4=C_g^{\frac{8}{4-d}}C_k^{\frac{8}{4-d}}\left(\frac{8}{4-d}\right)\left(\frac{4+d}{8\mu}\right)^{\frac{4+d}{4-d}}.
	\end{align}
	Combining \eqref{eqn-aes}, \eqref{eqn-ces},  \eqref{eqn-est-c0-4}, and \eqref{2.21}, we obtain 
	\begin{align}\label{eqn-diff}
		\langle\F(\bu)-\F(\bv),\bu-\bv\rangle+\varrho_4 \|\bv\|_{\L^4}^{\frac{8}{4-d}}\|\bu-\bv\|_{\H}^2 +(\varrho_{2,r}+\varrho_{3,r})\|\bu-\bv\|_{\H}^2\geq 0,
	\end{align}
	which implies 
	\begin{align}\label{fe2}
		\langle\F(\bu)-\F(\bv),\bu-\bv\rangle+\left(\varrho_4N^{\frac{8}{4-d}}+\varrho_{2,r}+\varrho_{3,r}\right)\|\bu-\bv\|_{\H}^2\geq 0,
	\end{align}
	for all $\bv\in{\mathbb{B}}_N$, where ${\mathbb{B}}_N$ is an $\L^4$-ball of radius $N$, that is,
	$
	{\mathbb{B}}_N:=\big\{\z\in\L^4:\|\z\|_{\L^4}\leq N\big\}.
	$ Thus, the operator $\F(\cdot)$ is locally monotone in this case. 
	

	\vskip 0.1cm
	\noindent 
	\underline{\emph{Part 4.} \emph{Pseudomonotonicity of the operator $\wi{\mathcal{F}}(\bu):=\mu\A\bu+\alpha\bu+\beta\mathcal{C}(\bu)+\kappa\mathcal{C}_0(\bu).$}}
Let us now demonstrate that the operator
$$
\widetilde{\mathcal{F}}(\cdot) := \mu \mathcal{A} + \alpha \mathrm{I} + \beta \mathcal{C}(\cdot)+\kappa\mathcal{C}_0(\cdot): \V_\sigma \cap \L^{r+1} \to \V_\sigma^* + \L^{\frac{r+1}{r}}
$$
is pseudomonotone. To this end, consider a sequence $\{\bu_n\}_{n \in \mathbb{N}} \subset \V_\sigma \cap \L^{r+1}$ such that
$$
\bu_n \rightharpoonup \bu\  \text{ in }\  \V_\sigma \cap \L^{r+1} \quad \text{and} \quad \limsup_{n \to \infty} \langle \widetilde{\mathcal{F}}(\bu_n), \bu_n - \bu \rangle \leq 0.
$$
By the definition of the norm in the space $\V_\sigma \cap \L^{r+1}$, it follows directly that
$$
\bu_n \rightharpoonup \bu \ \text{ in }\  \V_\sigma \ \text{ and } \  \bu_n \rightharpoonup \bu\  \text{ in }\  \L^{r+1}.
$$
Since weak convergence implies boundedness, the sequence $\{\bu_n\}$ is uniformly bounded in $\V_\sigma \cap \L^{r+1}$. 
Due to the compact embedding $\V_\sigma \hookrightarrow \H$, there exists a (not relabeled) subsequence such that
\begin{align}\label{eqn-strong-1}
\bu_n \to \bu \ \text{ strongly in }   \ \H,
\end{align}
and, possibly along a further subsequence,
\begin{align}\label{eqn-strong-2}
\bu_n(\x) \to \bu(\x) \ \text{ for a.e. } \ \x \in \mathcal{O}.
\end{align}
The operator $\widetilde{\mathcal{F}}+\widetilde{\varrho}\I$ is monotone as shown through its individual components (see \eqref{eqn-aes}, \eqref{eqn-ces}, and \eqref{eqn-est-c0-4}), so we obtain
$$
\langle \widetilde{\mathcal{F}}(\bu_n) - \widetilde{\mathcal{F}}(\bu), \bu_n - \bu \rangle+\widetilde\varrho\|\bu_n-\bu\|_{\H}^2 \geq 0,
$$
where $\widetilde\varrho=\varrho_{2,r}+\varrho_{3,r}$, which yields
$$
\langle \widetilde{\mathcal{F}}(\bu_n), \bu_n - \bu \rangle+\widetilde\varrho\|\bu_n-\bu\|_{\H}^2  \geq \langle \widetilde{\mathcal{F}}(\bu), \bu_n - \bu \rangle.
$$
Taking the limit infimum as $n \to \infty$ and using the strong convergence given in \eqref{eqn-strong-1}, we arrive at
$$
\liminf_{n \to \infty} \langle \widetilde{\mathcal{F}}(\bu_n), \bu_n - \bu \rangle \geq \liminf_{n \to \infty} \langle \widetilde{\mathcal{F}}(\bu), \bu_n - \bu \rangle = 0,
$$
where the last equality follows from the weak convergence of $\bu_n$ to $\bu$. Combining this with the original assumption,
$$
\limsup_{n \to \infty} \langle \widetilde{\mathcal{F}}(\bu_n), \bu_n - \bu \rangle \leq 0,
$$
we conclude that
$$
\lim_{n \to \infty} \langle \widetilde{\mathcal{F}}(\bu_n), \bu_n - \bu \rangle = 0.
$$
Next, we prove that $\widetilde{\mathcal{F}}(\bu_n) \rightharpoonup \widetilde{\mathcal{F}}(\bu)$ in $\V_\sigma^* + \L^{\frac{r+1}{r}}$. For any test function $\bv \in \V_\sigma \cap \L^{r+1}$, consider
\begin{align}\label{eqn-conv-est-1}
\langle \widetilde{\mathcal{F}}(\bu_n) - \widetilde{\mathcal{F}}(\bu), \bv \rangle &= \mu \mathfrak{a}(\bu_n - \bu, \bv) + \alpha(\bu_n - \bu, \bv) + \beta \langle \mathcal{C}(\bu_n) - \mathcal{C}(\bu), \bv \rangle\nonumber\\&\quad+\kappa\langle \mathcal{C}_0(\bu_n) - \mathcal{C}_0(\bu), \bv \rangle.
\end{align}
The first two terms on the right-hand side converge to zero due to the weak convergence in $\V$
and the strong convergence in $\H$, respectively. For the nonlinear terms involving $\mathcal{C}$ and $\mathcal{C}_0$, observe that from almost everywhere convergence of $\bu_n$ to $\bu$ (see \eqref{eqn-strong-2}), we get
$$
|\bu_n(\x)|^{r-1} \bu_n(\x) \to |\bu(\x)|^{r-1} \bu(\x) \quad \text{a.e. } \x \in \mathcal{O}.
$$
The uniform boundedness of $\|\mathcal{C}(\bu_n)\|_{\L^{\frac{r+1}{r}}}$ (since $\|\mathcal{C}(\bu_n)\|_{\L^{\frac{r+1}{r}}}=\|\bu_n\|_{\L^{r+1}}^r\leq C$) implies, via the generalized Lebesgue dominated convergence theorem (Brezis-Lions Lemma, see Lemma \ref{Lem-Lions}), that
$$
\mathcal{C}(\bu_n) \rightharpoonup \mathcal{C}(\bu) \ \text{ in } \ \L^{\frac{r+1}{r}}.
$$
A similar argument shows that 
$$
\mathcal{C}_0(\bu_n) \rightharpoonup \mathcal{C}_0(\bu) \ \text{ in } \ \L^{\frac{r+1}{r}}.
$$
Therefore, the last two terms in \eqref{eqn-conv-est-1} also vanishes as $n \to \infty$, giving
$$
\langle \widetilde{\mathcal{F}}(\bu_n) - \widetilde{\mathcal{F}}(\bu), \bv \rangle \to 0.
$$
Since the above convergence holds for all $\bv \in \V_\sigma \cap \L^{r+1}$, we conclude that $\widetilde{\mathcal{F}}(\bu_n) \rightharpoonup \widetilde{\mathcal{F}}(\bu)$. Thus, both conditions for pseudomonotonicity are satisfied, and we conclude that the operator $\widetilde{\mathcal{F}}: \V_\sigma \cap \L^{r+1} \to \V_\sigma^* + \L^{\frac{r+1}{r}}$ is pseudomonotone.

	\vskip 0.1cm
	\noindent 
	\underline{\emph{Part 5.} \emph{Pseudomonotonicity of the operator $\mathcal{F}(\bu):=\wi{\mathcal{F}}(\bu)+\B(\bu).$}}
	To complete the proof of pseudomonotone property of $\mathcal{F}$, we need to verify Definition \ref{def-pseudo}-(2).  Let us consider a sequence $\{\bu_n\}_{n\in\N}\in\V_{\sigma}\cap\L^{r+1}$ such that 
	\begin{align}\label{eqn-pseudo-1}
		\bu_n\rightharpoonup \bu\ \text{ in }\ \V_{\sigma}\cap\L^{r+1}\ \text{ and }\ \limsup\limits_{n\to\infty} \langle{\mathcal{F}}(\bu_n),\bu_n-\bu\rangle\leq 0, 
	\end{align}
	and let $\bv\in \V_{\sigma}\cap\L^{r+1}$. Note that we have in hand the convergences stated in \eqref{eqn-strong-1} and \eqref{eqn-strong-2}.	Using \eqref{eqn-b-est-4}, we immediately get 
	\begin{align}\label{eqn-bcon}
		\langle\B(\bu_n),\bu_n-\bv\rangle-	\langle\B(\bu),\bu-\bv\rangle&=\langle\B(\bu_n),\bu_n\rangle-\langle\B(\bu_n),\bv\rangle-\langle\B(\bu),\bu\rangle+\langle\B(\bu),\bv\rangle\nonumber\\&=\langle\B(\bu),\bv\rangle-\langle\B(\bu_n),\bv\rangle\nonumber\\&=\langle\B(\bu,\bu-\bu_n),\bv\rangle+\langle\B(\bu-\bu_n,\bu_n),\bv\rangle. 
	\end{align}
	Since $\B:\V_{\sigma}\to\V_{\sigma}^{*}$ is bounded (see \eqref{eqn-b-est-1}), the first term $\langle\B(\bu,\bu-\bu_n),\bv\rangle\to 0 $ as $n\to\infty$ by using the weak convergence $	\bu_n\rightharpoonup \bu\ \text{ in }\ \V_{\sigma}$. We use H\"older's and Gagliardo-Nirenberg's inequalities to estimate the second term in the right hand side of the equality \eqref{eqn-bcon} as 
	\begin{align*}
		|\langle\B(\bu-\bu_n,\bu_n),\bv\rangle|&\leq\|\bu-\bu_n\|_{\L^4}\|\nabla\bu_n\|_{\H}\|\bv\|_{\L^4}\nonumber\\&\leq C_k^2C_g \|\bu_n\|_{\V}\|\bu-\bu_n\|_{\H^1}^{\frac{d}{4}}\|\bu-\bu_n\|_{\H}^{1-\frac{d}{4}}\|\bv\|_{\V}\nonumber\\&\leq C_k^{2+\frac{d}{4}}C_g\|\bu_n\|_{\V}\left(\|\bu\|_{\V}^{\frac{d}{4}}+\|\bu_n\|_{\V}^{\frac{d}{4}}\right)\|\bu-\bu_n\|_{\H}^{1-\frac{d}{4}}\|\bv\|_{\V}\nonumber\\&\to 0\ \text{ as }\ n\to\infty, 
	\end{align*}
	where we have used the strong convergence \eqref{eqn-strong-1}. Therefore, from \eqref{eqn-bcon}, we  infer 
	\begin{align}\label{eqn-bcon-1}
		\lim_{n\to\infty}	\langle\B(\bu_n),\bu_n-\bv\rangle=	\langle\B(\bu),\bu-\bv\rangle\ \text{ for all }\ \bv\in\V_{\sigma}. 
	\end{align}
	By taking $\bu=\bv$ in \eqref{eqn-bcon-1}, we have 
	\begin{align}\label{eqn-bcon-2}
		\lim_{n\to\infty}	\langle\B(\bu_n),\bu_n-\bu\rangle=0. 
	\end{align}
	Let us now consider 
	\begin{align}
		\limsup_{n\to\infty}\langle\wi{\mathcal{F}}(\bu_n),\bu_n-\bu\rangle&= \limsup_{n\to\infty}\langle\wi{\mathcal{F}}(\bu_n),\bu_n-\bu\rangle+\lim_{n\to\infty}	\langle\B(\bu_n),\bu_n-\bu\rangle\nonumber\\&=\limsup_{n\to\infty}\langle\wi{\mathcal{F}}(\bu_n)+\B(\bu_n),\bu_n-\bu\rangle\nonumber\\&=\limsup_{n\to\infty}\langle{\mathcal{F}}(\bu_n),\bu_n-\bu\rangle\leq 0, 
	\end{align}
	by using \eqref{eqn-pseudo-1}. Moreover, the pseudomonotonicity of the operator $\wi{\mathcal{F}}$ implies 
	\begin{align}\label{eqn-fcon-1}
		\langle\wi{\mathcal{F}}(\bu),\bu-\bv\rangle\leq \liminf_{n\to\infty}\langle\wi{\mathcal{F}}(\bu_n),\bu_n-\bv\rangle\ \text{ for all }\ \bv\in\V_{\sigma}\cap\L^{r+1}. 
	\end{align}
	Therefore, using \eqref{eqn-bcon-1} and \eqref{eqn-fcon-1}, we finally have 
	\begin{align}
		\langle{\mathcal{F}}(\bu),\bu-\bv\rangle\leq \liminf_{n\to\infty}\langle{\mathcal{F}}(\bu_n),\bu_n-\bv\rangle\ \text{ for all }\ \bv\in\V_{\sigma}\cap\L^{r+1}, 
	\end{align}
	which completes the proof.

	\vskip 0.1cm
\noindent 
\textbf{Step 2:} \emph{Coercivity and pseudomonotonicity of the operator $\mathscr{F}=\mathcal{F}(\bu)+\gamma^*(\partial J(\gamma\bu))$.} 
	Let $\gamma : \mathbb{H}^{1/2}(\Gamma) \to \mathbb{L}^2(\Gamma)$ denote the natural  embedding operator, and let $\gamma^*$ be its adjoint. We consider the operator $\mathscr{F} : \V_{\sigma} \cap\L^{r+1}\to 2^{\V_{\sigma}^*+\L^{\frac{r+1}{r}}}$ defined by
		\begin{align}
		\mathscr{F}(\bu):&=\mu\A\bu+\B(\bu)+\alpha\bu+\beta\C(\bu)+\kappa\mathcal{C}_0(\bu)+\gamma^*(\partial J(\gamma\bu))\nonumber\\&=\mathcal{F}(\bu)+\gamma^*(\partial J(\gamma\bu)),
	\end{align}
	for all $\bu\in \V_{\sigma}\cap\L^{r+1}$. 
		Let us now show that the multi-valued operator $\mathscr{F}(\cdot)$  is pseudomonotone and it is coercive under the condition \eqref{eqn-cond}. 
		
		Let us first prove the coercivity of the map $\mathscr{F}$  for $\mu> \frac{k_{1}}{2\lambda_0}$ (see \eqref{eqn-cond}). Let $\bu\in\V_{\sigma}$ and $\bu^*\in\mathscr{F}(\bu)$. Then for some $\boldsymbol{\eta}\in \partial J(\gamma\bu)$, we have 
		\begin{align}
			\bu^*=\mathcal{F}(\bu)+\gamma^*\boldsymbol{\eta}. 
		\end{align}
		Therefore,  it is immediately that 
		\begin{align}\label{eqn-coer-1}
			\langle\bu^*,\bu\rangle=2\mu\|\bu\|_{\V}^2+\alpha\|\bu\|_{\H}^2+\beta\|\bu\|_{\L^{r+1}}^{r+1}+\kappa\|\bu\|_{\L^{q+1}}^{q+1}+(\boldsymbol{\eta},\gamma\bu)_{\L^2(\Gamma_1)}.
		\end{align}
		We estimate the final term in the right hand side of \eqref{eqn-coer-1} by using Hypothesis \ref{hyp-sup-j} (H3) (see Lemma \ref{lem-hemi-var} (2) also) and \eqref{eqn-trace} as 
		\begin{align}\label{eqn-coer-2}
			|(\boldsymbol{\eta},\gamma\bu)_{\L^2(\Gamma_1)}|&\leq \|\boldsymbol{\eta}\|_{\L^2(\Gamma_1)}\|\gamma\bu\|_{\L^2(\Gamma_1)}=\|\boldsymbol{\eta}\|_{\L^2(\Gamma_1)}\|\bu_{\tau}\|_{\L^2(\Gamma_1)}\nonumber\\&\leq\left(k_0|\Gamma_1|^{1/2}+k_1\|\gamma\bu\|_{\L^2(\Gamma_1)}\right)\|\bu_\tau\|_{\L^2(\Gamma_1)}\nonumber\\&\leq k_0|\Gamma_1|^{-1/2}\lambda_0^{1/2}\|\bu\|_{\V}+k_1\lambda_0\|\bu\|_{\V}^2.
		\end{align}
	One can derive  from \eqref{eqn-pump-est}, \eqref{eqn-coer-1}, and \eqref{eqn-coer-2} that 
		\begin{align}\label{eqn-coer-3}
			\langle\bu^*,\bu\rangle&\geq\left(2\mu-k_1\lambda_0^{-1}\right)\|\bu\|_{\V}^2+\alpha\|\bu\|_{\H}^2+\frac{\beta}{2}\|\bu\|_{\L^{r+1}}^{r+1}-|\kappa|^{\frac{r+1}{r-q}}|\mathcal{O}|\nonumber\\&\quad-k_0|\Gamma_1|^{1/2}\lambda_0^{-1/2}\|\bu\|_{\V}. 
		\end{align}
		A calculation similar to \eqref{eqn-coercive-1} yields 
		\begin{align}
			\frac{\langle\bu^*,\bu\rangle}{\|\bu\|_{\V\cap\L^{r+1}}}\geq\frac{\max\left\{\left(2\mu-k_1\lambda_0^{-1}\right),\frac{\beta}{2}\right\}\left(\|\bu\|_{\V}^2+\|\bu\|_{\L^{r+1}}^2-1\right)-|\kappa|^{\frac{r+1}{r-q}}|\mathcal{O}|-k_0|\Gamma_1|^{1/2}\lambda_0^{1/2}\|\bu\|_{\V}}{\sqrt{\|\bu\|_{\V}^2+\|\bu\|_{\L^{r+1}}^{2}}}.
		\end{align} 
		Therefore for $\mu> \frac{k_1\lambda_0^{-1}}{2}$, we immediately have 
		\begin{align*}
			\lim\limits_{\|\bu\|_{\V\cap\L^{r+1}}\to\infty}	\frac{\langle\bu^{*},\bu\rangle}{\|\bu\|_{\V\cap\L^{r+1}}}=\infty,
		\end{align*}
		so that the operator $\mathscr{F}:\V_{\sigma}\cap\L^{r+1}\to 2^{\V_{\sigma}^{*}+\L^{\frac{r+1}{r}}}$ is coercive.

		We now proceed to establish that the operator $\mathscr{F}$ is pseudomonotone. From the earlier discussion, it is already known that the operator $\mathcal{F}(\cdot)$ possesses the pseudomonotonicity property. To complete the argument, it remains to verify that the mapping $\gamma^* (\partial J(\gamma \cdot))$ is also pseudomonotone. This verification is carried out using Proposition \ref{prop-suff-pseudo}, following ideas similar to those found in \cite[Lemma 2]{PKa} and \cite[Proposition 5.6]{WHSM}.
		
		First, we address condition (1) of Proposition \ref{prop-suff-pseudo}. It is a well-known result that the Clarke subdifferential of a locally Lipschitz functional has nonempty, convex, and when the underlying space is reflexive, weakly compact values (see \cite[Proposition 2.1.2]{FHC} and also Proposition \ref{prop-clarke}). Since $\gamma$ is a linear operator, it follows that the set-valued map $\gamma^* (\partial J(\gamma \cdot))$ takes values in nonempty, closed, and convex subsets of $\H^{1/2}(\Gamma_1)^*$.
		
		Next, condition (2) of Proposition \ref{prop-suff-pseudo} is a consequence of the growth assumption imposed on $\partial J$ as stated in Hypothesis \ref{hyp-sup-j} (H3).
		
		To verify condition (3), consider a sequence $\bu_n \rightharpoonup \bu$ weakly in $\H^{1/2}(\Gamma_1)$, and assume $\boldsymbol{\xi}_n \rightharpoonup \boldsymbol{\xi}$ weakly in $(\H^{1/2}(\Gamma_1))^*$, with each $\boldsymbol{\xi}_n \in \gamma^*(\partial J(\gamma \bu_n))$. Since $\gamma: \H^{1/2}(\Gamma_1) \to \L^2(\Gamma_1)$ is a compact operator, we have strong convergence $\gamma \bu_n \to \gamma \bu$ in $\L^2(\Gamma_1)$.
		Now define $\boldsymbol{\eta}_n \in \partial J(\gamma \bu_n)$ such that $\boldsymbol{\xi}_n = \gamma^* \boldsymbol{\eta}_n$. By Hypothesis \ref{hyp-sup-j} (H3), the following uniform estimate holds:
			$$
		\| \boldsymbol{\eta}_n \|_{\L^2(\Gamma_1)} \leq k_0 |\Gamma_1|^{1/2} + k_1 \| \gamma \bu_n \|_{\L^2(\Gamma_1)} \leq C,
		$$
			for some constant $C > 0$, as weak convergence implies boundedness of the sequence $\{\bu_n\}_{n\in\N}$.
			Due to the reflexivity of $\L^2(\Gamma_1)$, and by the Banach-Alaoglu theorem, there exists a subsequence (still denoted $\boldsymbol{\eta}_n$) such that $\boldsymbol{\eta}_n \rightharpoonup \boldsymbol{\eta}$ weakly in $\L^2(\Gamma_1)$. The graph of the Clarke subdifferential $\partial J$ is weakly-strongly closed in $\L^2(\Gamma_1) \times \L^2(\Gamma_1)_w$ (see \cite[Proposition 2.1.5]{FHC}). Therefore, $\boldsymbol{\eta} \in \partial J(\gamma \bu)$, and hence $\boldsymbol{\xi} = \gamma^* \boldsymbol{\eta} \in \gamma^*(\partial J(\gamma \bu))$.
		Moreover, the duality pairing satisfies:
			$$
		\langle \boldsymbol{\xi}_n, \bu_n \rangle = (\boldsymbol{\eta}_n, \gamma \bu_n)_{\L^2(\Gamma_1)} \to (\boldsymbol{\eta}, \gamma \bu)_{\L^2(\Gamma_1)} = \langle \boldsymbol{\xi}, \bu \rangle,
		$$
			where the convergence holds for the entire sequence due to uniqueness of weak limits and strong convergence of $\gamma \bu_n$.
			This verifies condition (3) of Proposition \ref{prop-suff-pseudo}, establishing that $\gamma^*(\partial J(\gamma \cdot))$ is pseudomonotone. 
			
			Finally, since the sum of two pseudomonotone operators is itself pseudomonotone (see \cite[Proposition 1.3.68]{ZdSmP}), we conclude that the operator $\mathscr{F}$ is pseudomonotone.

				\vskip 0.1cm
			\noindent 
			\textbf{Step 3:} \emph{Surjectivity.} 
			By applying Theorem \ref{thm-surjective}, we conclude that the operator $\mathscr{F} : \V_{\sigma}\cap\L^{r+1} \to 2^{\V_{\sigma}^*+\L^{\frac{r+1}{r}}}$ is surjective. Consequently, for any $\f \in \V^*\hookrightarrow \V_{\sigma}^*\hookrightarrow\V_{\sigma}^* + \L^{\frac{r+1}{r}}$, there exists a function $\bu \in \V_{\sigma} \cap \L^{r+1}$ such that
			$$
			\mathscr{F}(\bu) \ni \f.
			$$
				This means that $\bu$ satisfies the following variational inequality:
			\begin{align} \label{eqn-variation-rephrased}
				\bu \in \V_{\sigma} \cap \L^{r+1},\quad &\mu \mathfrak{a}(\bu, \bv) + \mathfrak{b}(\bu, \bu, \bv) + \alpha \mathfrak{a}_0(\bu, \bv) + \beta \mathfrak{c}(\bu, \bv) +\kappa \mathfrak{c}_0(\bu,\bv)+ J^0(\bu_{\tau}; \bv_{\tau}) \nonumber \\
				&\geq {}_{\V^{\prime}}\langle \f, \bv \rangle_{\V} \quad \text{for all } \bv \in \V_{\sigma} \cap \L^{r+1}.
			\end{align}
				Furthermore, since the generalized directional derivative $J^0(\bu_{\tau}; \bv_{\tau})$ is bounded above by
				$$
			J^0(\bu_{\tau}; \bv_{\tau}) \leq \int_{\Gamma_1} j^0(\bu_{\tau}; \bv_{\tau})\d S,
			$$
				inequality \eqref{eqn-variation-rephrased} immediately implies the hemivariational inequality \eqref{eqn-hemi-2}. Thus, we conclude that Problem \ref{prob-hemi-2} admits at least one solution.
\end{proof}

Let us now demonstrate that every solution of Problem \ref{prob-hemi-2}  is bounded.
\begin{proposition}\label{prop-energy-est}
	Under the assumptions of Theorem \ref{thm-main-hemi}, if $\bu\in\V_{\sigma}\cap\L^{r+1}$ is a solution of Problem \ref{prob-hemi-2}, then 
	\begin{align}\label{eqn-bound}
		\|\bu\|_{\V}^2+	\|\bu\|_{\L^{r+1}}^{r+1}\leq 2\max\bigg\{\frac{1}{\left(2\mu-k_1\lambda_0^{-1}\right)},\frac{1}{\beta}\bigg\} K=:\widetilde{K},
	\end{align}
where
\begin{align}\label{eqn-value-k} K=\frac{1}{\left(2\mu-k_1\lambda_0^{-1}\right)}\left(\|\f\|_{\V^*}+k_0|\Gamma_1|^{1/2}\lambda_0^{-1/2}\right)^2+ |\kappa|^{\frac{r+1}{r-q}},
	\end{align}
	and  the constants $k_0$ and $k_1$ are from Hypothesis \ref{hyp-sup-j} (H3). 
\end{proposition}
\begin{proof}
Since $\bu\in\V_{\sigma}$ is a solution of Problem \ref{prob-hemi-2}, we find 
\begin{align}
	\langle\mathscr{F}(\bu),\bu\rangle+(\boldsymbol{\eta},\gamma\bu)_{\L^2(\Gamma_1)}=\langle\f,\bu\rangle, 
\end{align}
where $\boldsymbol{\eta}\in\partial J(\gamma\bu)$. Then by using the fact that $\mathfrak{b}(\bu,\bu,\bu)=0$, H\"older's inequality, \eqref{eqn-pump-est}, and  Hypothesis \ref{hyp-sup-j} (H3), we deduce 
\begin{align}\label{eqn-bound-2}
	&2\mu\|\bu\|_{\V}^2+\alpha\|\bu\|_{\H}^2+\beta\|\bu\|_{\L^{r+1}}^{r+1}\nonumber\\&= \langle\f,\bu\rangle-\kappa\|\bu\|_{\L^{q+1}}^{q+1}- (\boldsymbol{\eta},\gamma\bu)_{\L^2(\Gamma_1)}\nonumber\\&\leq\|\f\|_{\V^*}\|\bu\|_{\V}+\frac{\beta}{2}\|\bu\|_{\L^{r+1}}^{r+1}+|\kappa|^{\frac{r+1}{r-q}}|\mathcal{O}|+\|
\boldsymbol{\eta}\|_{\L^2(\Gamma_1)}\|\gamma\bu\|_{\L^2(\Gamma_1)}\nonumber\\&\leq  \|\f\|_{\V^*}\|\bu\|_{\V}+ \left(k_0|\Gamma_1|^{1/2}+k_1\|\bu_{\tau}\|_{\L^2(\Gamma_1)}\right)\|\bu_{\tau}\|_{\L^2(\Gamma_1)}+\frac{\beta}{2}\|\bu\|_{\L^{r+1}}^{r+1}+|\kappa|^{\frac{r+1}{r-q}}|\mathcal{O}|\nonumber\\&\leq\left(\|\f\|_{\V^*}+k_0|\Gamma_1|^{1/2}\lambda_0^{-1/2}\right)\|\bu\|_{\V}+k_1\lambda_0^{-1}\|\bu\|_{\V}^2+\frac{\beta}{2}\|\bu\|_{\L^{r+1}}^{r+1}+|\kappa|^{\frac{r+1}{r-q}}|\mathcal{O}|\nonumber\\&\leq\frac{\left(2\mu-k_1\lambda_0^{-1}\right)}{2}\|\bu\|_{\V}^2+\frac{1}{\left(2\mu-k_1\lambda_0^{-1}\right)}\left(\|\f\|_{\V^*}+k_0|\Gamma_1|^{1/2}\lambda_0^{-1/2}\right)^2 \nonumber\\&\quad+k_1\lambda_0^{-1}\|\bu\|_{\V}^2+\frac{\beta}{2}\|\bu\|_{\L^{r+1}}^{r+1}+|\kappa|^{\frac{r+1}{r-q}}|\mathcal{O}|,
\end{align}
so that the estimate \eqref{eqn-bound-2} follows provided $2\mu>k_1\lambda_0^{-1}$. 
\end{proof}

We now proceed to examine the uniqueness of the solution.

\begin{theorem}\label{thm-unique}
Under Hypothesis \ref{hyp-sup-j} (H1)-(H4), for $r\in(3,\infty)$, assume that either 
\begin{align}\label{eqn-unique-con-1}
	\mbox{$\mu>\frac{\delta_1}{2\lambda_0}$\  and \ $\alpha\geq (\widetilde{\varrho}_{1,r}+\varrho_{2,r}+\varrho_{3,r})$,}
\end{align}
where 
	\begin{align}\label{eqn-varrhotilde}
	\widetilde{\varrho}_{1,r}=\left(\frac{C_k^2}{2(2\mu-\delta_1\lambda_0^{-1})}\right)^{\frac{r-1}{r-3}}\left(\frac{r-3}{r-1}\right)\left(\frac{8}{\beta (r-1)}\right)^{\frac{2}{r-3}}
\end{align}
or 
	\begin{align}
	\mbox{ $\mu>\frac{\delta_1}{2\lambda_0}$, $\alpha\geq (\widehat{\varrho}_{1}+\varrho_{2,r}+\varrho_{3,r})$\  and \ $\beta\geq 4\widehat{\varrho}_{1}$,}
	\end{align}
	where 
		\begin{align}\label{eqn-varrhohat}
		\widehat{\varrho}_{1}=\frac{C_k^2}{2(2\mu-\delta_1\lambda_0^{-1})}
	\end{align}
holds.	Then Problem  \ref{prob-hemi-2} admits a unique solution. 

For $r\in[1,3]$, assume that 
\begin{align}\label{eqn-unique-con-2}
\mbox{$\mu>\frac{\delta_1}{2\lambda_0}$\  and \  $\alpha>\widehat{\varrho}_4(C_gC_k)^{\frac{8}{4-d}}\widetilde{K}^{\frac{4}{4-d}}+\varrho_{2,r}+\varrho_{3,r}$,}
\end{align}
where 
\begin{align}\label{eqn-varrhohat-1}
	\widehat{\varrho}_4=C_g^{\frac{8}{4-d}}C_k^{\frac{8}{4-d}}\left(\frac{8}{4-d}\right)\left(\frac{4+d}{2(2\mu-\delta_1\lambda_0^{-1})}\right)^{\frac{4+d}{4-d}},
\end{align}
and $\widetilde{K}$ is defined in \eqref{eqn-bound}, then Problem  \ref{prob-hemi-2} admits a unique solution. 
	
	Furthermore, the mapping $ \V^*\ni \boldsymbol{f} \mapsto \boldsymbol{u} \in \V_{\sigma}\cap\L^{r+1}$ is Lipschitz continuous.
\end{theorem}
\begin{proof}
	According to Theorem \ref{thm-main-hemi}, Problem  \ref{prob-hemi-2} admits a solution under the given assumptions. Therefore, it remains to establish the uniqueness of this solution. We only show the continuous dependence of the solution on the data, from which the uniqueness immediately follows. 
	
	Let us assume that Problem  \ref{prob-hemi-2}  has two solutions $\bu_1,\bu_2\in\V_{\sigma}\cap\L^{r+1}$ with the corresponding data $\f_1$ and $\f_2$, respectively.  Then $\bu_1-\bu_2$ satisfies for all $\bv\in\V\cap\L^{\frac{r+1}{r}}$
	\begin{align}\label{eqn-differece}
		\langle\mathcal{F}(\bu_1)-\mathcal{F}(\bu_2),\bv\rangle+(\boldsymbol{\eta}_1-\boldsymbol{\eta}_2,\gamma\bv)_{\L^2(\Gamma)}=\langle\f_1-\f_2,\bv\rangle, 
	\end{align}
	where $\boldsymbol{\eta}_1\in\partial J(\gamma\bu_1)$ and $\boldsymbol{\eta}_2\in\partial J(\gamma\bu_2)$. Taking $\bv=\bu_1-\bu_2$ in \eqref{eqn-differece}, we obtain 
	\begin{align}\label{eqn-differece-1}
		\langle\mathcal{F}(\bu_1)-\mathcal{F}(\bu_2),\bu_1-\bu_2\rangle&=\langle\f_1-\f_2,\bu_1-\bu_2\rangle-(\boldsymbol{\eta}_1-\boldsymbol{\eta}_2,\gamma(\bu_1-\bu_2))_{\L^2(\Gamma)}.
	\end{align}
	Using the Cauchy-Schwarz and Young's inequalities, we estimate $|\langle\f_1-\f_2,\bu_1-\bu_2\rangle|$ as 
	\begin{align}\label{eqn-differece-0}
	|\langle\f_1-\f_2,\bu_1-\bu_2\rangle|&\leq\|\f_1-\f_2\|_{\V^*}\|\bu_1-\bu_2\|_{\V}
		\nonumber\\&\leq\frac{(2\mu-\delta_1\lambda_0^{-1})}{4}\|\bu_1-\bu_2\|_{\V}^2+\frac{1}{(2\mu-\delta_1\lambda_0^{-1})}\|\f_1-\f_2\|_{\V^*}^2.
	\end{align}
	Using Hypothesis \ref{hyp-sup-j} (H4) and \eqref{eqn-trace}, we estimate the second term in the right hand side of \eqref{eqn-differece-1} as 
	\begin{align}\label{eqn-differ}
		-(\boldsymbol{\eta}_1-\boldsymbol{\eta}_2,\gamma(\bu_1-\bu_2))_{\L^2(\Gamma)}&\leq \delta_1\|\gamma(\bu_1-\bu_2)\|_{\L^2(\Gamma)}^2=\delta_1\|\bu_{1,\tau}-\bu_{2,\tau}\|_{\L^2(\Gamma_1)}^2\nonumber\\&\leq \delta_1\lambda_0^{-1}\|\bu_1-\bu_2\|_{\V}^2. 
	\end{align}
	Let us first consider the case $r>3$. Calculations similar to \eqref{eqn-bdes} and \eqref{eqn-des} yield 
	\begin{align}\label{eqn-differece-2}
		&	|	\langle\B(\bu_1)-\B(\bu_2),\bu_1-\bu_2\rangle|\nonumber\\&\leq\frac{(2\mu-\delta_1\lambda_0^{-1})}{2}\|\bu_1-\bu_2\|_{\V}^2+\frac{\beta}{4}\||\bu_2|^{\frac{r-1}{2}}(\bu_1-\bu_2)\|_{\H}^2+\widetilde{\varrho}_{1,r}\|\bu_1-\bu_2\|_{\H}^2,
	\end{align}
	where $\widetilde{\varrho}_{1,r}$ is defined in \eqref{eqn-varrhotilde}.	 Following a similar approach to the estimate in \eqref{eqn-est-c0-4}, we bound
	$|\kappa||\langle\mathcal{C}_0(\bu)-\mathcal{C}_0(\bv),\bu-\bv\rangle|$
	as follows:
	\begin{align}\label{eqn-difference-4}
		|\kappa||\langle\mathcal{C}_0(\bu_1)-\mathcal{C}_0(\bu_2),\bu_1-\bu_2\rangle|&\leq\frac{\beta}{2}\||\bu_1|^{\frac{r-1}{2}}(\bu_1-\bu_2)\|_{\L^2}^2+\frac{\beta}{4}\||\bu_2|^{\frac{r-1}{2}}(\bu_1-\bu_2)\|_{\L^2}^2\nonumber\\&\quad+(\varrho_{2,r}+\varrho_{3,r})\|\bu_1-\bu_2\|_{\H}^2,
	\end{align}
where $\varrho_{2,r}$ and $\varrho_{3,r}$  are defined in \eqref{eqn-rho-2} and \eqref{eqn-rho-3}, respectively. 	Therefore, using a calculation similar to \eqref{eqn-final-est}, we further have 
	\begin{align}\label{eqn-differece-3}
		&2 \mu\|\bu_1-\bu_2\|_{\V}^2+\alpha\|\bu_1-\bu_2\|_{\H}^2\nonumber\\&\leq \langle\mathcal{F}(\bu_1)-\mathcal{F}(\bu_2),\bu_1-\bu_2\rangle\nonumber\\&\quad+\frac{(2\mu-\delta_1\lambda_0^{-1})}{2}\|\bu_1-\bu_2\|_{\V}^2+(\widetilde{\varrho}_{1,r}+\varrho_{2,r}+\varrho_{3,r})\|\bu_1-\bu_2\|_{\H}^2. 
	\end{align}
	Combining \eqref{eqn-differece-0}, \eqref{eqn-differ}, \eqref{eqn-difference-4} and  \eqref{eqn-differece-3}, and applying it in \eqref{eqn-differece-1}, we deduce 
	\begin{align}\label{eqn-differece-4}
		&\left(2\mu-\delta_1\lambda_0^{-1}\right)\|\bu_1-\bu_2\|_{\V}^2+4\left(\alpha-(\widetilde{\varrho}_{1,r}+\varrho_{2,r}+\varrho_{3,r})\right)\|\bu_1-\bu_2\|_{\H}^2\nonumber\\& \quad \leq \frac{4}{(2\mu-\delta_1\lambda_0^{-1})}\|\f_1-\f_2\|_{\V^*}^2,
	\end{align}
which is the continuous dependence result. 	The uniqueness follows by taking $\f_1=\f_2\in\V^*$,  $\mu>\frac{\delta_1}{2\lambda_0}$ and $\alpha\geq (\widetilde{\varrho}_{1,r}+\varrho_{2,r}+\varrho_{3,r})$.  
	
	One can estimate $|\langle\B(\bu-\bv,\bu-\bv),\bv\rangle|$ in the following way also: 
	\begin{align}\label{2.26}
		&	|\langle\B(\bu_1-\bu_2,\bu_1-\bu_2),\bu_2\rangle|\nonumber\\&\leq\|\nabla(\bu_1-\bu_2)\|_{\H}\|\bu_2(\bu_1-\bu_2)\|_{\H}
	\leq C_k\|\bu_1-\bu_2\|_{\V}\|\bu_1(\bu_1-\bu_2)\|_{\H}\nonumber\\&\leq \frac{(2\mu-\delta_1\lambda_0^{-1})}{2} \|\bu_1-\bu_2\|_{\V}^2+\frac{C_k^2}{2(2\mu-\delta_1\lambda_0^{-1}) }\int_{\mathcal{O}}|\bu_2(\x)|^2|\bu_1(\x)-\bu_2(\x)|^2\d\x\nonumber\\&= \frac{(2\mu-\delta_1\lambda_0^{-1})}{2} \|\bu_1-\bu_2\|_{\V}^2\nonumber\\&\quad+\frac{C_k^2}{2 (2\mu-\delta_1\lambda_0^{-1}) }\int_{\mathcal{O}}|\bu_1(\x)-\bu_2(\x)|^2\left(|\bu_2(\x)|^{r-1}+1\right)\frac{|\bu_2(\x)|^2}{|\bu_2(\x)|^{r-1}+1}\d\x\nonumber\\&\leq \frac{(2\mu-\delta_1\lambda_0^{-1})}{2} \|\bu_1-\bu_2\|_{\V}^2+\frac{C_k^2}{2(2\mu-\delta_1\lambda_0^{-1})}\int_{\mathcal{O}}|\bu_2(\x)|^{r-1}|\bu_1(\x)-\bu_2(\x)|^2\d\x\nonumber\\&\quad+\frac{C_k^2}{2(2\mu-\delta_1\lambda_0^{-1})}\int_{\mathcal{O}}|\bu_1(\x)-\bu_2(\x)|^2\d\x,
	\end{align}
	where 	we have used the fact that $\left\|\frac{|\bu_2|^2}{|\bu_2|^{r-1}+1}\right\|_{\L^{\infty}}<1$, for $r\geq 3$. Therefore, \eqref{eqn-differece-4} reduces to 
	\begin{align}
	&	\left(2\mu-\delta_1\lambda_0^{-1}\right)\|\bu_1-\bu_2\|_{\V}^2+4\left(\alpha-(\widehat{\varrho}_{1}+\varrho_{2,r}+\varrho_{3,r})\right)\|\bu_1-\bu_2\|_{\H}^2\nonumber\\&\quad+4\left(\frac{\beta}{4}-	\widehat{\varrho}_{1}\right)\||\bu_2|^{\frac{r-1}{2}}(\bu_1-\bu_2)\|_{\L^2}^2\leq \frac{4}{(2\mu-\delta_1\lambda_0^{-1})}\|\f_1-\f_2\|_{\V^*}^2,
	\end{align}
$\widehat{\varrho}_1$ is defined in \eqref{eqn-varrhohat} 	and the uniqueness follows provided $\f_1=\f_2\in\V^*$, $\mu>\frac{\delta_1}{2\lambda_0}$, $\alpha\geq (\widehat{\varrho}_{1}+\varrho_{2,r}+\varrho_{3,r})$ and $\beta\geq 4\widehat{\varrho}_{1}$. 
	
	For $d\in\{2,3\}$ with $r\in[1,3]$, a calculation similar to \eqref{eqn-diff} gives 
	\begin{align}\label{eqn-difference-5}
		&2\mu\|\bu_1-\bu_2\|_{\V}^2+\alpha\|\bu_1-\bu_2\|_{\H}^2\nonumber\\&\quad \leq 	\langle\F(\bu_1)-\F(\bu_2),\bu_1-\bu_2\rangle+\frac{(2\mu-\delta_1\lambda_0^{-1})}{4}\|\bu_1-\bu_2\|_{\V}^2\nonumber\\&\quad \quad+\widehat{\varrho}_4 \|\bv\|_{\L^4}^{\frac{8}{4-d}}\|\bu-\bv\|_{\H}^2 +(\varrho_{2,r}+\varrho_{3,r})\|\bu-\bv\|_{\H}^2,
	\end{align}
	where $	\widehat{\varrho}_4$ is defined in \eqref{eqn-varrhohat-1}.
	Using \eqref{eqn-differece-1}-\eqref{eqn-differ},  \eqref{Eqn-mon-lip}, and \eqref{eqn-bound} in \eqref{eqn-difference-5}, we arrive at 
	\begin{align}
		&	\frac{\left(2\mu-\delta_1\lambda_0^{-1}\right)}{2}\|\bu_1-\bu_2\|_{\V}^2+\left(\alpha-	\widehat{\varrho}_4 (C_gC_k)^{\frac{8}{4-d}}\widetilde{K}^{\frac{4}{4-d}}-\varrho_{2,r}-\varrho_{3,r}\right)\|\bu_1-\bu_2\|_{\H}^2\nonumber\\& \quad \leq \frac{1}{(2\mu-\delta_1\lambda_0^{-1})}\|\f_1-\f_2\|_{\V^*}^2,
	\end{align}
where $\widetilde{K}$ is defined in \eqref{eqn-bound}. 	Therefore, for $\mu>\frac{\delta_1}{2\lambda_0}$ and $\alpha>	\widehat{\varrho}_4(C_gC_k)^{\frac{8}{4-d}}\widetilde{K}^{\frac{4}{4-d}}+\varrho_{2,r}+\varrho_{3,r}$,  the uniqueness follows. 
\end{proof}

\begin{remark}\label{rem-station-unique}
	1. We highlight that the uniqueness condition for $r > 3$ is independent of the external forcing term $\f$, whereas for $1 \leq r \leq 3$, the dependence on $\f$ is evident from condition \eqref{eqn-unique-con-2}. This is a crucial observation, particularly in the context of optimal control problems governed by hemivariational inequalities for the 2D and 3D CBFeD equations, an area we intend to explore in our upcoming work.

2. 	For $d\in\{2,3\}$ and $r=3$, a calculation similar to \eqref{eqn-critical} provides 
	\begin{align*}
			& \left(2\mu-\frac{1}{\beta}\right)\|\bu-\bv\|_{\V}^2+\alpha\|\bu-\bv\|_{\H}^2\nonumber\\& \quad \leq 
				\langle\F(\bu)-\F(\bv),\bu-\bv\rangle+(\varrho_{2,3}+\varrho_{3,3})\|\bu-\bv\|_{\H}^2.
	\end{align*}
	Combining the above inequality with \eqref{eqn-differece-1}-\eqref{eqn-differ}, we find 
	\begin{align*}
		&\frac{1}{2}\left(2\mu-\frac{1}{\beta}-\delta_1\lambda_0^{-1}\right)\|\bu_1-\bu_2\|_{\V}^2+(\alpha-\varrho_{2,3}-\varrho_{3,3})\|\bu_1-\bu_2\|_{\H}^2\nonumber\\& \quad \leq  \frac{1}{2\left(2\mu-\frac{1}{\beta}-\delta_1\lambda_0^{-1}\right)}\|\f_1-\f_2\|_{\V^*}^2,
	\end{align*}
	and the uniqueness follows provided $\mu>\frac{1}{2\beta}+\frac{\delta_1}{\lambda_0}$ and $\alpha>\varrho_{2,3}+\varrho_{3,3}$. 
\end{remark}

We now shift our focus to the analysis of Problem \ref{prob-hemi-1}. Let us prove the inf-sup condition (also known as the Ladyzhenskaya-Babu\u{s}ka-Brezzi (LBB) condition)  (see  \cite[Theorem 2.2]{JSHNJW} also). From now onward, we assume that $1\leq r\leq\frac{2d}{d-2}$ for $d>2$ and $1\leq r<\infty$ for $d=2$. We use the notation $1\leq r\leq\frac{2d}{(d-2)^+}$ for this condition. In this case, we know by Sobolev's embedding that $\H^1(\mathcal{O})\hookrightarrow\L^{r+1}(\mathcal{O})$. Therefore,  Problem \ref{prob-hemi-1}  reduces to find $\bu\in\V$. 

We infer  from \cite[Lemma III.3.1]{GGP1} that   the Bogovskii operator $\mathfrak{B}:\mathrm{L}^{2}(\mathcal{O})\to \mathbb{H}_0^{1}(\mathcal{O})$ provides the right inverse for divergence such that 
\begin{align*}
	\mathrm{div}(\mathfrak{B}q)=q, \ \|\mathfrak{B}q\|_{\V_0}\leq C\|q\|_{\mathrm{L}^{2}},
	\end{align*}
	where  $\V_0=\mathbb{H}^{1}_0(\mathcal{O})$ and  the constant $C$ depends on $d$ and $\mathcal{O}$. Therefore, given any $q\in\mathrm{L}^{2}(\mathcal{O})$, there exists $\bv=\mathfrak{B}(q)$ such that 
	\begin{align*}
		\mathfrak{d}(\bv,q)=-\int_{\mathcal{O}}q\mathrm{div\ }\bv\d\x=-\int_{\mathcal{O}}q^2\d\x. 
	\end{align*}
	But since we need a positive supremizer, we take $\bv\mapsto-\bv$:
		\begin{align*}
		\mathfrak{d}(-\bv,q)=\|q\|_{\mathrm{L}^2}^2.
	\end{align*}
Therefore,  by using the fact that $\|\bv\|_{\V}\leq C\|\bv\|_{\V_0} \leq C\|q\|_{\mathrm{L}^{2}},$ we have 
	\begin{align*}
		\sup_{\bv\in\V_0}\frac{\mathfrak{d}(\bv,q)}{\|\bv\|_{\V}}\geq \frac{\|q\|_{\mathrm{L}^{2}}^2}{C\|q\|_{\mathrm{L}^{2}}}=\frac{1}{C}\|q\|_{\mathrm{L}^{2}}  \ \text{ for all }\ q\in Q, 
	\end{align*}
and the inf-sup condition
	\begin{align}\label{eqn-inf-sup}
\vartheta\|q\|_{\mathrm{L}^2}\leq 	\sup_{\bv\in\V_0}\frac{\mathfrak{d}(\bv,q)}{\|\bv\|_{\V}}  \ \text{ for all }\ q\in Q,
\end{align}
follows. We use the above condition to prove the next result.
\begin{theorem}\label{thm-equivalent}
	For $1\leq r\leq\frac{2d}{(d-2)^+}$, under the assumptions stated in Theorem \ref{thm-unique}, Problem \ref{prob-hemi-1} admits a unique solution.
	
	Moreover, $p\in Q$ depends locally Lipschitz continuously on $\f\in\V^*$. 
\end{theorem}
\begin{proof}
Let $\bu\in\V$ be  the unique solution to Problem 	\ref{prob-hemi-1}. Then   for all $ \bv\in \V_{0,\sigma}$, we have
\begin{align*}
	\mu \mathfrak{a}(\bu,\bv)+\mathfrak{b}(\bu,\bu,\bv)+\alpha \mathfrak{a}_0(\bu,\bv)+\beta \mathfrak{c}(\bu,\bv)+\kappa \mathfrak{c}_0(\bu,\bv)=	\langle\f,\bv\rangle.
\end{align*}
Owing to the inf-sup condition \eqref{eqn-inf-sup} and a classical result from functional analysis (cf. \cite[Theorem 2.2]{JSHNJW}), one can guarantee the existence of a function $p\in Q$ such that for all $\bv\in\V_0$ 
\begin{align}\label{eqn-pressure-1}
	\mu \mathfrak{a}(\bu,\bv)+\mathfrak{b}(\bu,\bu,\bv)+\alpha \mathfrak{a}_0(\bu,\bv)+\beta \mathfrak{c}(\bu,\bv)+\kappa \mathfrak{c}_0(\bu,\bv)+\mathfrak{d}(\bv,p)=	\langle\f,\bv\rangle.
\end{align}
Given an arbitrary $\bv \in \V$, the inf-sup condition \eqref{eqn-inf-sup} ensures the existence of a function $\bv_1 \in \V_0$ such that
\begin{align}\label{eqn-pressure-2}
\mathfrak{d}(\bv_1, q) = \mathfrak{d}(\bv, q) \  \text{ for all } \ q \in Q.
\end{align}
Let $\bv_2 = \bv - \bv_1$. Then $\bv_2 \in \V$, and it satisfies
$$
\mathfrak{d}(\bv_2, q) = 0 \ \text{ for all } \ q \in Q.
$$
The above expression directly implies that $\bv_2 \in \V_{\sigma}$. We now select $\bv_2 \in \V_{\sigma}$ as the test function in \eqref{eqn-hemi-2} and obtain
\begin{align}
		\mu \mathfrak{a}(\bu,\bv-\bv_1)+\mathfrak{b}(\bu,\bu,\bv-\bv_1)+\alpha \mathfrak{a}_0(\bu,\bv-\bv_1)&+\beta \mathfrak{c}(\bu,\bv-\bv_1)+\kappa \mathfrak{c}_0(\bu,\bv-\bv_1)\nonumber\\+\int_{\Gamma_1}j^0(\bu_{\tau};\bv_{\tau}-\bv_{1,\tau})\d S&\geq 	\langle\f,\bv-\bv_1\rangle.
\end{align}
It is important to observe that $\bv_1=0$ on $\Gamma_1$. Consequently, using inequalities \eqref{eqn-pressure-1} and \eqref{eqn-pressure-2} in succession, we obtain from the preceding inequality that
\begin{align}
	&	\mu \mathfrak{a}(\bu,\bv)+\mathfrak{b}(\bu,\bu,\bv)+\alpha \mathfrak{a}_0(\bu,\bv)+\beta \mathfrak{c}(\bu,\bv)+\kappa \mathfrak{c}_0(\bu,\bv)+\int_{\Gamma_1}j^0(\bu_{\tau};\bv_{\tau})\d S\nonumber\\&\geq 	\langle\f,\bv-\bv_1\rangle+	\mu \mathfrak{a}(\bu,\bv_1)+\mathfrak{b}(\bu,\bu,\bv_1)+\alpha \mathfrak{a}_0(\bu,\bv_1)+\beta \mathfrak{c}(\bu,\bv_1)+\kappa \mathfrak{c}_0(\bu,\bv_1)\nonumber\\&=-\mathfrak{d}(\bv_1,p)+\langle\f,\bv\rangle\nonumber\\&=-\mathfrak{d}(\bv,p)+\langle\f,\bv\rangle\ \text{ for all }\ \bv\in\V,
\end{align}
that is, the  first inequality in \eqref{eqn-hemi-1} holds. It can be seen that the second equation in \eqref{eqn-hemi-1} holds as a direct consequence of the fact that $\bu\in\V_{\sigma}$. 

 We now derive an estimate for $\|p\|_Q$. Using the inf-sup condition \eqref{eqn-inf-sup} together with \eqref{eqn-pressure-1}, we obtain
\begin{align*}
	\vartheta\|p\|_{Q}\leq 	\sup_{\bv\in\V_0}\frac{1}{\|\bv\|_{\V}}[\langle\f,\bv\rangle -\mu \mathfrak{a}(\bu,\bv)-\mathfrak{b}(\bu,\bu,\bv)-\alpha \mathfrak{a}_0(\bu,\bv)-\beta \mathfrak{c}(\bu,\bv)-\kappa \mathfrak{c}_0(\bu,\bv)]. 
\end{align*}
Using \eqref{eqn-a-est-1}, \eqref{eqn-b-est-1}, \eqref{eqn-c-est-1}, and \eqref{eqn-bound} in the above expression, we deduce 
\begin{align}
	\vartheta\|p\|_{Q}&\leq\|\f\|_{\V^*}+(2\mu+\alpha C_k^2)\|\bu\|_{\V}+C_b\|\bu\|_{\V}^2+C\beta \|\bu\|_{\L^{r+1}}^r+C_s|\kappa||\mathcal{O}|^{\frac{r-q}{(q+1)(r+1)}}\|\bu\|_{\L^{q+1}}^q\nonumber\\&\leq  \|\f\|_{\V^*}+(2\mu+\alpha C_k^2)\widetilde{K}^{1/2}+C_b\widetilde{K}+C_s\beta\widetilde{K}^{\frac{r}{r+1}}+C_s|\kappa||\mathcal{O}|^{\frac{r-q}{r+1}}\widetilde{K}^{\frac{q}{r+1}},
\end{align}
where we have used the fact that for $1\leq r\leq\frac{2d}{(d-2)^+}$, 
\begin{align}
\|\bv\|_{\L^{r+1}}\leq C_s\|\bv\|_{\V}\ \text{ for all }\ \bv\in\V,
\end{align}
by an application of Sobolev's inequality. 

Finally, we establish the uniqueness of the solution $(\bu, p)$ to Problem \ref{prob-hemi-1}. Since the velocity component $\bu$ satisfies equation \eqref{eqn-hemi-2}, and Problem \ref{prob-hemi-2} admits a unique solution, it follows that $\bu$ is uniquely determined. Suppose $(\bu, p_1)$ and $(\bu, p_2)$ are two solutions to Problem \ref{prob-hemi-1} corresponding to the same forcing term $\f$. Then, using \eqref{eqn-pressure-1}, we deduce that
\begin{align*}
	\mathfrak{d}(\bv,p_1-p_2)=	0 \ \text{ for all }\ \bv\in\V_0. 
\end{align*}
Applying the inf-sup condition \eqref{eqn-inf-sup} with $q = p_1 - p_2$, we conclude that $\|p_1 - p_2\|_Q = 0$, and hence $p_1 = p_2$ in $Q$.

We now show that the pressure component $p$ depends locally Lipschitz continuously on the external force $\f \in \V^*$. Suppose $(\bu_1, p_1)$ and $(\bu_2, p_2)$ are two solutions to Problem \ref{prob-hemi-1} corresponding to the forces $\f_1$ and $\f_2$, respectively. Then, by equation \eqref{eqn-pressure-1}, we have for all $\bv \in \V_0$,
\begin{align*}
		\mu \mathfrak{a}(\bu_1,\bv)+\mathfrak{b}(\bu_1,\bu_1,\bv)+\alpha \mathfrak{a}_0(\bu_1,\bv)+\beta \mathfrak{c}(\bu_1,\bv)+\kappa \mathfrak{c}_0(\bu_1,\bv)+\mathfrak{d}(\bv,p_1)&=	\langle\f_1,\bv\rangle,\\
			\mu \mathfrak{a}(\bu_2,\bv)+\mathfrak{b}(\bu_2,\bu_2,\bv)+\alpha \mathfrak{a}_0(\bu_2,\bv)+\beta \mathfrak{c}(\bu_2,\bv)+\kappa \mathfrak{c}_0(\bu_2,\bv)+\mathfrak{d}(\bv,p_2)&=	\langle\f_2,\bv\rangle.
\end{align*}
By using  \eqref{eqn-a-est-1} and \eqref{eqn-b-est-1} in the above expression, we deduce 
\begin{align}\label{eqn-pressure-3}
&	\mathfrak{d}(\bv,p_1-p_2)\nonumber\\&=	\langle\f_1-\f_2,\bv\rangle-\mu \mathfrak{a}(\bu_1-\bu_2,\bv) -\mathfrak{b}(\bu_1-\bu_2,\bu_1,\bv)-\mathfrak{b}(\bu_2,\bu_1-\bu_2,\bv)\nonumber\\&\quad-\alpha \mathfrak{a}(\bu_1-\bu_2,\bv)-\beta[\mathfrak{c}(\bu_1,\bv)-\mathfrak{c}(\bu_2,\bv)]-\kappa[\mathfrak{c}_0(\bu_1,\bv)-\mathfrak{c}_0(\bu_2,\bv)]\nonumber\\&\leq \big[\|\f_1-\f_2\|_{\V^*}+(2\mu+\alpha C_k^2)\|\bu_1-\bu_2\|_{\V}+C_b\|\bu_1-\bu_2\|_{\V}(\|\bu_1\|_{\V}+\|\bu_2\|_{\V})\big]\|\bv\|_{\V}\nonumber\\&\quad-\beta[\mathfrak{c}(\bu_1,\bv)-\mathfrak{c}(\bu_2,\bv)]-\kappa[\mathfrak{c}_0(\bu_1,\bv)-\mathfrak{c}_0(\bu_2,\bv)].
\end{align}
An application of  Taylor's formula and the bound \eqref{eqn-bound} yield
\begin{align}\label{eqn-pressure-4}
	|\mathfrak{c}(\bu_1,\bv)-\mathfrak{c}(\bu_2,\bv)|&=\bigg|\bigg\langle\int_0^1\mathcal{C}'(\theta\bu_1+(1-\theta)\bu_2)(\bu_1-\bu_2)\d\theta,\bv\bigg\rangle\bigg|\nonumber\\&
	\leq r\left(\|\bu_1\|_{\L^{r+1}}+\|\bu_2\|_{\L^{r+1}}\right)^{r-1}\|\bu_1-\bu_2\|_{\L^{r+1}}\|\bv\|_{\L^{r+1}}
	\nonumber\\&\leq C_s^2r2^{r-1}\widetilde{K}^{\frac{r-1}{r+1}}\|\bu_1-\bu_2\|_{\V}\|\bv\|_{\V}. 
\end{align}
A similar calculation leads to 
\begin{align}\label{eqn-pressure-5}
		|\mathfrak{c}_0(\bu_1,\bv)-\mathfrak{c}_0(\bu_2,\bv)|&\leq C_s^2|\mathcal{O}|^{\frac{r-q}{r+1}}q2^{q-1}\widetilde{K}^{\frac{q-1}{r+1}}\|\bu_1-\bu_2\|_{\V}\|\bv\|_{\V}. 
\end{align}
Let us now consider the case $d \in \{2, 3\}$ with $r > 3$. Substituting equations \eqref{eqn-pressure-3} and \eqref{eqn-pressure-4} into \eqref{eqn-pressure-5}, we obtain
\begin{align}\label{eqn-pressure-6}
		&\mathfrak{d}(\bv,p_1-p_2)\nonumber\\&\leq  \big[\|\f_1-\f_2\|_{\V^*}+\big(2\mu+\alpha C_k^2+2C_b\widetilde{\kappa}+C_s^2r2^{r-1}\widetilde{K}^{\frac{r-1}{r+1}}+C_s^2|\mathcal{O}|^{\frac{r-q}{r+1}}q2^{q-1}\widetilde{K}^{\frac{q-1}{r+1}}\big)\nonumber\\&\qquad\times\|\bu_1-\bu_2\|_{\V}\big]\|\bv\|_{\V}\nonumber\\&\leq \bigg[1+\frac{2}{(2\mu-\delta_1\lambda_0^{-1})}\big(2\mu+\alpha C_k^2+2C_b\widetilde{\kappa}+C_s^2r2^{r-1}\widetilde{K}^{\frac{r-1}{r+1}}+C_s^2|\mathcal{O}|^{\frac{r-q}{r+1}}q2^{q-1}\widetilde{K}^{\frac{q-1}{r+1}}\big)\bigg]\nonumber\\&\quad\times\|\f_1-\f_2\|_{\V^*}\|\bv\|_{\V},
\end{align}
where we have used \eqref{eqn-differece-4} also. Now the inf-sup condition given in  \eqref{eqn-inf-sup} leads to 
\begin{align}
	&\|p-p_0\|_Q\nonumber\\&\leq \sup_{\bv\in\V_0}\frac{\mathfrak{d}(\bv,p-p_0)}{\|\bv\|_{\V}}\nonumber\\&\leq \bigg[1+\frac{2}{(2\mu-\delta_1\lambda_0^{-1})}\big(2\mu+2C_b\widetilde{\kappa}+C_s^2r2^{r-1}\widetilde{K}^{\frac{r-1}{r+1}}+C_s^2|\mathcal{O}|^{\frac{r-q}{r+1}}q2^{q-1}\widetilde{K}^{\frac{q-1}{r+1}}\big)\bigg]\|\f_1-\f_2\|_{\V^*}.
\end{align}
Therefore, the pressure component $p\in Q$ depends locally Lipschitz continuously on the data $\f\in\V^*$. The case of $d\in\{2,3\}$ with $r\in[1,3]$ can be treated in a similar way. 
\end{proof}

	\section{Mixed finite element method for CBF  hemivariational inequality}\label{sec4}\setcounter{equation}{0}
Under the assumptions of Theorem \ref{thm-unique}, which ensure the uniqueness of solutions to Problems \ref{prob-hemi-1} and \ref{prob-hemi-2}, we now proceed to apply the mixed finite element method to numerically solve Problem \ref{prob-hemi-1}.

For simplicity, let us assume that $\mathcal{O}$ is a polygonal (in 2D) or polyhedral (in 3D) domain, and consider ${ \mathcal{T}^h }$ to be a regular family of finite element meshes of $\overline{\mathcal{O}}$. Associated with each mesh $\mathcal{T}^h$, we define finite element spaces $\V^h$ and $Q^h$ that approximate the continuous spaces $\V$ and $Q$, respectively. Let $\V^{h}_{0} := \V^h \cap \H_0^1(\mathcal{O})$. We further assume that the discrete inf-sup (or LBB) condition holds, that is,  there exists a constant $\vartheta > 0$, independent of the mesh parameter $h > 0$, such that for all $q^h \in Q^h$, the following condition is satisfied uniformly:
	\begin{align}\label{eqn-inf-sup-discrete}
		\vartheta\|q^h\|_{\mathrm{L}^2}\leq 	\sup_{\bv^h\in\V_0^h}\frac{\mathfrak{d}(\bv^h,q^h)}{\|\bv^h\|_{\V}} \ \text{ for all }\ q^h\in Q^h.
	\end{align}
	The finite element formulation corresponding to Problem \ref{prob-hemi-1} is given as follows:
		\begin{problem}\label{prob-hemi-3}
		Find $\bu^h\in\V^h$ and $p^h\in Q^h$  such that
		\begin{equation}\label{eqn-discrete-hemi}
			\left\{
			\begin{aligned}
				\mu \mathfrak{a}(\bu^h,\bv^h)+\mathfrak{b}(\bu^h,\bu^h,\bv^h)+\alpha \mathfrak{a}_0(\bu^h,\bv^h)&+\beta \mathfrak{c}(\bu^h,\bv^h)+\kappa \mathfrak{c}_0(\bu^h,\bv^h)+\mathfrak{d}(\bv^h,p^h)\\+\int_{\Gamma_1}j^0(\bu_{\tau}^h;\bv_{\tau}^h)\d S&\geq 	(\f,\bv^h)_{\L^2(\mathcal{O})} \ \text{ for all }\ \bv^h\in \V^h,\\
				\mathfrak{d}(\bu^h,q^h)&=0\ \text{ for all }\ q^h\in Q^h. 
			\end{aligned}\right.
		\end{equation}
	\end{problem}
	Analogous to the continuous setting, we introduce a subspace of $\V^h$ defined as follows:
	\begin{align}
		\V^h_{\sigma}:=\left\{\bv^h\in\V^h: \mathfrak{d}(\bv^h,q^h)=0\ \text{ for all }\ q^h\in Q^h\right\},
	\end{align}
	and present a simplified version of Problem \ref{prob-hemi-3}.
		\begin{problem}\label{prob-hemi-4}
		Find $\bu^h\in\V_{\sigma}^h$ such that 
		\begin{equation}\label{eqn-hemi-4}
			\left\{
			\begin{aligned}
				\mu \mathfrak{a}(\bu^h,\bv^h)+\mathfrak{b}(\bu^h,\bu^h,\bv^h)+\alpha \mathfrak{a}_0(\bu^h,\bv^h)&+\beta \mathfrak{c}(\bu^h,\bv^h)+\kappa \mathfrak{c}_0(\bu^h,\bv^h)\\+\int_{\Gamma_1}j^0(\bu_{\tau}^h;\bv_{\tau}^h)\d S&\geq 	\langle\f,\bv^h\rangle \ \text{ for all }\ \bv^h\in \V^h_{\sigma}.
			\end{aligned}\right.
		\end{equation}
	\end{problem}
	Under the assumptions of Theorem \ref{thm-unique} and the discrete inf-sup condition given in \eqref{eqn-inf-sup-discrete}, it can be shown similarly to Theorems \ref{thm-unique} and \ref{thm-equivalent} that Problems \ref{prob-hemi-3} and \ref{prob-hemi-4} are equivalent and each possesses a unique solution. Additionally, the discrete counterpart of Proposition \ref{prop-energy-est} is given by
	\begin{align}\label{eqn-discrete-bound}
	\| \bu^h \|_{\V} ^2+\|\bu^h\|_{\L^{r+1}}^{r+1}\leq \widetilde{K}, 
	\end{align}
		where $\widetilde{K}$ is defined in \eqref{eqn-bound}. 
		
	Let us now turn our attention to estimating the error.
	We aim to derive the following C\'ea-type inequality, which provides an error bound for the finite element approximation $(\bu^h, p^h)$.

	\begin{theorem}\label{thm-cea}
		Under the assumptions of Theorem \ref{thm-unique}, we have 
		\begin{align}\label{eqn-cea}
			\|\bu-\bu^h\|_{\V}+\|p-p^h\|_Q&\leq C\inf_{\bv^h\in\V^h}\left(\|\bu-\bv^h\|_{\V}+\|\bu_{\tau}-\bv^h_{\tau}\|_{\L^2(\Gamma_1)}^{1/2}\right)\nonumber\\&\quad+C\inf_{q^h\in Q^h}\|p-q^h\|_Q.
		\end{align}
	\end{theorem}
	\begin{proof}
	Referring back to equation \eqref{eqn-pressure-1}, we have   for all  $\bv\in\V_0$
	\begin{align}\label{eqn-equality}
		\mu \mathfrak{a}(\bu,\bv)+\mathfrak{b}(\bu,\bu,\bv)+\alpha \mathfrak{a}_0(\bu,\bv)+\beta \mathfrak{c}(\bu,\bv)+\kappa \mathfrak{c}_0(\bu,\bv)+\mathfrak{d}(\bv,p)=\langle\f,\bv\rangle.
	\end{align}
	The discrete counterpart of equation \eqref{eqn-equality}  can be derived from  \eqref{eqn-discrete-hemi}  for all $ \bv^h\in\V_0^h$: 
	\begin{align}\label{eqn-equality-1}
		\mu \mathfrak{a}(\bu^h,\bv^h)+\mathfrak{b}(\bu^h,\bu^h,\bv^h)+\alpha \mathfrak{a}_0(\bu^h,\bv^h)+\beta \mathfrak{c}(\bu^h,\bv^h)+\kappa \mathfrak{c}_0(\bu^h,\bv^h)+\mathfrak{d}(\bv^h,p^h)=\langle\f,\bv^h\rangle. 
	\end{align} 
	Let $q^h \in Q^h$ be arbitrary. Then, we can write using triangle inequality as 
	\begin{align}\label{eqn-triangle}
		\|p-p^h\|_{Q}\leq\|p-q^h\|_{Q}+\|q^h-p^h\|_{Q}. 
	\end{align}
	Using the discrete inf-sup condition given in \eqref{eqn-inf-sup-discrete}, we find  
	\begin{align}\label{eqn-inf-sup-discrete-1}
		\vartheta\|p^h-q^h\|_Q\leq \sup_{\bv^h\in\V^h_0}\frac{\mathfrak{d}(\bv^h,p^h-q^h)}{\|\bv^h\|_{\V}}. 
	\end{align}
	Using the linearity of the second variable in $\mathfrak{d}(\cdot,\cdot)$, we get 
	\begin{align}\label{eqn-difference-discrte-1}
		\mathfrak{d}(\bv^h,p^h-q^h)=\mathfrak{d}(\bv^h,p^h-p)+\mathfrak{d}(\bv^h,p-q^h),
		\end{align}
		and 
		\begin{align*}
			\mathfrak{d}(\bv^h,p^h-p)=\mathfrak{d}(\bv^h,p^h)-\mathfrak{d}(\bv^h,p).
		\end{align*}
	From   \eqref{eqn-equality} and \eqref{eqn-equality-1}, we further have 
	\begin{align*}
		\mathfrak{d}(\bv^h,p^h)&=\langle\f,\bv^h\rangle-\mu \mathfrak{a}(\bu^h,\bv^h)-\mathfrak{b}(\bu^h,\bu^h,\bv^h)-\alpha \mathfrak{a}_0(\bu^h,\bv^h)-\beta \mathfrak{c}(\bu^h,\bv^h)-\kappa \mathfrak{c}_0(\bu^h,\bv^h),\\
		\mathfrak{d}(\bv^h,p)&=\langle\f,\bv^h\rangle-\mu \mathfrak{a}(\bu,\bv^h)-\mathfrak{b}(\bu,\bu,\bv^h)-\alpha \mathfrak{a}_0(\bu,\bv^h)-\beta \mathfrak{c}(\bu,\bv^h)-\kappa \mathfrak{c}_0(\bu,\bv^h),
	\end{align*}
	so that 
	\begin{align*}
			\mathfrak{d}(\bv^h,p^h-p)&=\mu \mathfrak{a}(\bu-\bu^h,\bv^h)+\mathfrak{b}(\bu,\bu,\bu^h)-\mathfrak{b}(\bu^h,\bu^h,\bv^h)+\alpha \mathfrak{a}_0(\bu-\bu^h,\bv^h)\nonumber\\&\quad+\beta[\mathfrak{c}(\bu,\bv^h)-\mathfrak{c}(\bu^h,\bu)]+\kappa[\mathfrak{c}_0(\bu,\bv^h)-\mathfrak{c}_0(\bu^h,\bu)]\nonumber\\&=\mu \mathfrak{a}(\bu-\bu^h,\bv^h)+\mathfrak{b}(\bu,\bu-\bu^h,\bv^h)+\mathfrak{b}(\bu-\bu^h,\bu^h,\bv^h)+\alpha \mathfrak{a}_0(\bu-\bu^h,\bv^h)\nonumber\\&\quad+\beta\bigg<\int_0^1\mathcal{C}^{\prime}(\theta\bu+(1-\theta)\bu^h)\d\theta(\bu-\bu^h),\bv^h\bigg>\nonumber\\&\quad+\kappa\bigg<\int_0^1\mathcal{C}_0^{\prime}(\theta\bu+(1-\theta)\bu^h)\d\theta(\bu-\bu^h),\bv^h\bigg>.
	\end{align*}
	Using \eqref{eqn-inf-sup-discrete-1}, \eqref{eqn-difference-discrte-1} and a calculation similar to \eqref{eqn-pressure-6} yields 
	\begin{align}
		\vartheta\|p^h-q^h\|_Q&\leq \big(2\mu+\alpha C_k^2+C_\mathfrak{b}(\|\bu\|_{\V}+\|\bu^h\|_{\V})+C_s^2r(\|\bu\|_{\L^{r+1}}+\|\bu^h\|_{\L^{r+1}})^{r-1}\nonumber\\&\quad+C_s^2q(\|\bu\|_{\L^{q+1}}+\|\bu^h\|_{\L^{q+1}})^{q-1} \big)\|\bu-\bu^h\|_{\V}+C_k\|p-q^h\|_Q.
	\end{align}
	Since have the bounds given in \eqref{eqn-bound} and  \eqref{eqn-discrete-bound}, we immediately have 
	\begin{align}
		\vartheta\|p^h-q^h\|_Q&\leq \big(2\mu+\alpha C_k^2+2C_b\widetilde{\kappa}+C_s^2r2^{r-1}\widetilde{K}^{\frac{r-1}{r+1}}+C_s^2|\mathcal{O}|^{\frac{r-q}{r+1}}q2^{q-1}\widetilde{K}^{\frac{q-1}{r+1}}\big)\|\bu-\bu^h\|_{\V}\nonumber\\&\quad+C_k\|p-q^h\|_Q\nonumber\\&\leq C(\|\bu-\bu^h\|_{\V}+\|p-q^h\|_Q). 
	\end{align}
From \eqref{eqn-triangle}, we also infer 
\begin{align}\label{eqn-traingle-1}
		\|p-p^h\|_{Q}\leq C(\|\bu-\bu^h\|_{\V}+\|p-q^h\|_Q). 
\end{align}

On the other hand, for any $\bv^h\in\V^h$,  using \eqref{eqn-a-est-1}, we have
\begin{align}\label{eqn-error-1}
&	2\mu\|\bu-\bu^h\|_{\V}^2+\alpha\|\bu-\bu^h\|_{\H}^2\nonumber\\& \quad =\mu \mathfrak{a}(\bu-\bu^h,\bu-\bu^h)+\alpha \mathfrak{a}_0(\bu-\bu^h,\bu-\bu^h)
	\nonumber\\&\quad =\mu \mathfrak{a}(\bu,\bu-\bu^h)-\mu \mathfrak{a}(\bu^h,\bu-\bu^h)+\alpha \mathfrak{a}_0(\bu,\bu-\bu^h)-\alpha \mathfrak{a}_0(\bu^h,\bu-\bu^h)
		\nonumber\\&\quad =\mu \mathfrak{a}(\bu,\bu-\bu^h)-\mu \mathfrak{a}(\bu^h,\bu-\bv^h)+\mu \mathfrak{a}(\bu^h,\bu^h-\bv^h)\nonumber\\&\quad \quad +\alpha \mathfrak{a}_0(\bu,\bu-\bu^h)-\alpha \mathfrak{a}_0(\bu^h,\bu-\bv^h)+\alpha \mathfrak{a}_0(\bu^h,\bu^h-\bv^h). 
\end{align}
By taking $\bv=\bu^h-\bu$ in \eqref{eqn-hemi-1}, we obtain 
\begin{align}\label{eqn-error-2}
		&\mu \mathfrak{a}(\bu,\bu-\bu^h)+\alpha \mathfrak{a}_0(\bu,\bu-\bu^h)\nonumber\\& \quad \leq \mathfrak{b}(\bu,\bu,\bu^h-\bu)+\beta \mathfrak{c}(\bu,\bu^h-\bu)+\kappa \mathfrak{c}_0(\bu,\bu^h-\bu)\nonumber\\&\quad \quad+\mathfrak{d}(\bu^h-\bu,p)+\int_{\Gamma_1}j^0(\bu_{\tau};\bu^h_{\tau}-\bu_{\tau})\d S-	\langle\f,\bu^h-\bu\rangle. 
\end{align}
By substituting $\bv^h$ with $\bv^h - \bu^h$ in \eqref{eqn-discrete-hemi}, we get 
\begin{align}\label{eqn-error-3}
	&	\mu \mathfrak{a}(\bu^h,\bu^h-\bv^h)+\alpha \mathfrak{a}_0(\bu^h,\bu^h-\bv^h)\nonumber\\& \quad \leq \mathfrak{b}(\bu^h,\bu^h,\bv^h-\bu^h)+\beta \mathfrak{c}(\bu^h,\bv^h-\bu^h)+\kappa \mathfrak{c}_0(\bu^h,\bv^h-\bu^h)\nonumber\\&\quad \quad+\mathfrak{d}(\bv^h-\bu^h,p^h)+\int_{\Gamma_1}j^0(\bu_{\tau}^h;\bv^h_{\tau}-\bu^h_{\tau})\d S- 	\langle\f,\bv^h-\bu^h\rangle. 
	\end{align}
	Moreover, the bilinearity  of $\mathfrak{a}(\cdot,\cdot)$ gives 
	\begin{align}\label{eqn-error-4}
		-&\mu \mathfrak{a}(\bu^h,\bu-\bv^h)-\alpha \mathfrak{a}_0(\bu^h,\bu-\bv^h)\nonumber\\&=\mu \mathfrak{a}(\bu-\bu^h,\bu-\bv^h)+\alpha \mathfrak{a}_0(\bu-\bu^h,\bu-\bv^h)+\mu \mathfrak{a}(\bu,\bv^h-\bu)+\alpha \mathfrak{a}_0(\bu,\bv^h-\bu). 
	\end{align}
Let us take $\bv=\bu-\bv^h$ in \eqref{eqn-hemi-1} to find 
\begin{align}\label{eqn-error-5}
		&\mu \mathfrak{a}(\bu,\bv^h-\bu)+\alpha \mathfrak{a}_0(\bu,\bv^h-\bu)\nonumber\\&\leq \mathfrak{b}(\bu,\bu,\bu-\bv^h)+\beta \mathfrak{c}(\bu,\bu-\bv^h)+\kappa \mathfrak{c}_0(\bu,\bu-\bv^h)\nonumber\\&\quad+\mathfrak{d}(\bu-\bv^h,p)+\int_{\Gamma_1}j^0(\bu_{\tau};\bu_{\tau}-\bv^h_{\tau})\d S-\langle\f,\bu-\bv^h\rangle. 
\end{align}
Using \eqref{eqn-error-2}-\eqref{eqn-error-5} in \eqref{eqn-error-1}, we deduce 
\begin{align}\label{eqn-error-main}
	&2\mu\|\bu-\bu^h\|_{\V}^2+\alpha \|\bu-\bu^h\|_{\H}^2\nonumber\\&\leq \mu \mathfrak{a}(\bu-\bu^h,\bu-\bv^h)+\alpha \mathfrak{a}_0(\bu-\bu^h,\bu-\bv^h)+K_b+K_c+K_{c_0}+K_d+K_j,
\end{align}
where 
\begin{align*}
	K_b&=\mathfrak{b}(\bu,\bu,\bu^h-\bu)+ \mathfrak{b}(\bu^h,\bu^h,\bv^h-\bu^h)+\mathfrak{b}(\bu,\bu,\bu-\bv^h),\\
	K_c&=\beta[\mathfrak{c}(\bu,\bu^h-\bu)+\mathfrak{c}(\bu^h,\bv^h-\bu^h)+\mathfrak{c}(\bu,\bu-\bv^h)],\\
	K_{c_0}&=\kappa[\mathfrak{c}_0(\bu,\bu^h-\bu)+\mathfrak{c}_0(\bu^h,\bv^h-\bu^h)+\mathfrak{c}_0(\bu,\bu-\bv^h)],\\
	K_d&=\mathfrak{d}(\bu^h-\bu,p)+\mathfrak{d}(\bv^h-\bu^h,p^h)+\mathfrak{d}(\bu-\bv^h,p),\\
	K_j&=\int_{\Gamma_1}\Big[j^0(\bu_{\tau};\bu^h_{\tau}-\bu_{\tau})+j^0(\bu_{\tau}^h;\bv^h_{\tau}-\bu^h_{\tau})+j^0(\bu_{\tau};\bu_{\tau}-\bv^h_{\tau})\Big]\d S. 
\end{align*}
We rewrite $K_b$ as 
\begin{align}\label{eqn-error-b0}
	K_b&=\mathfrak{b}(\bu,\bu,\bu^h-\bv^h)+ \mathfrak{b}(\bu^h,\bu^h,\bv^h-\bu^h)\nonumber\\
	     &=-[\mathfrak{b}(\bu,\bu,\bu-\bu^h)-\mathfrak{b}(\bu^h,\bu^h,\bu-\bu^h)]+[\mathfrak{b}(\bu,\bu,\bu-\bv^h)-\mathfrak{b}(\bu^h,\bu^h,\bu-\bv^h)]\nonumber\\
	     &=-\mathfrak{b}(\bu-\bu^h,\bu^h,\bu-\bu^h)+\mathfrak{b}(\bu-\bu^h,\bu,\bu-\bv^h)+\mathfrak{b}(\bu^h,\bu-\bu^h,\bu-\bv^h),
\end{align}
where in the final step, we have used the fact that $\mathfrak{b}(\bu,\bu-\bu^h,\bu^h-\bu)=0$. Rest of the proof is divided into the following two different cases. 

\textbf{Case 1.} \emph{$d\in\{2,3\}$ with $r\in(3,\infty)$.}

We employ similar calculations as in \eqref{eqn-differece-2} to estimate the  term 
$-\mathfrak{b}(\bu-\bu^h,\bu^h,\bu-\bu^h)$ as follows:
\begin{align}\label{eqn-error-b1}
&	|\mathfrak{b}(\bu-\bu^h,\bu^h,\bu-\bu^h)|\nonumber\\&\leq 
\frac{(2\mu-\delta_1\lambda_0^{-1})}{2}\|\bu-\bu^h\|_{\V}^2+\frac{\beta}{4}\||\bu^h|^{\frac{r-1}{2}}(\bu-\bu^h)\|_{\H}^2+\widetilde{\varrho}_{1,r}\|\bu-\bu^h\|_{\H}^2.
\end{align}
Using the estimates \eqref{eqn-b-est-1}, \eqref{eqn-bound} and \eqref{eqn-discrete-bound}, and H\"older's inequality,  we estimate $\mathfrak{b}(\bu-\bu^h,\bu,\bu-\bv^h)+\mathfrak{b}(\bu^h,\bu-\bu^h,\bu-\bv^h)$ as 
\begin{align}\label{eqn-error-b2}
&	|\mathfrak{b}(\bu-\bu^h,\bu,\bu-\bv^h)+\mathfrak{b}(\bu^h,\bu-\bu^h,\bu-\bv^h)|
\nonumber\\&\leq  C_b\|\bu\|_{\V}\|\bu-\bu^h\|_{\V}\|\bu^h-\bv^h\|_{\V}+C_b\|\bu^h\|_{\V}\|\bu-\bu^h\|_{\V}\|\bu^h-\bv^h\|_{\V}\nonumber\\&\leq 2C_b\widetilde{K}\|\bu-\bu^h\|_{\V}\|\bu-\bv^h\|_{\V}\nonumber\\&\leq  \frac{(2\mu-\delta_1\lambda_0^{-1})}{8}\|\bu-\bu^h\|_{\V}^2+\frac{8C_b^2\widetilde{K}^2}{(2\mu-\delta_1\lambda_0^{-1})}\|\bu-\bv^h\|_{\V}^2. 
	\end{align}
	Using the estimates \eqref{eqn-error-b1} and \eqref{eqn-error-b2} in \eqref{eqn-error-b0}, we estimate $K_b$ as 
	\begin{align}\label{eqn-error-b}
		K_b&\leq \frac{5(2\mu-\delta_1\lambda_0^{-1})}{8}\|\bu-\bu^h\|_{\V}^2+\frac{\beta}{4}\||\bu^h|^{\frac{r-1}{2}}(\bu-\bu^h)\|_{\H}^2+\widetilde{\varrho}_{1,r}\|\bu-\bu^h\|_{\H}^2\nonumber\\&\quad+\frac{8C_b^2\widetilde{K}^2}{(2\mu-\delta_1\lambda_0^{-1})}\|\bu-\bv^h\|_{\V}^2. 
	\end{align}

Let us now rewrite $K_c$ as
\begin{align}\label{eqn-error}
	K_c&=\beta[\mathfrak{c}(\bu,\bu^h-\bv^h)+\mathfrak{c}(\bu^h,\bv^h-\bu^h)]\nonumber\\&=- \beta[\mathfrak{c}(\bu,\bu-\bu^h)-\mathfrak{c}(\bu^h,\bu-\bu^h)]
	+ \beta[\mathfrak{c}(\bu,\bu-\bv^h)-\mathfrak{c}(\bu^h,\bu-\bv^h)].
\end{align}
Applying  \eqref{2.23}, we infer 
\begin{align}\label{eqn-error-0}
	&\beta [\mathfrak{c}(\bu,\bu-\bu^h)-\mathfrak{c}(\bu^h,\bu-\bu^h)]\geq \frac{\beta}{2}\||\bu|^{\frac{r-1}{2}}(\bu-\bu^h)\|_{\H}^2+\frac{\beta}{2}\||\bu^h|^{\frac{r-1}{2}}(\bu-\bu^h)\|_{\H}^2.
\end{align}
We estimate $ \beta[\mathfrak{c}(\bu,\bu-\bv^h)-\mathfrak{c}(\bu^h,\bu-\bv^h)]$ using Taylor's formula and H\"older's inequality as 
\begin{align}\label{eqn-errror-1}
	 &\beta|\mathfrak{c}(\bu,\bu-\bv^h)-\mathfrak{c}(\bu^h,\bu-\bv^h)|\nonumber\\&=\beta\left|\left\langle\int_0^1\mathcal{C}^{\prime}(\theta\bu+(1-\theta)\bu^h)(\bu-\bu^h)\d\theta,(\bu-\bv^h)\right\rangle\right|\nonumber\\&\leq\beta r\left\langle\int_0^1|\theta\bu+(1-\theta)\bu^h|^{r-1}|\bu-\bu^h|\d\theta,|\bu-\bv^h| \right\rangle\nonumber\\&\leq \beta r\left(\|\bu\|_{\L^{r+1}}+\|\bu^h\|_{\L^{r+1}}\right)^{r-1}\|\bu-\bu^h\|_{\L^{r+1}}\|\bu-\bv^h\|_{\L^{r+1}}\nonumber\\&\leq C_s^2 2^{r-1}r\beta\widetilde{K}^{\frac{r-1}{r+1}}\|\bu-\bu^h\|_{\V}\|\bu-\bv^h\|_{\V}\nonumber\\&\leq \frac{(2\mu-\delta_1\lambda_0^{-1})}{16}\|\bu-\bu^h\|_{\V}^2+\frac{4(C_s^2 2^{r-1}r\beta\widetilde{K}^{\frac{r-1}{r+1}})^2}{(2\mu-\delta_1\lambda_0^{-1})}\|\bu-\bv^h\|_{\V}^2,
	\end{align}where we have used \eqref{eqn-bound} and \eqref{eqn-discrete-bound} also. 
	Substituting  \eqref{eqn-error-0} and \eqref{eqn-error-1} in \eqref{eqn-error}, we achieve 
	\begin{align}\label{eqn-error-6}
			K_c&\leq -\frac{\beta}{2}\||\bu|^{\frac{r-1}{2}}(\bu-\bu^h)\|_{\H}^2-\frac{\beta}{2}\||\bu^h|^{\frac{r-1}{2}}(\bu-\bu^h)\|_{\H}^2\nonumber\\&\quad +\frac{(2\mu-\delta_1\lambda_0^{-1})}{16}\|\bu-\bu^h\|_{\V}^2+\frac{4(C_s^2 2^{r-1}r\beta\widetilde{K}^{\frac{r-1}{r+1}})^2}{(2\mu-\delta_1\lambda_0^{-1})}\|\bu-\bv^h\|_{\V}^2.
	\end{align}
	
We rewrite $K_{c_0}$ as
	\begin{align}\label{eqn-error-7}
		K_{c_0}&=\kappa[\mathfrak{c}_0(\bu,\bu^h-\bv^h)+\mathfrak{c}_0(\bu^h,\bv^h-\bu^h)]\nonumber\\&=- \kappa[\mathfrak{c}_0(\bu,\bu-\bu^h)-\mathfrak{c}_0(\bu^h,\bu-\bu^h)]
		+ \kappa[\mathfrak{c}_0(\bu,\bu-\bv^h)-\mathfrak{c}_0(\bu^h,\bu-\bv^h)].
	\end{align}
Following similar calculations as in \eqref{eqn-est-c0-4}, 	we estimate $\kappa[\mathfrak{c}_0(\bu,\bu-\bu^h)-\mathfrak{c}_0(\bu^h,\bu-\bu^h)]$  as
	\begin{align*}
|\kappa||[\mathfrak{c}_0(\bu,\bu-\bu^h)-\mathfrak{c}_0(\bu^h,\bu-\bu^h)]|&\leq\frac{\beta}{2}\||\bu|^{\frac{r-1}{2}}(\bu-\bu^h)\|_{\H}^2+\frac{\beta}{4}\||\bu^h|^{\frac{r-1}{2}}(\bu-\bu^h)\|_{\H}^2\nonumber\\&\quad+(\varrho_{2,r}+\varrho_{3,r})\|\bu-\bu^h\|_{\H}^2. 
	\end{align*}
	A calculation similar to \eqref{eqn-errror-1} yields 
	\begin{align*}
	&	|\kappa||\mathfrak{c}_0(\bu,\bu-\bv^h)-\mathfrak{c}_0(\bu^h,\bu-\bv^h)|\nonumber\\&\leq  C_s^2|\mathcal{O}|^{\frac{r-q}{r+1}} 2^{q-1}q\beta\widetilde{K}^{\frac{q-1}{r+1}}\|\bu-\bu^h\|_{\V}\|\bu-\bv^h\|_{\V}\nonumber\\&\leq \frac{(2\mu-\delta_1\lambda_0^{-1})}{16}\|\bu-\bu^h\|_{\V}^2+\frac{4(C_s^2|\mathcal{O}|^{\frac{r-q}{r+1}} 2^{q-1}q\beta\widetilde{K}^{\frac{q-1}{r+1}})^2}{(2\mu-\delta_1\lambda_0^{-1})}\|\bu-\bv^h\|_{\V}^2. 
	\end{align*}
	Therefore, the equation \eqref{eqn-error-8} leads to 
	\begin{align}\label{eqn-error-8}
		K_{c_0}&\leq \frac{(2\mu-\delta_1\lambda_0^{-1})}{16}\|\bu-\bu^h\|_{\V}^2+\frac{\beta}{2}\||\bu|^{\frac{r-1}{2}}(\bu-\bu^h)\|_{\H}^2+\frac{\beta}{4}\||\bu^h|^{\frac{r-1}{2}}(\bu-\bu^h)\|_{\H}^2\nonumber\\&\quad+(\varrho_{2,r}+\varrho_{3,r})\|\bu-\bu^h\|_{\H}^2+\frac{4(C_s^2|\mathcal{O}|^{\frac{r-q}{r+1}} 2^{q-1}q\beta\widetilde{K}^{\frac{q-1}{r+1}})^2}{(2\mu-\delta_1\lambda_0^{-1})}\|\bu-\bv^h\|_{\V}^2. 
	\end{align}
Using the second  condition given in \eqref{eqn-hemi-1} and \eqref{eqn-discrete-hemi}, we deduce for any $q^h\in Q^h$
\begin{align*}
	K_d&=\mathfrak{d}(\bu^h-\bv^h,p-p^h)\nonumber=
	\mathfrak{d}(\bu^h,p)-\mathfrak{d}(\bv^h,p)+\mathfrak{b}(\bv^h,p^h) 
	\nonumber\\&= \mathfrak{d}(\bu^h-\bu,p-q^h)+\mathfrak{d}(\bu-\bv^h,p)+\mathfrak{d}(\bv^h-\bu,p^h)
	\nonumber\\&=  \mathfrak{d}(\bu^h-\bu,p-q^h)+\mathfrak{d}(\bu-\bv^h,p-p^h). 
\end{align*}
Therefore, by using the Cauchy-Schwarz inequality, \eqref{eqn-equiv-1}, \eqref{eqn-traingle-1}, and Young's inequality, it is immediate that 
\begin{align}\label{eqn-error-9}
	K_d&\leq C_k\left(\|\bu-\bu^h\|_{\V}\|p-q^h\|_Q+\|\bu-\bv^h\|_{\V}\|p-p^h\|_Q\right)\nonumber\\&\leq C(\|\bu-\bu^h\|_{\V}\|p-q^h\|_Q+\|\bu-\bv^h\|_{\V}\|\bu-\bu^h\|_{\V}+\|\bu-\bv^h\|_{\V}\|p-q^h\|_Q)\nonumber\\&\leq \frac{(2\mu-\delta_1\lambda_0^{-1})}{8}\|\bu-\bu^h\|_{\V}^2+C\left(\|p-q^h\|_Q^2+\|\bu-\bv^h\|_{\V}^2\right).
\end{align}
Using the sub-additivity property given in \eqref{eqn-subaddtive}, we get 
\begin{align*}
	j^0(\bu_{\tau}^h;\bv_{\tau}^h-\bu_{\tau}^h)\leq  	j^0(\bu_{\tau}^h;\bu_{\tau}-\bu_{\tau}^h) + 	j^0(\bu_{\tau}^h;\bv_{\tau}^h-\bu_{\tau}). 
\end{align*}
The above estimate, Hypothesis \ref{hyp-sup-j}, \eqref{eqn-alternative} and \eqref{eqn-trace} help  us to evaluate $K_j$ as 
\begin{align}\label{eqn-error-10}
	K_j&=\int_{\Gamma_1}\Big[j^0(\bu_{\tau};\bu^h_{\tau}-\bu_{\tau})+j^0(\bu_{\tau}^h;\bv^h_{\tau}-\bu^h_{\tau})+j^0(\bu_{\tau};\bu_{\tau}-\bv^h_{\tau})\Big]\d S\nonumber\\&\leq  \int_{\Gamma_1}\Big[j^0(\bu_{\tau};\bu^h_{\tau}-\bu_{\tau})	+j^0(\bu_{\tau}^h;\bu_{\tau}-\bu_{\tau}^h) + 	j^0(\bu_{\tau}^h;\bv_{\tau}^h-\bu_{\tau})+j^0(\bu_{\tau};\bu_{\tau}-\bv^h_{\tau})\Big]\d S\nonumber\\&\leq \int_{\Gamma_1} \left[\delta_1|\bu_{\tau}-\bu_{\tau}^h|^2+C\left(1+|\bu^h_{\tau}|+|\bu_{\tau}|\right)|\bu_{\tau}-\bv^h_{\tau}|\right]\d S\nonumber\\&\leq\delta_1\lambda_0^{-1}\|\bu-\bu^h\|_{\V}^2+C\left(|\Gamma_1|^{1/2}+\|\bu^h_{\tau}\|_{\L^2(\Gamma_1)}+\|\bu_{\tau}\|_{\L^2(\Gamma_1)}\right)\|\bu_{\tau}-\bu^h_{\tau}\|_{\L^2(\Gamma_1)}\nonumber\\&\leq \delta_1\lambda_0^{-1}\|\bu-\bu^h\|_{\V}^2+C\left(|\Gamma_1|^{1/2}+\lambda_0^{-1/2}\|\bu^h\|_{\V}+\lambda_0^{-1/2}\|\bu\|_{\V}\right)\|\bu_{\tau}-\bu^h_{\tau}\|_{\L^2(\Gamma_1)}. 
\end{align}
Using the estimates \eqref{eqn-bound} and \eqref{eqn-discrete-bound} in \eqref{eqn-error-10}, we deduce 
\begin{align}\label{eqn-error-11}
	K_j\leq \delta_1\lambda_0^{-1}\|\bu-\bu^h\|_{\V}^2+C\widetilde{K}\|\bu_{\tau}-\bu^h_{\tau}\|_{\L^2(\Gamma_1)}. 
\end{align}
Combining \eqref{eqn-error-b}, \eqref{eqn-error-6}, \eqref{eqn-error-8}, \eqref{eqn-error-9}, \eqref{eqn-error-9} and \eqref{eqn-error-11} and substituting the resultant  in \eqref{eqn-error-main}, we arrive at 
\begin{align}\label{eqn-error-12}
&	\frac{(2\mu-\delta_1\lambda_0^{-1})}{8}\|\bu-\bu^h\|_{\V}^2+(\alpha-(\widetilde{\varrho}_{1,r}+\varrho_{2,r}+\varrho_{3,r}))\|\bu-\bu^h\|_{\H}^2\nonumber\\&\leq C\left(\|\bu-\bv^h\|_{\V}^2+\|\bu_{\tau}-\bu^h_{\tau}\|_{\L^2(\Gamma_1)}+\|p-q^h\|_Q^2\right).
\end{align}
 For $\mu>\frac{\delta_1}{2\lambda_0}$ and $\alpha> (\widetilde{\varrho}_{1,r}+\varrho_{2,r}+\varrho_{3,r})$, using \eqref{eqn-traingle-1} and \eqref{eqn-error-12}, we finally have 
 \begin{align}\label{eqn-error-13}
 	&	\|\bu-\bu^h\|_{\V}^2+\|\bu-\bu^h\|_{\H}^2+\|p-p^h\|_Q^2\nonumber\\&\leq C\left(\|\bu-\bv^h\|_{\V}^2+\|\bu_{\tau}-\bu^h_{\tau}\|_{\L^2(\Gamma_1)}+\|p-q^h\|_Q^2\right),
 \end{align}
so that C\'ea’s inequality \eqref{eqn-cea} follows. Instead of the estimate \eqref{eqn-error-b1}, if one uses an estimate similar to \eqref{2.26}, then we can obtain \eqref{eqn-error-13}, provided $\mu>\frac{\delta_1}{2\lambda_0}$, $\alpha\geq (\widehat{\varrho}_{1}+\varrho_{2,r}+\varrho_{3,r})$ and $\beta\geq 4\widehat{\varrho}_{1}$.

\textbf{Case 2.} \emph{$d\in\{2,3\}$ with $r\in[1,3]$.} 
We can use an estimate similar to \eqref{2.21} to evaluate $|\mathfrak{b}(\bu-\bu^h,\bu^h,\bu-\bu^h)|$ as 
\begin{align}\label{eqn-error-14}
&	|\mathfrak{b}(\bu-\bu^h,\bu^h,\bu-\bu^h)|\nonumber\\&\leq \frac{(2\mu-\delta_1\lambda_0^{-1})}{4}\|\bu-\bu^h\|_{\V}^2+\widehat{\varrho}_4 \|\bu^h\|_{\L^4}^{\frac{8}{4-d}}\|\bu-\bu^h\|_{\H}^2\nonumber\\&\leq \frac{(2\mu-\delta_1\lambda_0^{-1})}{4}\|\bu-\bu^h\|_{\V}^2+\widehat{\varrho}_4(C_gC_k)^{\frac{8}{4-d}}\widetilde{K}^{\frac{4}{4-d}}\|\bu-\bu^h\|_{\H}^2, 
	\end{align}	
where we have used  \eqref{eqn-discrete-bound} and  $\widehat{\varrho}_4$ is defined in \eqref{eqn-varrhohat-1}.  Using \eqref{eqn-b-est-1}, \eqref{eqn-bound}, \eqref{eqn-discrete-bound} and \eqref{eqn-error-14},  we estimate $K_b$ given in \eqref{eqn-error-b0} as 
\begin{align}
	K_b&\leq \frac{(2\mu-\delta_1\lambda_0^{-1})}{4}\|\bu-\bu^h\|_{\V}^2+\widehat{\varrho}_4(C_gC_k)^{\frac{8}{4-d}}\widetilde{K}^{\frac{4}{4-d}}\|\bu-\bu^h\|_{\H}^2\nonumber\\&\quad+C_b\left(\|\bu\|_{\V}+\|\bu^h\|_{\V}\right)\|\bu-\bu^h\|_{\V}\|\bu^h-\bv^h\|_{\V}\nonumber\\&\leq \frac{(2\mu-\delta_1\lambda_0^{-1})}{4}\|\bu-\bu^h\|_{\V}^2+\widehat{\varrho}_4(C_gC_k)^{\frac{8}{4-d}}\widetilde{K}^{\frac{4}{4-d}}\|\bu-\bu^h\|_{\H}^2\nonumber\\&\quad+ 2C_b\widetilde{K}\|\bu-\bu^h\|_{\V}\|\bu-\bv^h\|_{\V}\nonumber\\&\leq \frac{3(2\mu-\delta_1\lambda_0^{-1})}{8}\|\bu-\bu^h\|_{\V}^2+\widehat{\varrho}_4(C_gC_k)^{\frac{8}{4-d}}\widetilde{K}^{\frac{4}{4-d}}\|\bu-\bu^h\|_{\H}^2\nonumber\\&\quad+\frac{8C_b^2\widetilde{K}^2}{(2\mu-\delta_1\lambda_0^{-1})}\|\bu-\bv^h\|_{\V}^2.
\end{align}
	Therefore, for $\mu>\frac{\delta_1}{2\lambda_0}$ and $\alpha>	\widehat{\varrho}_4(C_gC_k)^{\frac{8}{4-d}}\widetilde{K}^{\frac{4}{4-d}}+\varrho_{2,r}+\varrho_{3,r}$,  the estimate \eqref{eqn-error-13} follows.
\end{proof}

As an application of Theorem \ref{thm-cea}, we consider the widely used $P1b/P1$ finite element pair (see \cite{DNAFB}).
\begin{align}
	\V^h&=\left\{\bv^h\in \V\cap\mathrm{C}_0(\overline{\mathcal{O}};\mathbb{R}^d):\bv^h\big|_T\in[P_1(T)]^d\oplus B(T)\ \text{ for all }\ T\in \mathcal{T}^h\right\}\label{eqn-p1b},\\
	Q^h&=\left\{p^h\in Q\cap\mathrm{C}_0(\overline{\mathcal{O}}):p^h\big|_T\in P_1(T) \ \text{ for all }\ T\in \mathcal{T}^h\right\},\label{eqn-p1}
\end{align}
where, $P_1(T)$ denotes the space of polynomials of degree less than or equal to $1$ on an element $T$, and $B(T)$ represents the space of bubble functions defined on $T$. With these finite element spaces, the discrete inf-sup condition \eqref{eqn-inf-sup-discrete} is satisfied (\cite{DNAFB}). Consequently, under appropriate regularity assumptions on the exact solution, we can derive an optimal-order error estimate for the $P1b/P1$ finite element solution using estimate \eqref{eqn-optimal-order} in conjunction with standard interpolation error bounds from finite element theory. We represent $\overline{\Gamma}_1$ as the union of a finite number of flat components:
$$\overline{\Gamma}_1=\bigcup_{i=1}^{i_0}\Gamma_{1,i},$$
where each $\Gamma_{1,i}$ is a line segment in two dimensions or a planar polygonal surface in three dimensions.

\begin{theorem}\label{thm-optimal}
	Let the assumptions of Theorem \ref{thm-unique} hold, and suppose that $(\bu, p)$ and $(\bu^h, p^h)$ are the exact and discrete solutions of Problems \ref{prob-hemi-1} and \ref{prob-hemi-3}, respectively, using the $P1b/P1$ finite element pair \eqref{eqn-p1b}-\eqref{eqn-p1}. Assume the following regularity conditions: $\bu \in \H^2(\mathcal{O})$, $\bu_{\tau} \in \H^2(\Gamma_{1,i})$ for $1 \leq i \leq i_0$, and $p \in \mathrm{H}^1(\mathcal{O})$. Then the following a priori error estimate holds:
	\begin{align}\label{eqn-optimal-order}
		\|\bu - \bu^h\|_{\V} +\|p - p^h\|_{Q} \leq C h.
	\end{align}
\end{theorem}

\begin{proof}
	Let $\bv^h = \Pi^h \bu$ denote the standard finite element interpolant of $\bu$ in the velocity space, and let $q^h = P^h p$ be the $L^2$-projection of $p$ onto the pressure space. 
	Using standard interpolation estimates for conforming finite elements, we have:
	\begin{align*}
		\|\bu-\Pi^h\bu\|_{\V}&\leq Ch\|\bu\|_{\H^2(\mathcal{O})},\\
		\|p-P^hp\|_{Q}&\leq Ch\|p\|_{\mathrm{H}^1(\mathcal{O})},\\
		\|\bu_{\tau}-(\Pi^h\bu)_{\tau}\|_{\L^2(\Gamma_{1,i})}&\leq Ch^2\|\bu_{\tau}\|_{\H^2(\Gamma_{1,i})},\ 1\leq i\leq i_0. 
	\end{align*}
	Substituting these bounds into  C\'ea's inequality \eqref{eqn-cea} yields the desired estimate \eqref{eqn-optimal-order}.
\end{proof}

	\section{Solution Algorithm and Numerical Examples}\label{sec5}\setcounter{equation}{0}
	In this section, we present an iterative algorithm for solving Problem \ref{prob-hemi-1} and provide numerical simulation results to illustrate its performance.
	\subsection{Solution algorithm}\label{sub-sol-alg} As discussed in \cite[Section 5.1]{WHHQLM}, we employ a Newton-type method to address the nonlinear terms $\mathfrak{c}(\cdot, \cdot)$ and $\mathfrak{c}_0(\cdot,\cdot)$ arising in Problem \ref{eqn-discrete-hemi}. This approach is inspired by a linearization strategy. Specifically, we represent the next iterate $\bu^h_{n+1}$ as
	$
	\bu^h_{n+1} = \bu^h_n + \boldsymbol{\delta}^h_n,
	$
		where the correction term $\boldsymbol{\delta}^h_n = \bu^h_{n+1} - \bu^h_n$ is assumed to be small.
		To proceed with the linearization, let us consider the following function of a real variable with real values:
		\begin{align*}
			g(\theta)&=|\bu^h_n+\theta\boldsymbol{\delta}^h_n|^{r-1}(\bu^h_n+\theta\boldsymbol{\delta}^h_n),\ \theta\in\R,
			\end{align*}
			and use the approximation 
			\begin{align*}
				g(1)\approx g(0)+g^{\prime}(0),
			\end{align*}
			where $	g(0)=|\bu^h_n|^{r-1}\bu^h_n$  and 
			\begin{align*}
			g^{\prime}(0)=\left\{\begin{array}{cl}\boldsymbol{\delta}^h_n,&\text{ for }r=1,\\ \left\{\begin{array}{cc}|\bu^h_n|^{r-1}\boldsymbol{\delta}^h_n+(r-1)\frac{\bu^h_n}{|\bu^h_n|^{3-r}}(\bu^h_n\cdot\boldsymbol{\delta}^h_n),&\text{ if }\bu^h_n\neq \boldsymbol{0},\\\boldsymbol{0},&\text{ if }\bu^h_n=\boldsymbol{0},\end{array}\right.&\text{ for } 1<r<3,\\ |\bu^h_n|^{r-1}\boldsymbol{\delta}^h_n+(r-1)\bu^h_n|\bu^h_n|^{r-3}(\bu^h_n\cdot\boldsymbol{\delta}^h_n)& \text{ for } 3\leq r<\infty.\end{array}\right.
			\end{align*}
			Similarly, we consider
			\begin{align*}
				h(\theta)=\B(\bu^h_n+\theta\boldsymbol{\delta}^h_n,\bu^h_n+\theta\boldsymbol{\delta}^h_n),\ \theta\in\R,
			\end{align*}
			and use the approximation 
			\begin{align*}
				h(1)\approx h(0)+h^{\prime}(0),
			\end{align*}
			where $	h(0)=\B(\bu^h_n,\bu^h_n)$  and 
			\begin{align*}
				h^{\prime}(0)=\B(\bu^h_n,\boldsymbol{\delta}^h_n)+\B(\boldsymbol{\delta}^h_n,\bu^h_n). 
			\end{align*}
	
Accordingly, using the notation introduced in \eqref{eqn-c-c0}, we now present the following iterative scheme for the case $3 \leq r < \infty$:
		\begin{algorithm}\label{alg-1}
	Start by selecting an initial approximation  $(\bu^h_0,p^h) \in \V^h\times Q^h$. Then, for each $n \geq 0$, find $(\bu^h_{n+1}, p^h_{n+1}) \in \V^h \times Q^h$ such that:
			\begin{equation}\label{eqn-algo}
			\left\{
			\begin{aligned}
				&\mu \mathfrak{a}(\bu^h_{n+1},\bv^h)+\mathfrak{b}(\bu^h_{n+1},\bu^h_n,\bv^h)+\mathfrak{b}(\bu^h_{n},\bu^h_{n+1},\bv^h)+\alpha \mathfrak{a}_0(\bu^h_{n+1},\bv^h)+\beta \mathfrak{c}(\bu^h_n,\bu^h_{n+1},\bv^h)\\&\quad+\beta(r-1)(\bu^h_n|\bu^h_n|^{r-3}(\bu^h_n\cdot\bu^h_{n+1}),\bv^h)+\kappa \mathfrak{c}_0(\bu^h_n,\bu^h_{n+1},\bv^h)\\&\quad+\kappa(q-1)(\bu^h_n|\bu^h_n|^{q-3}(\bu^h_n\cdot\bu^h_{n+1}),\bv^h)+\mathfrak{d}(\bv^h,p_{n+1}^h)+\int_{\Gamma_1}j^0(\bu_{\tau,n+1}^h;\bv_{\tau}^h)\d S \\ &\geq 	(\f,\bv^h)_{\L^2(\mathcal{O})}+\mathfrak{b}(\bu_n^h,\bu_n^h,\bv^h)+\beta(r-1) \mathfrak{c}(\bu^h_n,\bu^h_n,\bv^h)+\kappa (q-1)\mathfrak{c}_0(\bu^h_n,\bu^h_n,\bv^h) \\ &\hskip10cm \text{ for all }\ \bv^h\in \V^h,\\
				&\mathfrak{d}(\bu^h_{n+1},q^h)=0\ \text{ for all }\ q^h\in Q^h. 
			\end{aligned}\right.
		\end{equation}
	\end{algorithm}
	An analogous iterative scheme can be developed for the case $r \in [1,3)$.

	\subsection{Numerical examples} We present numerical results for three examples to demonstrate the performance of Algorithm \ref{alg-1}.  In these examples, we set the following: 
\begin{align*}
	j(\z)=\int_0^{|\z|}\omega(t)\d t,
\end{align*}
where $\omega:[0,\infty)\to\R$  is continuous, $\omega(0)>0$. In the examples given below, we have taken $\omega(t)=(a-b)e^{-\rho t}+b$, where $a>b>0$ and $\rho>0$ are constants, specified in Table \ref{tab:gen data} below. The nonlinear slip boundary condition $-\boldsymbol{\sigma}_{\tau}\in \partial j(\bu_{\tau})$  is equivalent to
\begin{align*}
	|\boldsymbol{\sigma}_{\tau}|\leq \omega(0)\ \text{ if }\ \bu_{\tau}=\boldsymbol{0},\quad -\boldsymbol{\sigma}_{\tau}=\omega(|\bu_{\tau}|)\frac{\bu_{\tau}}{|\bu_{\tau}|} \ \text{ if }\ \bu_{\tau}\neq\boldsymbol{0}. 
\end{align*}
We define a Lagrange multiplier $$\boldsymbol{\lambda}= \frac{-\boldsymbol{\sigma}_{\tau}}{\omega(|\bu_{\tau}|)} $$ and introduce a set
\begin{align*}
	\Lambda=\left\{\boldsymbol{\lambda}\in\L^2(\Gamma_1)\big| |\boldsymbol{\lambda}|\leq 1\ \text{ a.e. on }\ \Gamma_1 \right\}. 
\end{align*}
Therefore, the weak formulation of the problem \eqref{eqn-CBF}-\eqref{eqn-boundary-2} can then be stated as follows:
	\begin{problem}\label{prob-hemi-num}
	Find $\bu\in\V$, $p\in Q$ and $\boldsymbol{\lambda}\in\Lambda$  such that
	\begin{equation}\label{eqn-hemi-num}
		\left\{
		\begin{aligned}
			\mu \mathfrak{a}(\bu,\bv)+\mathfrak{b}(\bu,\bu,\bv)+\alpha \mathfrak{a}_0(\bu,\bv)&+\beta \mathfrak{c}(\bu,\bv)+\kappa \mathfrak{c}_0(\bu,\bv)+\mathfrak{d}(\bv,p)\\+\int_{\Gamma_1}\omega(|\bu_{\tau}|)\boldsymbol{\lambda}\cdot\bv_{\tau}\d S&=	\langle\f,\bv\rangle  \ \text{ for all }\ \bv\in \V,\\
			\mathfrak{d}(\bu,q)&=0\ \text{ for all }\ q\in Q,\\
			\boldsymbol{\lambda}\cdot\bu_{\tau}&=|\bu_{\tau}|\ \text{ a.e. on }\ \Gamma_1. 
		\end{aligned}\right.
	\end{equation}
\end{problem}
For the case, $r\in[3,\infty)$, the iteration step in Algorithm \ref{alg-1} is then reformulated as follows: find $(\bu^h_{n+1},p^h_{n+1},\boldsymbol{\lambda}^h_{n+2})\in\V^h\times Q^h\times\Lambda$  such that
	\begin{equation}\label{eqn-algo-1}
	\left\{
	\begin{aligned}
		&\mu \mathfrak{a}(\bu^h_{n+1},\bv^h)+\mathfrak{b}(\bu^h_{n+1},\bu^h_n,\bv^h)+\mathfrak{b}(\bu^h_{n},\bu^h_{n+1},\bv^h)+\alpha \mathfrak{a}_0(\bu^h_{n+1},\bv^h)+\beta \mathfrak{c}(\bu^h_n,\bu^h_{n+1},\bv^h)\\&\quad+\beta(r-1)(\bu^h_n|\bu^h_n|^{r-3}(\bu^h_n\cdot\bu^h_{n+1}),\bv^h)+\kappa \mathfrak{c}_0(\bu^h_n,\bu^h_{n+1},\bv^h)\\&\quad+\kappa(q-1)(\bu^h_n|\bu^h_n|^{q-3}(\bu^h_n\cdot\bu^h_{n+1}),\bv^h)+\mathfrak{d}(\bv^h,p_{n+1}^h)+\int_{\Gamma_1}\omega(|\bu^h_{\tau,n}|)\boldsymbol{\lambda}^h_{n+1}\cdot\bv^h_{\tau}\d S \\ &\geq 	(\f,\bv^h)_{\L^2(\mathcal{O})}+\mathfrak{b}(\bu_n^h,\bu_n^h,\bv^h)+\beta(r-1) \mathfrak{c}(\bu^h_n,\bu^h_n,\bv^h)+\kappa (q-1)\mathfrak{c}_0(\bu^h_n,\bu^h_n,\bv^h) \\ &\hskip10cm \text{ for all }\ \bv^h\in \V^h,\\
		&\mathfrak{d}(\bu^h_{n+1},q^h)=0\ \text{ for all }\ q^h\in Q^h,\\
		&\boldsymbol{\lambda}^h_{n+2}\cdot\bu^h_{\tau,n+1}=|\bu^h_{\tau,n+1}|\ \text{ a.e. on }\ \Gamma_1. 
	\end{aligned}\right.
\end{equation}
For the implementation, we employ the following Uzawa iteration scheme \cite[Section 5]{YLKL}: choose an initial guess $(\bu^h_0,p^h_0) \in \V_h\times Q_h$, and then, for each $n \geq 0$, find $(\bu^h_{n+1},p^h_{n+1},\boldsymbol{\lambda}^h_{n+2})\in\V^h\times Q^h\times\Lambda$  such that
	\begin{equation}\label{eqn-algo-2}
	\left\{
	\begin{aligned}
		&\mu \mathfrak{a}(\bu^h_{n+1},\bv^h)+\mathfrak{b}(\bu^h_{n+1},\bu^h_n,\bv^h)+\mathfrak{b}(\bu^h_{n},\bu^h_{n+1},\bv^h)+\alpha \mathfrak{a}_0(\bu^h_{n+1},\bv^h)+\beta \mathfrak{c}(\bu^h_n,\bu^h_{n+1},\bv^h)\\&\quad+\beta(r-1)(\bu^h_n|\bu^h_n|^{r-3}(\bu^h_n\cdot\bu^h_{n+1}),\bv^h)+\kappa \mathfrak{c}_0(\bu^h_n,\bu^h_{n+1},\bv^h)\\&\quad+\kappa(q-1)(\bu^h_n|\bu^h_n|^{q-3}(\bu^h_n\cdot\bu^h_{n+1}),\bv^h)+\mathfrak{d}(\bv^h,p_{n+1}^h) \\ &\geq 	(\f,\bv^h)_{\L^2(\mathcal{O})}+\mathfrak{b}(\bu_n^h,\bu_n^h,\bv^h)+\beta(r-1) \mathfrak{c}(\bu^h_n,\bu^h_n,\bv^h)+\kappa (q-1)\mathfrak{c}_0(\bu^h_n,\bu^h_n,\bv^h), \\ &\quad -\int_{\Gamma_1}\omega(|\bu^h_{\tau,n}|)\boldsymbol{\lambda}^h_{n+1}\cdot\bv^h_{\tau}\d S\  \text{ for all }\ \bv^h\in \V^h,\\
		&\mathfrak{d}(\bu^h_{n+1},q^h)=0\ \text{ for all }\ q^h\in Q^h,\\
		&\boldsymbol{\lambda}^h_{n+2}=P_{\Lambda}(\boldsymbol{\lambda}^h_{n+1}+\eta \bu^h_{\tau,n+1})\ \text{ on }\ \Gamma_1, 
	\end{aligned}\right.
\end{equation}
where  $\eta>0$, and $P_{\Lambda}$ is the projection of $\L^2(\Gamma_1)$ onto $\Lambda$. 
This problem is solved as follows: choose an initial guess $\boldsymbol{\lambda}^h_{n+1,0}$, then for $\ell = 0, 1, \ldots$,
\begin{enumerate}
	\item [1.] solve \eqref{eqn-algo-2} for $(\bu^h_{n+1,\ell}, p^h_{n+1,\ell})$;
	\item [2.] $\boldsymbol{\lambda}^h_{n+1,\ell+1}=P_{\Lambda}(\boldsymbol{\lambda}^h_{n+1,\ell}+\eta \bu^h_{\tau,n+1,\ell})$;
	\item [3.]
	iterate Steps 1 and 2 until the stopping criterion
	$
	\|\bu^h_{n+1,\ell} - \bu^h_{n+1,\ell-1}\|_{\H} \leq \varepsilon_2
	$
	is satisfied. At this point, define
		$
	(\bu^h_{n+1}, p^h_{n+1}, \boldsymbol{\lambda}^h_{n+2})
	$
		as the most recent iterates.
\end{enumerate}
Repeat the above procedure until $\|\bu^h_{n+1}-\bu^h_{n}\|_{\H}\leq\eps_1$.

We consider three examples for implementation numerically and to support our theoretical results. In the forthcoming numerical experiments, we consider the computational domain as the unit square $ \mathcal{O} = (0, 1) \times (0, 1) $ with slip boundary condition enforced along the top boundary $ \Gamma_1 = (0,1) \times \{1\} $, while homogeneous Dirichlet boundary conditions are imposed on the remaining portion of the boundary. We also use the nonlinear slip coefficient of the form
$$\omega(t)=(a-b)e^{-\rho t}+b,$$
where $ a > b>0 $ are given as in the Table \ref{tab:gen data}. The numerical simulations are carried out using the finite element library FEniCS, implemented in Python. The iterative procedure employs the following setup: for the outer iteration, we initialize with $ \bu_h^0 = \boldsymbol{0} $ and apply a stopping criterion of either a relative tolerance $\varepsilon_1 = 10^{-8} $ or a maximum of fifty iterations. 

The Uzawa algorithm is initialized with $ \boldsymbol{\lambda}^h_{0,0} = \boldsymbol{0} $ for the first iteration, while for subsequent outer iterations, the solution from the previous step $ \boldsymbol{\lambda}^h_{n-1} $ is used as the initial guess, that is, $ \boldsymbol{\lambda}^h_{n,0} = \boldsymbol{\lambda}^h_{n-1} $. We set the Uzawa step parameter $ \eta = 1 $ for first two examples and $\eta=0.8$ for the third example, and use a tolerance of $ \varepsilon_2 = 10^{-8} $ or terminate after twenty iterations, whichever occurs first.

The parameters for the those three examples are listed below:
\begin{table}[h!]
	\centering
	\begin{tabular}{@{}llcccccccc@{}}
		\toprule
		\textbf{Example} & $\mu$ & $\alpha$ & $\beta$ & $\kappa$ & $(r, q)$ & $a$ & $b$ & $\rho$ & $\eta$ \\
		\midrule
		1  & 1.2 & 2 & 1.5 & 0 & $(3,-)$ & 1.55 & 1.53 & 8.0 & 1 \\
		2 & 0.8 & 1.5 & 2 & -1.2 & $(3, 2)$ & 5.01 & 5.00 & 8.0 & 1\\
		3  & 1.0 & 0.5 & 1.2 & -1.0 & $(4, 3)$ & 3.25 & 3.20 & 6.0 & 0.8\\
		\bottomrule
	\end{tabular}
	\caption{Setup for numerical experiments} \label{tab:gen data}
\end{table}

\begin{remark} Note that
	\begin{itemize}
		\item Example 1 corresponds to the classical convective Brinkman–Forchheimer model (no pumping term) as we choose $\kappa=0$;
		\item Example 2 includes both Forchheimer and extended Darcy terms;
		\item the forcing function $\f$ is computed using the exact solution is substituted into \eqref{eqn-CBF} for chosen $\bu^0$ and $p^0$ (see inside the examples below).
	\end{itemize}
\end{remark}

\noindent \textbf{Example 1.} As mentioned in the remark above, this example is without pumping term, that is, for $\kappa=0.$ In this example, we consider the force function $\f$ to be calculated as 
$$ \f = -\mu\Delta\bu^0 + (\bu^0 \cdot \nabla)\bu^0 + \alpha\bu^0 + \beta|\bu^0|^{r-1}\bu^0 + \nabla p^0,$$
where $\bu^0=(u_1^0, u_2^0)$ and $p^0$ are
\begin{align*}
	u_1^0(x, y) &= -x^2 (x - 1) \, y \, (3y - 2), \\
	u_2^0(x, y) &= x (3x - 2) \, y^2 (y - 1), \\
	p^0(x, y)   &= (2x - 1)(2y - 1).
\end{align*}
The other parameters are as tabulated in Table \ref{tab:gen data}. The computed velocity field and pressure are plotted in Figure \ref{fig:ex1-velocity_pressure} while the tangential components of velocity $\bu_\tau$ and stress tensor $\sigma_\tau$ along slip boundary $y=1$ are plotted in Figure \ref{fig:ex12-Slip}. The computed errors in vector field and pressure term along with rate of convergence are tabulated in Table \ref{tab:ex1-roc}, where to compute the errors, the reference solutions are considered as $\bu^*$ and $p^*,$ obtained on a mesh with grid size $350 \times 350$. In Figure \ref{fig:ex1-error-plot}, we plot the errors $\boldsymbol{u}_h - \boldsymbol{u}^*$, $|\boldsymbol{u}_h - \boldsymbol{u}^*|$, and $|p_h - p^*|$ in mesh with grid sizes $10\times 10,$ $30\times 30,$ and $50\times 50.$
\begin{figure}[h!]
	\centering
	\begin{subcaptionbox}{Velocity field\label{fig:ex1-velocity}}[0.48\textwidth]
		{\includegraphics[width=\linewidth]{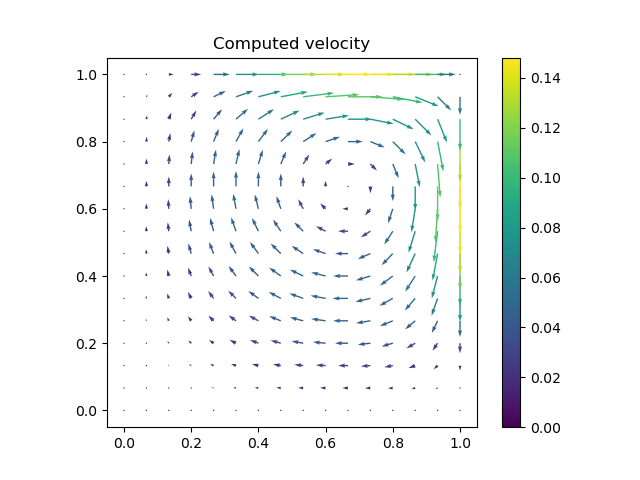}}
	\end{subcaptionbox}
	\hfill
	\begin{subcaptionbox}{Pressure \label{fig:ex-1pressure}}[0.48\textwidth]
		{\includegraphics[width=\linewidth]{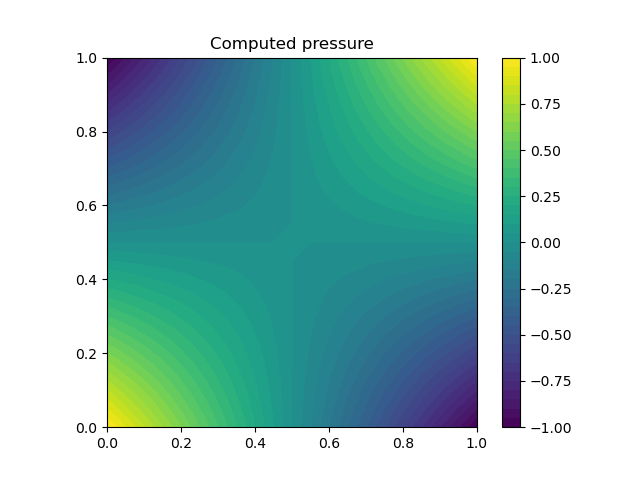}}
	\end{subcaptionbox}
	\caption{Velocity and pressure plots for the computed solution.}
	\label{fig:ex1-velocity_pressure}
\end{figure}

\begin{figure}[h!]
	\centering
	\begin{subcaptionbox}{Velocity field\label{fig:ex1-u slip}}[0.48\textwidth]
		{\includegraphics[width=\linewidth]{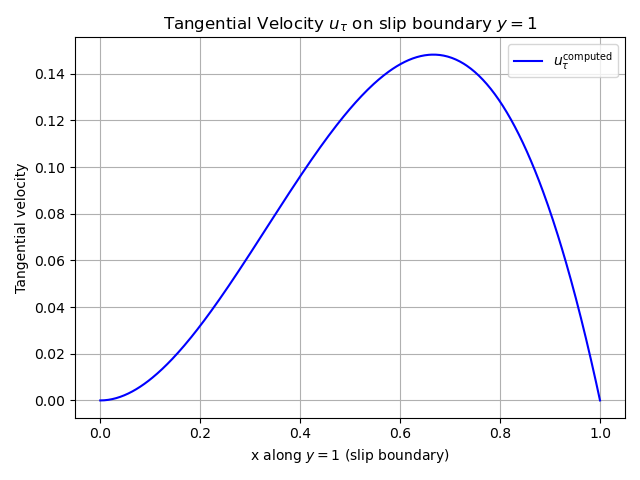}}
	\end{subcaptionbox}
	\hfill
	\begin{subcaptionbox}{Pressure Field\label{fig:ex1-sigma_slip}}[0.48\textwidth]
		{\includegraphics[width=\linewidth]{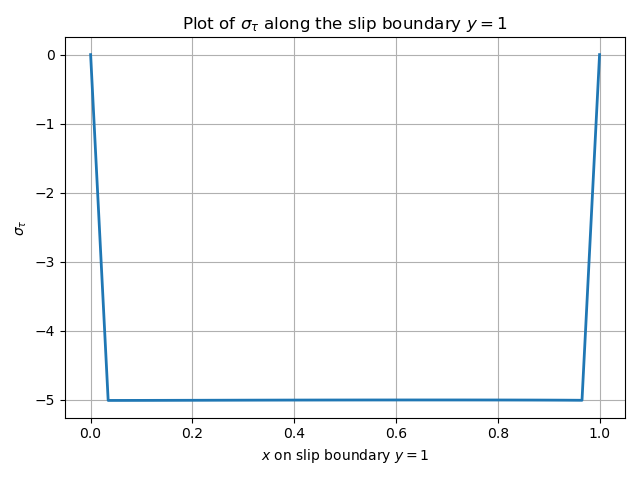}}
	\end{subcaptionbox}
	\caption{Velocity and pressure plots for the computed solution.}
	\label{fig:ex12-Slip}
\end{figure}

\begin{table}[h!]
	\centering
	\begin{tabular}{c c c c c c c}
		\toprule
		Grid size & $\|\bu^h - \bu^*\|_{\L^2}$ & Order($h$) & $\|\bu^h - \bu^*\|_{\V}$ & Order($h$) & $\|p^h - p^*\|_{L^2}$ & Order($h$) \\
		\midrule
		$5 \times 5$   & 3.778e-03 & -- & 1.414e-01 & -- & 1.361e-02 & -- \\
		$10 \times 10$   & 9.856e-04 & 1.94  & 7.376e-02 & 0.94  & 3.417e-03 & 1.99 \\
		$15 \times 15$  & 4.409e-04 & 1.98 & 4.959e-02 & 0.98  & 1.520e-03 & 2.00 \\
		$20 \times 20$  & 2.482e-04 & 2.00 & 3.728e-02 & 0.99 & 8.553e-04 & 2.00 \\
		$25 \times 25$  & 1.589e-04 & 2.00 & 2.973e-02 & 1.01 & 5.475e-04 & 2.00 \\
		$30 \times 30$ & 1.100e-04 & 2.02 & 2.490e-02 & 0.97 & 3.802e-04 & 2.00 \\
		$35 \times 35$ & 8.076e-05 & 2.00 & 2.115e-02 & 1.06 & 2.794e-04 & 2.00 \\
		$40 \times 40$ & 6.148e-05 & 2.04 & 1.869e-02 & 0.93 & 2.139e-04 & 2.00 \\
		$45 \times 45$ & 4.837e-05 & 2.04 & 1.662e-02 & 0.99 & 1.690e-04 & 2.00 \\
		$50 \times 50$ & 3.918e-05 & 2.00 & 1.466e-02 & 1.19 & 1.369e-04 & 2.00 \\
		\bottomrule
	\end{tabular} 
	\caption{Numerical convergence orders}\label{tab:ex1-roc}
\end{table}

\begin{figure}[h!]
	\centering
	\includegraphics[width=\textwidth]{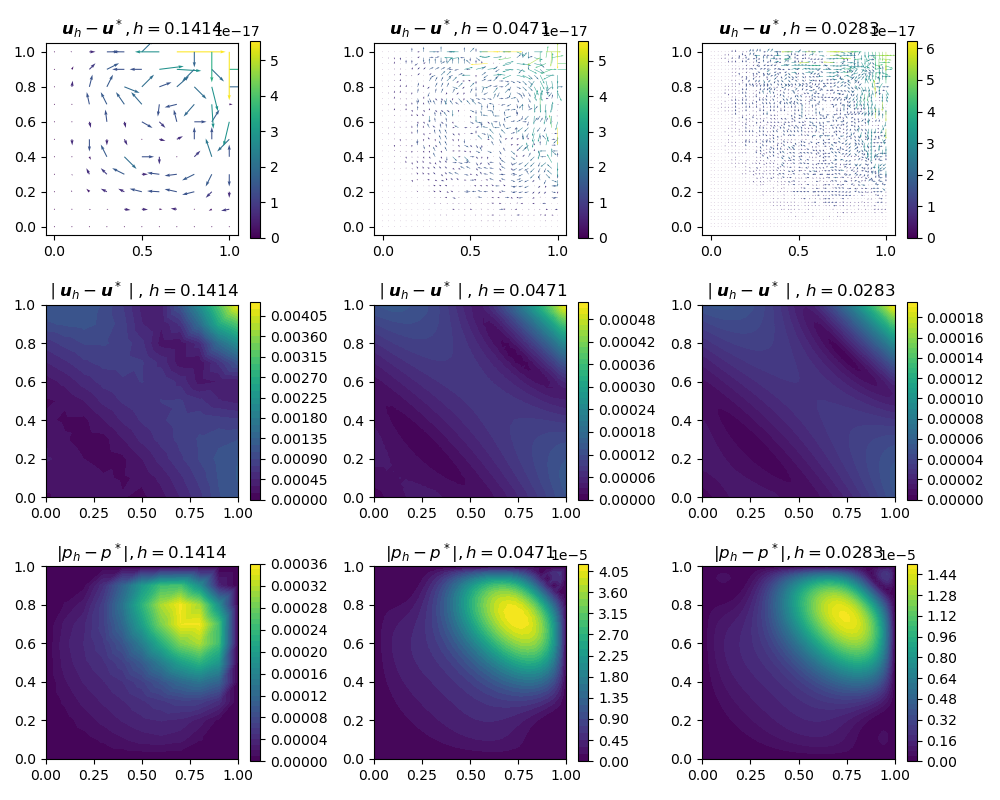}  
	\caption{Errors $\boldsymbol{u}_h - \boldsymbol{u}^*$, $|\boldsymbol{u}_h - \boldsymbol{u}^*|$, and $|p_h - p^*|$ for mesh with grid sizes $10\times 10,$ $30\times 30,$ and $50\times 50,$}
	\label{fig:ex1-error-plot}
\end{figure}

\noindent \textbf{Example 2.} For this example, the given source function $\f$ is calculated from 
\begin{align*}
	\f:=-\mu\Delta\bu^0 + (\bu^0 \cdot \nabla)\bu^0 + \alpha\bu^0 + \beta|\bu^0|^{r-1}\bu^0 + \kappa|\bu^0|^{q-1}\bu^0 + \nabla p^0,
\end{align*}
where $\bu^0=(u_1^0, u_2^0)$ and $p^0$ are chosen as
\begin{align*}
	u_1^0(x, y) &= -\cos(2\pi x)\sin(2\pi y) + \sin(2\pi y), \\
	u_2^0(x, y) &= \sin(2\pi x)\cos(2\pi y) - \sin(2\pi x), \\
	p^0(x, y)   &= 2\pi\left(\cos(2\pi y) - \cos(2\pi x)\right),
\end{align*}
and with the parameters as mentioned in Table \ref{tab:gen data}. 
The outcomes for this example are presented in Figures \ref{fig: ex2-velocity_pressure} - \ref{fig:ex2-Slip} and summarized in Table \ref{tab:ex2-roc}. The errors reported in Table \ref{tab:ex2-roc} are computed using the reference solution $(\bu^*, p^*)$, obtained on a mesh with grid size $350 \times 350$.

\begin{figure}[h!]
	\centering
	\begin{subcaptionbox}{Velocity field\label{fig:velocity}}[0.48\textwidth]
		{\includegraphics[width=\linewidth]{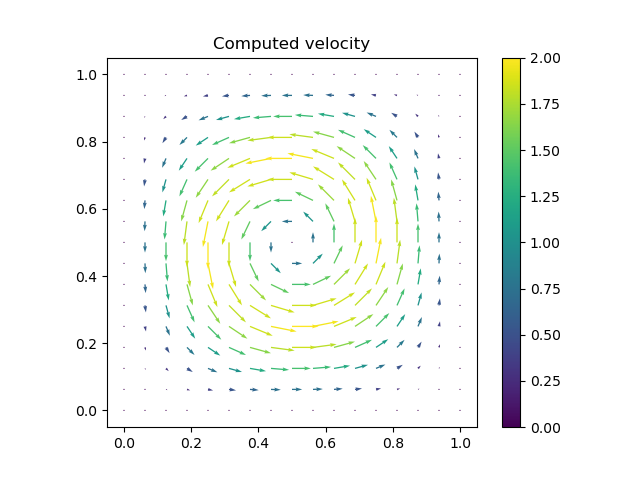}}
	\end{subcaptionbox}
	\hfill
	\begin{subcaptionbox}{Pressure Field\label{fig:pressure}}[0.48\textwidth]
		{\includegraphics[width=\linewidth]{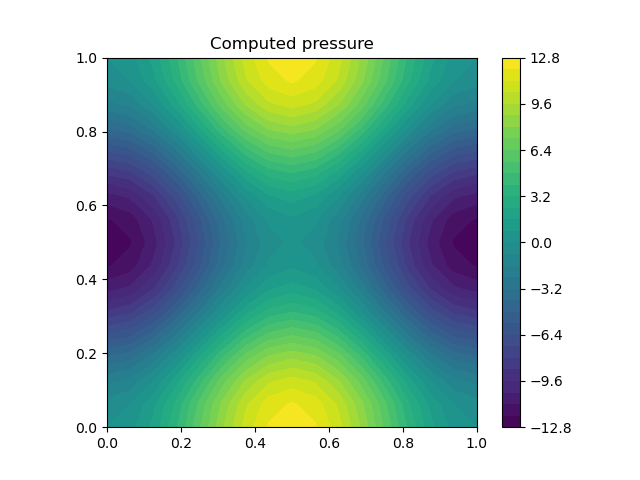}}
	\end{subcaptionbox}
	\caption{Velocity and pressure plots for the computed solution.}
	\label{fig: ex2-velocity_pressure}
\end{figure}

\begin{figure}[h!]
	\centering
	\begin{subcaptionbox}{Velocity field\label{fig:u slip}}[0.48\textwidth]
		{\includegraphics[width=\linewidth]{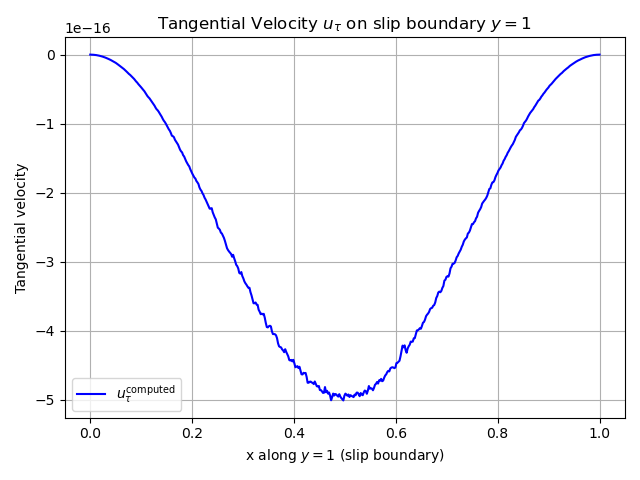}}
	\end{subcaptionbox}
	\hfill
	\begin{subcaptionbox}{Pressure Field\label{fig:sigma_slip}}[0.48\textwidth]
		{\includegraphics[width=\linewidth]{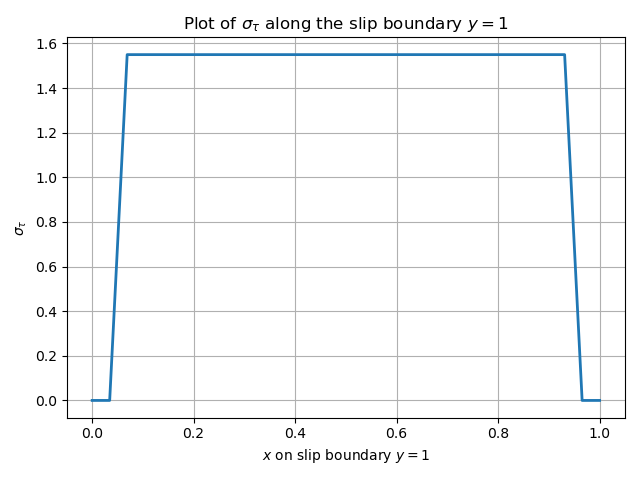}}
	\end{subcaptionbox}
	\caption{Velocity and pressure plots for the computed solution.}
	\label{fig:ex2-Slip}
\end{figure}

\begin{table}[h!]
	\centering
	\begin{tabular}{c c c c c c c}
		\toprule
		Grid size & $\|\bu^h - \bu^*\|_{\L^2}$ & Order($h$) & $\|\bu^h - \bu^*\|_{\V}$ & Order($h$) & $\|p^h - p^*\|_{L^2}$ & Order($h$) \\
		\midrule
		$5 \times 5$   & 1.209e-01 & -- & 4.524e+00 & -- & 2.796e-01 & -- \\
		$10 \times 10$   & 3.357e-02 & 1.85 & 2.512e+00 & 0.85 & 9.297e-02 & 1.59 \\
		$15 \times 15$  & 1.521e-02 & 1.95 & 1.711e+00 & 0.95 & 4.426e-02 & 1.83 \\
		$20 \times 20$  & 8.604e-03 & 1.98 & 1.292e+00 & 0.98 & 2.553e-02 & 1.91 \\
		$25 \times 25$  & 5.520e-03 & 1.99 & 1.033e+00 & 1.01 & 2.553e-02 & 1.96 \\
		$30 \times 30$ & 3.824e-03 & 2.01 & 8.662e-01 & 0.96 & 1.154e-02 & 1.96 \\
		$35 \times 35$ & 2.811e-03 & 2.00 & 7.363e-01 & 1.05 & 8.472e-03 & 2.01 \\
		$40 \times 40$ & 2.139e-03 & 2.04 & 6.509e-01 & 0.92 & 6.525e-03 & 1.96 \\
		$45 \times 45$ & 1.684e-03 & 2.03 & 5.791e-01 & 0.99 & 5.160e-03 & 1.99 \\
		$50 \times 50$ & 1.365e-03 & 1.99 & 5.108e-01 & 1.19 & 4.139e-03 & 2.09 \\
		\bottomrule
	\end{tabular} 
	\caption{Numerical convergence orders} \label{tab:ex2-roc}
\end{table}

\begin{figure}[h!]
	\centering
	\includegraphics[width=\textwidth]{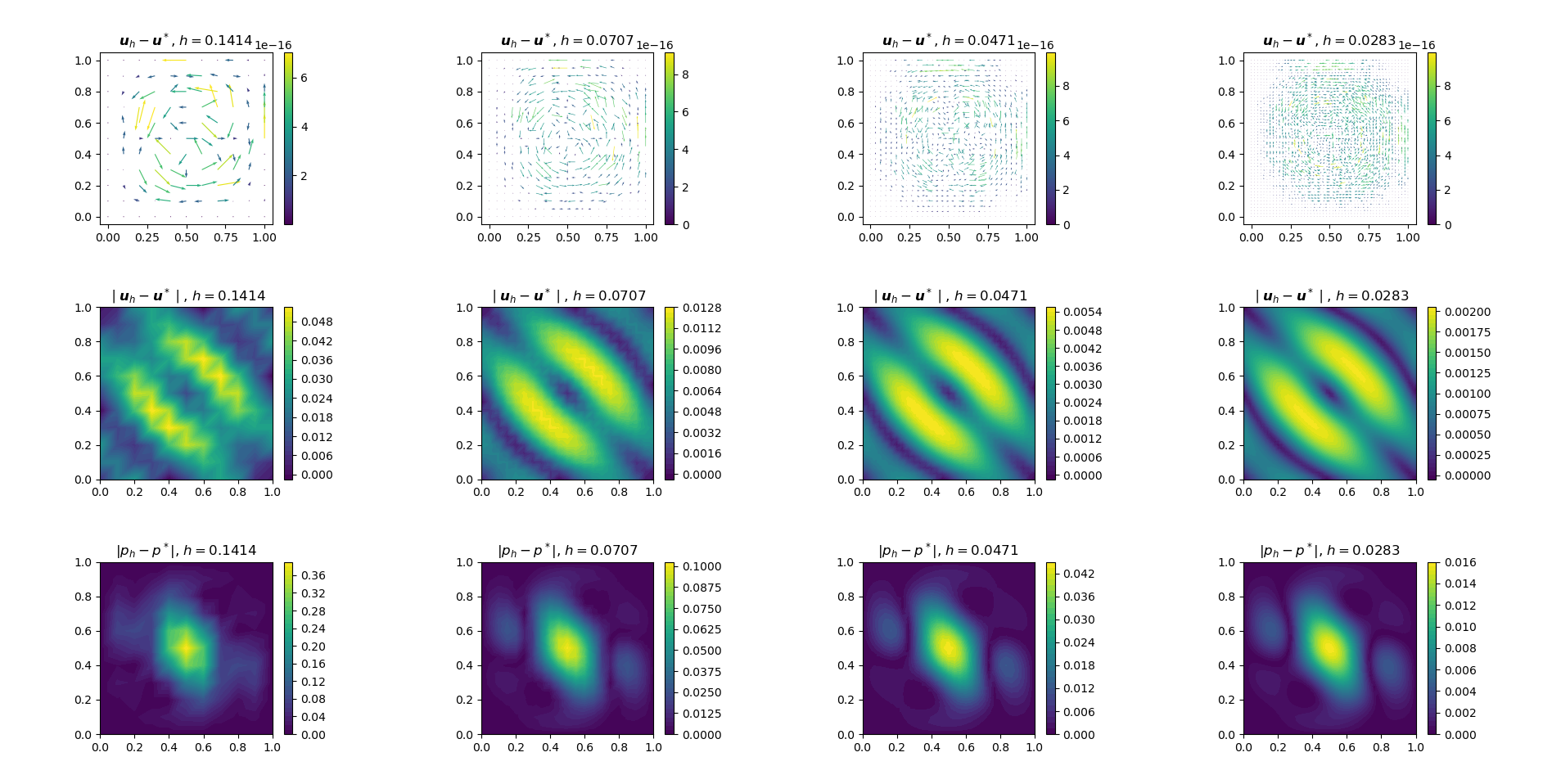}  
	\caption{Errors $\boldsymbol{u}_h - \boldsymbol{u}^*$, $|\boldsymbol{u}_h - \boldsymbol{u}^*|$, and $|p_h - p^*|$ for mesh with grid sizes $10\times 10,$ $20\times 20,$ $30\times 30,$ and $50\times 50,$}
	\label{fig:ex2-error-plot}
\end{figure}

\noindent \textbf{Example 3.} For this example, we take the force function 
$\f$ to be computed using 
\begin{align*}
	\f:=-\mu\Delta\bu^0 + (\bu^0 \cdot \nabla)\bu^0 + \alpha\bu^0 + \beta|\bu^0|^{r-1}\bu^0 + \kappa|\bu^0|^{q-1}\bu^0 + \nabla p^0,
\end{align*}
where $\bu^0=(u_1^0, u_2^0)$ and $p^0$ are chosen as
\begin{align*}
	u_1^0(x, y) &= -x^2 (1 - x) \, y \, (3y - 2), \\
	u_2^0(x, y) &= x (3x - 2) \, y^2 (1-y), \\
	p^0(x, y)   &= (2x - 1)(2y - 1).
\end{align*}
The results for this example are shown in Figures \ref{fig:ex3-velocity_pressure}–\ref{fig:ex3-Slip} and summarized in Table \ref{tab:ex3-roc}. The errors in Table \ref{tab:ex3-roc} are computed with respect to the reference solution $(\bu^*, p^*)$, which is obtained on a mesh with grid size $350 \times 350$.

\begin{figure}[h!]
	\centering
	\begin{subcaptionbox}{Velocity field\label{fig:3velocity}}[0.48\textwidth]
		{\includegraphics[width=\linewidth]{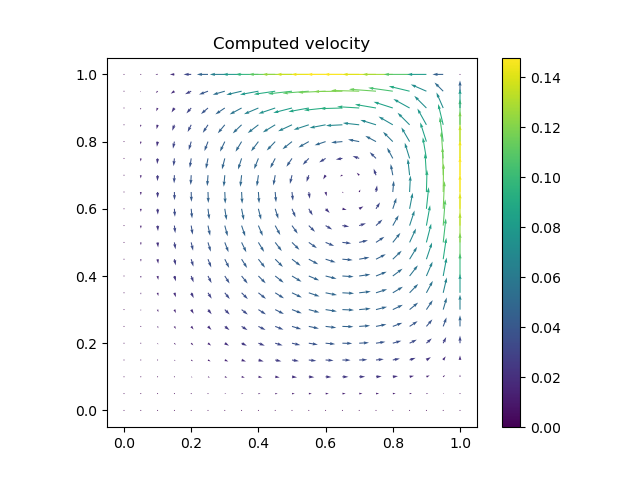}}
	\end{subcaptionbox}
	\hfill
	\begin{subcaptionbox}{Pressure Field\label{fig:3pressure}}[0.48\textwidth]
		{\includegraphics[width=\linewidth]{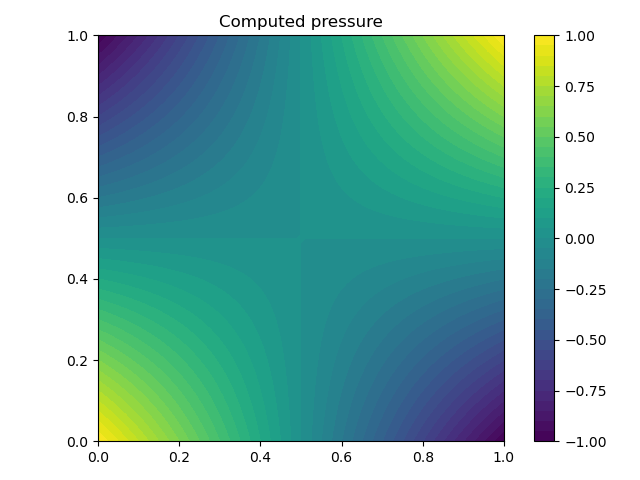}}
	\end{subcaptionbox}
	\caption{Velocity and pressure plots for the computed solution.}
	\label{fig:ex3-velocity_pressure}
\end{figure}

\begin{figure}[h!]
	\centering
	\begin{subcaptionbox}{Velocity field\label{fig:3u slip}}[0.48\textwidth]
		{\includegraphics[width=\linewidth]{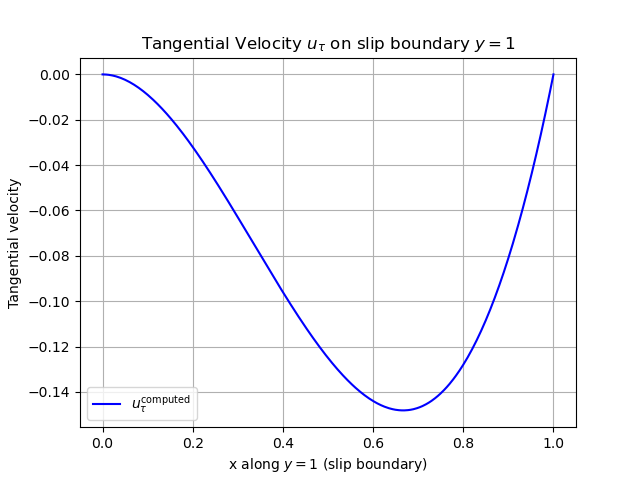}}
	\end{subcaptionbox}
	\hfill
	\begin{subcaptionbox}{Pressure Field\label{fig:3sigma_slip}}[0.48\textwidth]
		{\includegraphics[width=\linewidth]{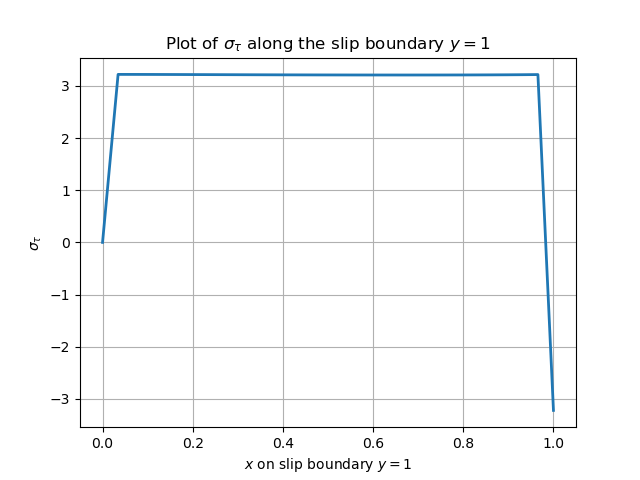}}
	\end{subcaptionbox}
	\caption{Velocity and pressure plots for the computed solution.}
	\label{fig:ex3-Slip}
\end{figure}

\begin{table}[h!]
	\centering
	\begin{tabular}{c c c c c c c}
		\toprule
		grid size & $\|\bu^h - \bu^*\|_{\L^2}$ & Order($h$) & $\|\bu^h - \bu^*\|_{\V}$ & Order($h$) & $\|p^h - p^*\|_{L^2}$ & Order($h$) \\
		\midrule
		$8\times 8$   & 1.529e-03 & -- & 9.159e-02 & -- & 5.335e-03 & -- \\
		$16\times 16$   & 3.877e-04 & 1.98 & 4.653e-02 & 0.98 & 1.336e-03 & 2.00 \\
		$24\times 24$  & 1.723e-04 & 2.00 & 3.112e-02 & 0.99 & 5.940e-04 & 2.00 \\
		$32\times 32$  & 9.658e-05 & 2.01 & 2.334e-02 & 1.00 & 3.342e-04 & 2.00 \\
		$40\times 40$  & 3.342e-04 & 2.02 & 1.869e-02 & 1.00 & 2.139e-04 & 2.00 \\
		$48\times 48$ & 4.239e-05 & 2.04 & 1.559e-02 & 1.00 & 1.485e-04 & 2.00 \\
		$56 \times 56$ & 3.090e-05 & 2.05 & 1.335e-02 & 1.01 & 1.092e-04 & 2.00 \\
		$64 \times 64$ & 2.342e-05 & 2.07 & 1.169e-02 & 0.99 & 8.356e-05 & 1.96 \\
		$72 \times 72$ & 1.830e-05 & 2.00 & 1.040e-02 & 1.00 & 6.605e-05 & 2.00 \\
		$80 \times 80$ & 1.464e-05 & 2.12 & 9.357e-03 & 1.00 & 5.350e-05 & 2.00 \\
		\bottomrule
	\end{tabular} 
	\caption{Numerical convergence orders} \label{tab:ex3-roc}
\end{table}

\begin{figure}[h!]
	\centering
	\includegraphics[width=\textwidth]{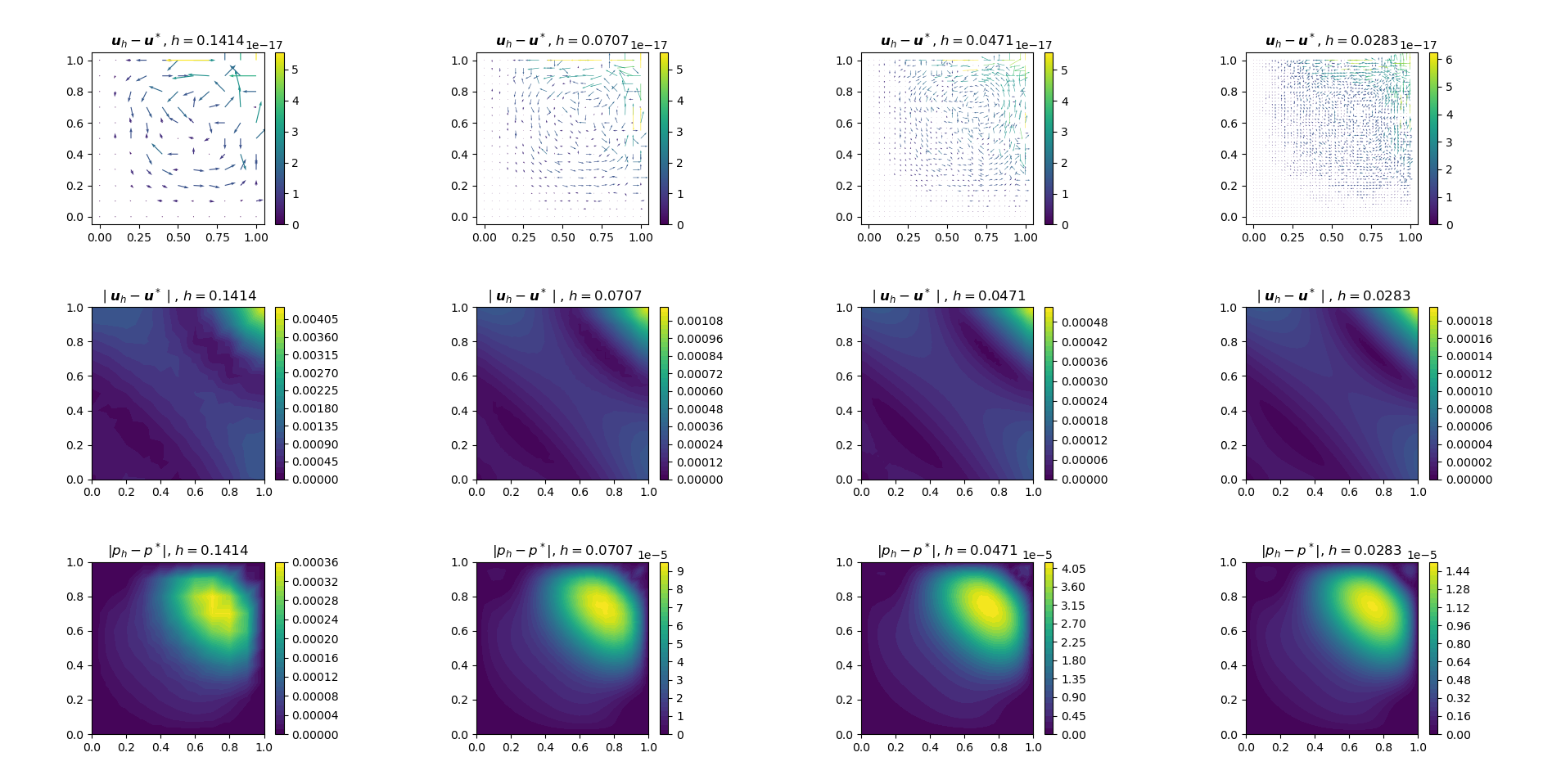}  
	\caption{Errors $\boldsymbol{u}_h - \boldsymbol{u}^*$, $|\boldsymbol{u}_h - \boldsymbol{u}^*|$, and $|p_h - p^*|$ for mesh with grid sizes $10\times 10,$ $20\times 20,$ $30\times 30,$ and $50\times 50,$}
	\label{fig:ex3-error-plot}
\end{figure}
	
\medskip
\noindent
\textbf{Acknowledgments:}  Support for M. T. Mohan's research received from the National Board of Higher Mathematics (NBHM), Department of Atomic Energy, Government of India (Project No. 02011/13/2025/NBHM(R.P)/R\&D II/1137). Wasim Akram is supported by NBHM (National Board of Higher Mathematics, Department of Atomic Energy) postdoctoral fellowship, No. 0204/16(1)(2)/2024/R\&D-II/10823.

\end{document}